\documentclass[oneside, 12pt]{amsart}

\usepackage{amsmath, amssymb, amsthm, hyperref, graphicx,   verbatim, enumerate, xcolor, tikz-cd, stmaryrd, epsfig, xypic}
\usepackage{comment}
\usepackage{array, colortbl}
\usepackage{mathrsfs}  
\usepackage{mathabx}
\usepackage{times}
 
 \DeclareFontFamily{U}{mathc}{}
\DeclareFontShape{U}{mathc}{m}{it}%
{<->s*[1.03] mathc10}{}

\DeclareMathAlphabet{\mathscr}{U}{mathc}{m}{it}

\usepackage[paper=a4paper, marginpar=2.4cm]{geometry}
\usepackage[all]{xy}
\usepackage[utf8]{inputenc}

\usepackage[english]{babel}
 
 \newcommand{\disk}{\mathbb{D}}

\newcommand{\fr}{\partial}

\newcommand{\set}[1]{\left\{#1\right\}}
\newcommand{\norm}[1]{{\left\Vert#1\right\Vert}}
\newcommand{\abs}[1]{\left\vert#1\right\vert}

\newcommand{\rest}[1]{ \arrowvert_{#1}}

\newcommand{\unsur}[1]{\frac{1}{#1}}
\newcommand{\el}{\mathcal{L}}

\newcommand{\rond}{\!\circ\!}
\newcommand{\cst}{\mathrm{C}^\mathrm{st}}
\newcommand{\lrpar}[1]{\left(#1\right)}

\newcommand{\la}{\lambda}
\newcommand{\lo}{{\lambda_0}}

\newcommand{\inv}{^{-1}}

\newcommand{\jstar}{J^\varstar}

\DeclareMathOperator{\supp}{Supp}

\DeclareMathOperator{\tr}{tr}

\DeclareMathOperator{\Crit}{Crit}

\DeclareMathOperator{\jac}{Jac}

\DeclareMathOperator{\dist}{dist}

\DeclareMathOperator{\Per}{Per} 
\DeclareMathOperator{\SPer}{SPer} 
\DeclareMathOperator{\Fix}{Fix}

\DeclareMathOperator{\Trace}{Trace}
\DeclareMathOperator{\Mult}{Mult}
\DeclareMathOperator{\SMult}{SMult}
\DeclareMathOperator{\sm}{sm}

\newcommand{\C}{\mathbf{C}}
\newcommand{\R}{\mathbf{R}}

\newcommand{\Z}{\mathbf{Z}}
\newcommand{\N}{\mathbf{N}}

\newcommand{\bfK}{{\mathbf{K}}}

\newcommand{\bb}{\mathbb{B}}

\newcommand{\A}{\mathbb{A}}

\newcommand{\Aut}{\mathsf{Aut}}

\newcommand{\Sym}{{\sf{Sym}}}

\theoremstyle{plain}

\newtheorem{thm}{Theorem}[section]
\newtheorem{cor}[thm]{Corollary}
\newtheorem{pro}[thm]{Proposition}
\newtheorem{lem}[thm]{Lemma}
\newtheorem{propdef}[thm]{Proposition-Definition}

\newtheorem{mthm}{Theorem}

\newtheorem{que}[thm]{Question}

\theoremstyle{definition}
\newtheorem{defi}[thm]{Definition}

\newtheorem{eg}[thm]{Example}
\newtheorem{rem}[thm]{Remark}

\newtheorem{convention}[thm]{Convention}

\numberwithin{equation}{section}       

\addtocounter{section}{0}             
\numberwithin{equation}{section}       

\title{Multiplier rigidity for complex Hénon maps}

\date{\today}

\author{Serge Cantat}
\address{Serge Cantat, IRMAR, Campus de Beaulieu,
b\^atiments 22-23
263 avenue du G\'en\'eral Leclerc, CS 74205
35042  RENNES C\'edex, France}
\email{serge.cantat@univ-rennes1.fr}
 \author{Romain Dujardin}
\address{Romain Dujardin,  Sorbonne Universit\'e, CNRS, Laboratoire de Probabilit\'es, Statistique  et Mod\'elisation  (LPSM), F-75005 Paris, France}
\email{romain.dujardin@sorbonne-universite.fr}
\thanks{{\footnotesize{The research activities of the authors  are partially funded by the European Research Council (ERC GOAT 101053021) and the  French National Research Agency under the project DynAtrois (ANR-24-CE40-1163). 
}}}

\includecomment{vlongue}
\excludecomment{vcourte}

\begin{document}

\setlength{\parskip}{.2em}
\setlength{\baselineskip}{1.26em}  


\begin{abstract}
We investigate the multiplier rigidity problem for polynomial automorphisms of $\C^2$.  
A first result states   that a 
complex Hénon map of given degree  is determined up to finitely many choices by its multiplier 
spectrum, or more generally by the unstable multipliers of its saddle periodic points. This is the counterpart in this setting of a classical result of McMullen for one-dimensional rational maps.
For compositions of 
Hénon maps, the same rigidity holds provided the multi-degree and the multi-Jacobian are fixed. 
As in McMullen’s theorem, this follows from the nonexistence of stable algebraic families in the corresponding parameter space. This  in turn relies on precise asymptotic bounds for the Lyapunov exponents of the maximal entropy measure along diverging families.
\end{abstract}

\maketitle

\setcounter{tocdepth}{1}
\tableofcontents 

\section{Introduction}

\subsection{Overview of the problem}
The multiplier rigidity problem 
asks whether a dynamical system in a certain class  is determined by the 
multipliers (or eigenvalues) of its periodic orbits. It is part of a more general landscape 
 of spectral rigidity problems, which has many different  incarnations in geometry and dynamics (see 
Wilkinson's  recent monograph~\cite{wilkinson:book} for an overview). 

In the dynamics of one-dimensional rational functions, multipliers are often viewed as algebraic
  functions on moduli spaces and used to describe their geometry; 
  (local) 
  multiplier rigidity means that these functions define (local) coordinates.
     This point of view was 
     popularized by Milnor~\cite{milnor:quadratic, milnor:two_critical} 
 and Silverman~\cite{silverman:book} (the term ``multiplier spectrum'' apparently first appeared in 
 \cite{milnor:lattes}). 
 In this setting, parameter spaces are   finite-dimensional. Therefore, 
 compared to classical spectral rigidity problems  in dynamics, the results and 
 methods have a more algebraic flavor. The  local multiplier rigidity problem
 was settled in an early paper of McMullen~\cite{mcm:algorithms}.
  Still, despite   recent breakthroughs by Ji and Xie~\cite{ji-xie:multipliers, ji-xie:injective} 
  and Huguin~\cite{huguin:polynomials}, the complete characterization of pairs of isospectral 
 rational maps remains open. 

 In this paper, we address similar questions for polynomial automorphisms of $\C^2$. 
 In this context, the multiplier rigidity problem and some of its variants
were already considered in the foundational paper of  Friedland and Milnor~\cite{friedland-milnor};  
and results on the global analytic rigidity of polynomial automorphisms 
 were recently obtained in~\cite{bera-verma, conjugate, rigidity}.

\subsection{Classification of polynomial automorphisms}\label{subs:fm}
In order to state our results, we first need to recall a few basic facts. 
We denote by $\Aut(\C^2)$ the group of polynomial automorphisms of $\C^2$. 
The dynamical degree of $f\in \Aut(\C^2)$ is 
\begin{equation}
\lambda_1(f) :=\lim_{n\to \infty}\deg(f^n)^{1/n}\in[1, \infty).
\end{equation}
This number turns out to be a positive integer   and is invariant under conjugacy.
If $\lambda_1(f)=1$, $f$ is said \emph{elementary} and its dynamics is easy to analyze. 
When $\lambda_1(f)>1$,  we say that $f$ is \emph{loxodromic}. A  
 {\emph{Hénon map}} $h\colon \C^2\to \C^2$ is a polynomial automorphism of the form
 \begin{equation}
 h(x,y) = (a y+p(x), x)
 \end{equation}
 for some $a\in \C^*$ and some $p\in \C[x]$ of degree $\geq 2$. Hénon maps are the primary examples of loxodromic automorphisms: $\lambda_1(h)=\deg(p)$.

According to~\cite{friedland-milnor}, if $f$ is loxodromic there exists  a sequence of integers $(d_k, \ldots , d_1)\in \N_{\geq 2}^k$, a sequence of non-zero complex numbers 
$(a_k, \ldots , a_1)\in (\C^*)^k$ and a sequence of monic and centered polynomials 
$p=(p_k, \ldots , p_1)\in \C[x]^k$ of respective degrees $d_i$ such that $f$ is conjugate to the composition   
$g = h_k\circ \cdots \circ h_1$,
 where 
 $h_i(x,y) = (a_iy+p_i(x), x)$. We refer to the decomposition $g=h_k\circ \cdots \circ h_1$ as a \emph{Friedland-Milnor normal form} of~$f$. 
The Jacobians of $f$ and  $g$ are both equal to $\jac(g) = (-1)^k \prod_{i=1}^k a_i$.  
The degree of $g$ is   $d=\prod_{i=1}^k d_i$, and $\deg(g^n)=\deg(g)^n$, hence $\lambda_1(f)=\lambda_1(g)=d$. 
In particular, if  $\lambda_1(f)$ is a prime number then $k=1$ and  $f$ 
is conjugate to a monic and centered Hénon map.
By definition, 
$\mathbf d = (d_k, \ldots , d_1)$ is the \emph{multidegree} of 
$g$ and $\mathbf a=(a_k, \ldots , a_1)$ is its \emph{multi-Jacobian}.  
We let    $\mathcal H^k_{\mathbf d}$ 
  (resp. $\mathcal H^k_{\mathbf d,\mathbf a}$)
   be the space of    compositions of $k$ monic and centered  Hénon mappings of  
  multidegree $\mathbf d\in \N_{\geq 2}^k$ (resp. and   multi-Jacobian $\mathbf a$). 
As an algebraic variety, $\mathcal H^k_{\mathbf d}$ is isomorphic to $(\C^*)^k\times \C^{d_1-1}\times \ldots \C^{d_k-1}$.    For $k=1$ we   
simply write $\mathcal H^1_{\mathbf d} = \mathcal H^1_d$.

The normal form is not unique, but the defect of uniqueness is well understood: it is unique up to permutation of the factors $h_i$ and a diagonal action of the group 
$\mathbb{U}_{d-1}$ of $(d-1)^{\rm th}$ 
roots of unity (see \S~\ref{subs:M_for_loxodromic} or~\cite{friedland-milnor} for details). 
Thus, coming back to $f$,  
the multidegree and the multi-Jacobian define  conjugacy invariants only up to a finite group action. 

\subsection{Multiplier rigidity}\label{par:into_multiplier_rigidity}

If $z$ is a periodic point of (exact) period $n$ of $f\colon \C^2\to \C^2$, 
the \emph{multipliers} $\lambda_1(z)$, $\lambda_2(z)$ of $z$ 
are the (unordered) eigenvalues of $Df^n_z$. 
A convenient way to encode the multipliers is by the 
 \emph{trace} $\tr(Df^n_z)$ and the Jacobian $\jac(f^n)$,
   which satisfy 
\begin{equation}
\lambda_1+\lambda_2 = \tr(Df^n_z) \; \text{ and } \;  \lambda_1\lambda_2  = \jac(f^n).
\end{equation} 
If $z$ is of saddle type, the eigenvalues can be distinguished and are   denoted by
$\lambda^s(z)$ and  $\lambda^u(z)$, with  
$\abs{\lambda^s(z)}<1<\abs{\lambda^u(z)}$.  
The \emph{multiplier spectrum}
(resp.\  the \emph{trace spectrum}) of $f$ is the set of all multipliers (resp.\  traces) of all periodic cycles of $f$, organized by their periods. The {\emph{unstable multiplier spectrum}} 
is the set of unstable eigenvalues $\lambda^u$ of saddle periodic cycles, organized by their periods.
We refer to \S~\ref{subs:spectra} for precise definitions.

\begin{mthm}\label{mthm:henon_multipliers}
A complex Hénon map $f(x,y)=(ay+p(x),x)$ 
is determined
up to finitely many choices by its trace spectrum (resp. its unstable multiplier spectrum).
\end{mthm}

In particular any complex Hénon map $f_0\in \mathcal H^1_d$ 
has the following local multiplier rigidity property: 
any $f$ sufficiently close to $f_0$ with the same trace 
(resp. unstable multiplier) spectrum is equal to $f_0$.  
An interesting    point in the theorem is  that the knowledge  of the Jacobian 
is not necessary for this  spectral characterization. 

Let $S(f_0)$ be the number of maps $f\in \mathcal H^1_d$ with the same multiplier (or trace) spectrum as $f_0$. 
Theorem~\ref{mthm:henon_multipliers} says that $S(f_0)$ is finite. In fact, 
$S$ is uniformly bounded in $\mathcal H^1_d$  and  $S(f_0)=1$ for a generic $f_0$ (see 
Theorem~\ref{thm:noether} and Theorem~\ref{thm:henon_huguin2}). Thus, 
    a generic Hénon map is actually 
determined up to conjugacy 
by its multiplier spectrum.  It is an open problem whether this holds for every loxodromic automorphism
(cf.~\cite[Conjecture 1.7]{ji-xie:injective} for one-dimensional rational maps).

By Noetherianity,
for every $d$ there exists $P=P(d)$ 
such that a complex Hénon map $f$ is determined
up to finitely many choices by the  traces of its  periodic points up to period 
$P$ (see  Theorem~\ref{thm:noether}). Friedland and Milnor proved that for $d=2$ or 3, the knowledge of the 
traces of the  \emph{fixed} points together with the Jacobian determines the 
conjugacy class of $f$ (see~\cite[\S 7]{friedland-milnor}). 
 For the analogous problem of the determination of a polynomial $p:\C\to \C$ by its periodic point multipliers, Huguin showed in~\cite{huguin:polynomials} that $P = 2$ is sufficient. 
 We will give in Section~\ref{sec:huguin} a quick proof of Theorem~\ref{mthm:henon_multipliers} when the 
Jacobian is different from $-1$ (i.e. $a\neq 1$): it is 
  based on Huguin's theorem, using  the   specific formula defining the Hénon map $f$, and it shows that $P=2$ is enough for these maps as well.  
On the other hand the family $(f_\lambda)_{\lambda\in \C}$ defined by 
\begin{equation}\label{eq:jac-1}
f_\lambda (x,y) = (y+ (x^2-\lambda^2)^2, x) \end{equation} 
is a non-trivial family of Hénon maps of Jacobian $-1$
whose periodic data is constant in period 1 and 2 (see Example~\ref{eg:jacobian-1}). 
This shows that Huguin's method cannot be applied  to this setting. 
We will present a more conceptual 
(and  complicated) proof of Theorem~\ref{mthm:henon_multipliers}, which 
works in  Jacobian   $-1$ and  
yields the following result.

\begin{mthm}\label{mthm:multi-Jacobian}
The conjugacy class of a loxodromic automorphism is determined, up to finitely many possibilities, by: (a) its multidegree, (b) its multi-jacobian, and (c)  
its trace spectrum (resp. its unstable multiplier spectrum). 
\end{mthm}

Note that in Theorem~\ref{mthm:multi-Jacobian} the multi-Jacobian is fixed. 
We can also address the $C^1$ rigidity problem
 suggested by Friedland-Milnor in~\cite[Thm 7.5]{friedland-milnor}, as well as the global
  real-analytic rigidity result of~\cite[Thm 7.1]{rigidity}. 

\begin{mthm}\label{mthm:C1}
The $C^1$ conjugacy class of any complex Hénon map $f_0$ of given degree  (resp. any composition $f_0$ of Hénon maps 
with given multidegree and multi-Jacobian) is finite.  
In particular in $\mathcal H = \mathcal H^1_d$ or $\mathcal H^k_{\mathbf{d}, \mathbf{a}}$ the following rigidity property holds:  for $f_0\in \mathcal H$,
  if  $f\in \mathcal H$ is $C^1$ conjugate to $f_0$ in a neighborhood of its Julia set 
and is sufficiently close to $f_0$, then $f=f_0$.   
\end{mthm}

\subsection{Stable algebraic families}
Following a beautiful idea of 
McMullen~\cite{mcm:algorithms}, 
the results above
will follow from a global structural stability statement.  

\begin{mthm}\label{mthm:stable}
Any stable irreducible algebraic family of complex Hénon maps is trivial. Likewise, any stable 
irreducible  algebraic family of compositions of Hénon maps with fixed multi-Jacobian is trivial. 
\end{mthm}

 By an \emph{algebraic  family of (compositions of) Hénon maps}, we mean an algebraic variety together with an algebraic map $\Lambda\to \mathcal H^k_{\mathbf d}$  for some $k\geq 1$ and $\mathbf d\in \N^k_{\geq 2}$.
 The family is \emph{trivial} if its image  is reduced to a point. Most of the time, a family can be identified with its image in $\mathcal H^k_{\mathbf d}$. 
To derive Theorems~\ref{mthm:henon_multipliers}, \ref{mthm:multi-Jacobian} and~\ref{mthm:C1} from 
Theorem~\ref{mthm:stable}, we  argue as in~\cite{mcm:algorithms}: 
given $f_0$ in $\mathcal H  = \mathcal H^1_d$ or $\mathcal H^k_{\mathbf d, \mathbf a}$, 
  we consider the set $\Lambda(f_0)\subset \mathcal H$ 
  of mappings with the same trace  (resp.\  unstable multiplier) 
  spectrum and show that it is a stable algebraic family. For the trace spectrum, $\Lambda(f_0)$ is obviously algebraic and 
  the difficulty is to prove the stability, because we do not specify    the Jacobian. 
 Conversely, for the unstable multiplier spectrum, stability is not an issue, what is delicate 
  is that $\Lambda(f_0)$ is only a priori a countable 
 intersection of semi-algebraic subsets. 
  The details are in Proposition~\ref{pro:multipliers}.
Actually for the   unstable multipliers  there is an additional finiteness issue that  
 can only be resolved   after Theorem~\ref{mthm:lyapunov} is established (see \S~\ref{subs:comments_corollaries}). 

Note that the word   `stability' has not been defined yet.
In  first approximation, this refers  to the   notion of 
 weak stability defined by  Lyubich and the second author in~\cite{tangencies}, 
however  a subtlety arises whether we authorize a small set of periodic points to bifurcate or not. This issue is 
  analyzed in \S~\ref{subs:stability}.

\subsection{Divergence  of Lyapunov exponents}\label{par:intro_divergence_of_Lyap}
The starting point of the proof of Theorem~\ref{mthm:stable}  is that a non-trivial stable algebraic family $\Lambda$
is  unbounded in parameter space, so it should exhibit some degeneration phenomena. 
  We analyze  these degenerations by  taking  inspiration from the 
theory of polynomials in one variable. In that case, there is 
a simple  approach to Theorem~\ref{mthm:stable} based on the Lyapunov
exponent of the equilibrium measure. 

  Indeed, let $\mathcal P_d = \mathcal P_d(\C)$ be the space of affine conjugacy classes of polynomials of degree $d$ in one complex variable. 
Branner and Hubbard~\cite{branner-hubbard1} 
have first studied the asymptotic properties of the dynamics when $p$ tends to infinity in 
$\mathcal P_d$ (other relevant sources for us include~\cite{DeMarco-McMullen, huguin:polynomials}). 
Define the \emph{maximal escape rate} of a polynomial $p$ by 
\begin{equation}
\label{eq:maximal_escape_rate_intro}
M(p)= \max\set{G_p(c)\; ; \; \ p'(c)=0},
\end{equation}
where $G_p(z) = \lim_{n\to\infty} d^{-n} \log^+ \abs{p^n(z)}$ is the dynamical Green function.
The quantity $M(p)$ is invariant under conjugacy so it defines a function on $\mathcal P_d$. 
It is shown in~\cite{branner-hubbard1} that $M$ defines a proper map $\mathcal P_d  \to \R_+$  
and provides a good measurement of the amount of degeneracy of $p$.
 Most importantly for us,  the Manning-Przytycki formula  \cite{manning, przytycki} 
  \begin{equation}\label{eq:manning-przytycki}
 \chi(\mu_p)= \log d+  \sum_{ p'(c)=0} G_p(c)
\end{equation}
 for the Lyapunov exponent 
 $\chi(\mu_p)$ 
 of the equilibrium measure $\mu_p=\Delta G_p$ 
     implies that  
\begin{equation}\label{eq:lyap_1D}
\log d+ M(p)\leq \chi(\mu_p)\leq \log d+ (d-1)M(p).
\end{equation}

In a stable family the multipliers of repelling points move holomorphically and take values in $\C\setminus \overline \disk$, thus if  the family is irreducible  and algebraic they must remain constant. On the other hand, by the equidistribution theory of periodic points, 
  for large $n$  the Lyapunov exponent $\unsur{n}\log\abs{\lambda^u}$ of a typical point of period $n$ 
is close to $\chi(\mu_p)$. These two facts and the lower estimate  
in~\eqref{eq:lyap_1D} imply that $M(p)$ is bounded on the family. Since 
$M$ is proper on $\mathcal P_d$, we conclude that a stable algebraic family must be 
bounded in the affine variety $\mathcal P_d$, hence a finite set.

\begin{rem}
Huguin's approach in~\cite{huguin:polynomials} is based on an
  estimate similar to the lower bound in~\eqref{eq:lyap_1D} 
for  the Lyapunov exponents of periodic points of period $1$ and $2$  
(see~\cite[Cor A.1]{huguin:polynomials}). This estimate from one variable dynamics  
 plays an important  role in Section~\ref{sec:huguin}, where we treat Hénon maps with  $a\neq1$.   
The example in~\eqref{eq:jac-1} shows that the corresponding  
result is not true in $\Aut(\C^2)$. 
\end{rem}

  In Section~\ref{sec:lyapunov} we design a two-dimensional version of this argument.
If $f = h_k\circ \cdots \circ h_1$ is a composition of monic and centered Hénon maps as in 
 \S~\ref{subs:fm} (with possibly $k=1$), 
 we define
 \begin{equation}
 \label{eq:escape_rate_for_automorphisms_intro}
 M(f) = \max (M(p_i), i = 1, \ldots , k).
 \end{equation}
 (Because of the $\mathbb U_{d-1}$ action alluded to in \S~\ref{subs:fm}, $M(f)$ is not invariant under  conjugacy, we refer to  \S~\ref{subs:M_for_loxodromic} for a discussion and a remedy.)
According to~\cite{bs3}, $f$ admits a canonical equilibrium measure $\mu_f$, whose Lyapunov exponents 
satisfy 
\begin{equation}
\chi^-(\mu_f)\leq -\log d <0<\log d\leq \chi^+(\mu_f). 
\end{equation}
  In~\cite{bs5}, Bedford and Smillie provide an exact expression for $\chi^\pm(\mu_f)$ which is 
  analogous to  the Manning-Przytycki Formula~\eqref{eq:manning-przytycki} and  involves the delicate
  notion of unstable critical point (see \S~\ref{subsub:conclusion_proof_henon} for a presentation). 
  We use it to obtain the following theorem.

\begin{mthm}\label{mthm:lyapunov}
 Let  $\mathbf d\in \N_{\geq 2}^k$ be a multidegree and $\mathbf a\in (\C^*)^k$ be a multi-Jacobian, for some $k\geq 1$.
There exists  positive constants $C_1$ and $C_2$
 depending only on $(\mathbf d, \mathbf a)$
  such that  
$$ C_1M(f)\leq  \chi^+(\mu_f)\leq  C_2 M(f), $$
for all  $f$    in $ \mathcal H^k_{\mathbf d,\mathbf a}$ with $M(f)\geq 1$.
\end{mthm}

From the relation $\chi^+(\mu_f) + \chi^-(\mu_f) = \log\abs{\jac(f)}$, 
  a similar result  holds for  $\chi^-(\mu_f)$. 
Our estimate is actually  much more precise, and      allows the 
multi-Jacobian to diverge at a controlled rate with respect to $M(f)$: see Propositions~\ref{pro:lyapunov_henon}  and~\ref{pro:lyapunov_multi}. 
The proof is perturbative in nature  (even if we allow for pretty large perturbations), 
and ultimately relies on the one-dimensional estimates of~\cite{branner-hubbard1}, together with  
a geometric analysis of unstable critical points in the spirit of \cite{lyapunov}.

To deduce Theorem~\ref{mthm:stable} from Theorem~\ref{mthm:lyapunov},  consider a stable 
  algebraic family $\Lambda$. As for polynomial maps, we observe that 
the multipliers of saddle periodic points must remain constant along each component of  $\Lambda$. 
In particular the Jacobian is constant as well (this explains why there is no assumption on the Jacobian in the Hénon case). 
Then, Theorem~\ref{mthm:lyapunov} implies that $M(f)$ must be bounded along the family  
$\Lambda$, which 
implies that $\Lambda$ is a finite set.

Thus, to drop the assumption on the multi-Jacobian in Theorem~\ref{mthm:stable}, we should prove a version of Theorem~\ref{mthm:lyapunov} in which the `constant multi-Jacobian' assumption is replaced by `constant Jacobian'. 
It is unclear to us whether  such a statement holds. 
What is clear is that our method breaks down 
without such an assumption (see Example~\ref{eg:fafinva}). 
A related issue, which  appears in~\cite{branner-hubbard1}, is  
the compactness of the connectedness locus. 
In $\mathcal P_d$ this follows from the   fact that connectedness of the Julia set is equivalent to $M(p)=0$. 
Likewise,
Theorem~\ref{mthm:lyapunov} together with the results of~\cite{bs6} and~\cite{connex} imply that 
the connectedness locus is compact in $\mathcal H^k_{\mathbf d,\mathbf a}$ for fixed $\mathbf a$
(see Corollary~\ref{cor:connectedness}).  
On the other hand this result is not  true for variable Jacobian  
\begin{vlongue}
(see \S~\ref{subs:eg_fa})
\end{vlongue}
: this outlines a difference between the one-dimensional and the two-dimensional situations.

The asymptotic behavior of the Lyapunov exponent
along algebraic families (or along degenerating sequences) 
is an important topic in the dynamics of rational 
functions (see e.g., \cite{favre:2020,  luo:2022, favre-rivera:2025, favre:2025, favre-gong} and the references therein)
and  groups of Möbius transformations
(see e.g., \cite{morgan-shalen, degenerations_SL2}). This asymptotic behavior can be studied with techniques from non-Archimedean degenerations: such techniques are not yet available for automorphisms of $\C^2$ 
and would constitute a natural continuation of this work. 

\subsection{Logical structure} 
The following diagrams will be helpful to understand the logical relationships between the main results.  

\begin{small}
 \begin{tikzcd}[
    column sep=1.3em, 
    row sep=0.4em,
    nodes={anchor=west},
    arrows={double, Rightarrow}
]
 & &  \text{Thm \ref{thm:traces}} \arrow[r] & \text{Thms \ref{mthm:henon_multipliers} and \ref{mthm:multi-Jacobian} (trace spectrum)} \\
 \text{Prop \ref{pro:multipliers} + Thm \ref{mthm:stable}} 
  \arrow[rru] 
  \arrow[rr] 
  \arrow[rrdd] 
 & &  \text{Thm \ref{thm:C1}} \arrow[r]   & \text{Thm \ref{mthm:C1} (local rigidity)} \\
\\
&   &  \text{Thm \ref{thm:u_multipliers}} \arrow[r] & \text{Thms \ref{mthm:henon_multipliers} and \ref{mthm:multi-Jacobian} (unstable multipliers)} \\
& \text{Thm \ref{mthm:lyapunov}}
\arrow[ru, end anchor=south west]\arrow[r] & \text{Prop \ref{pro:C1_finite}} \arrow[r] & \text{Thm \ref{mthm:C1} (global finiteness)}
\end{tikzcd} 

\smallskip

\begin{tikzcd}[column sep=2em, arrows={double, Rightarrow}]
 \text{Prop \ref{pro:lyapunov_henon}}\arrow[r, dashed] &  \text{Prop \ref{pro:lyapunov_multi}}
\arrow[r] & \text{Thm \ref{mthm:lyapunov}} \arrow[r] & \text{Thm \ref{mthm:stable}}
\end{tikzcd}
\end{small}

\section{Stability theory of loxodromic automorphisms}\label{sec:prel_stab}

\subsection{Dynamics of loxodromic automorphisms}
\label{par:basics_on_loxodromic_automorphisms}
We use the standard dynamical objects associated to a loxodromic automorphism $f\colon \C^2\to \C^2$, 
as summarized  in \cite[\S 1]{bs6} or \cite[\S 2.1]{tangencies},  with the same notation. 
In particular we will consider the Green functions $G^\pm$, the equilibrium measure $\mu$, and the ``small Julia set'' $\jstar = \supp(\mu)$. Most often we work with families, and the  dependence in $f$  is underlined: we write $G^\pm_f$, $\mu_f$, $K_f$, $\ldots$, or $G^\pm_\lambda$, $\mu_\lambda$, $K_\lambda$, $\ldots$, when $f$ depends on a parameter $\lambda$. 

Let $d\geq 2$ be the dynamical degree of $f$. 
We denote by $\Fix (f)$ the set of fixed points of $f$ and by $\Per_n(f)$ the set of periodic points of exact period $n$. 
So $\Fix(f^n)$ is the sets  $\Per_k(f)$ for $k$ dividing $n$. 
It was  shown in~\cite[Thm 3.1]{friedland-milnor} that all periodic points are isolated and $\#\Fix(f^n) = d^n$, counting multiplicities. 
An exact expression for  $\#\Per_n(f)$ can be obtained by the   Möbius inversion formula. 
Since $\Per_n(f) =  \Fix(f^n)\setminus \bigcup_{k\leq n/2} \Fix(f^k)$,  we get  
\[
\#\Per_n(f)\sim d^n 
\] 
as $n\to\infty$.
We denote by  $\SPer_n(f)$ the set of saddle periodic points of exact period $n$ and by $\SPer(f)$ the set of all saddle periodic points.
We already introduced the notation $\lambda^{s/u}(p)$ in \S~\ref{par:into_multiplier_rigidity} 
for the stable/unstable multipliers of $p\in \SPer_n(f)$. 
We will  denote by $\chi^{s/u}$ the corresponding Lyapunov exponents, that is 
 \[
 \chi^{s/u}(p)   = \unsur{n} \log\abs{\lambda^{s/u}(p)}.
 \]

If $\mathrm{P}$ is a set of periodic points, we write $\mathrm{P}_n$  for $ \mathrm{P} \cap \Per_n(f)$. 
 We   say that   $\mathrm{P}$   is of asymptotic density $1$  if $\#\mathrm{P}_n  \sim d^n$ as $n\to\infty$, 
 and of   positive upper density if $\limsup_{n\to\infty} d^{-n}\#\mathrm{P}_ n  >0$.  
 
 The equidistribution theorem obtained by Bedford, Lyubich and Smillie in~\cite{bls2} 
asserts  that $\SPer_n(f)$ is of asymptotic density $1$  and  
that saddle periodic points are asymptotically equidistributed towards the equilibrium measure $\mu_f$. 
The following reinforcement of this result was obtained in~\cite{rigidity}.

 \begin{thm}\label{thm:equidistribution_reinforced}
Let $f\in \Aut(\C^2)$ be a loxodromic automorphism of dynamical degree $d$. 
There exists a set $\SPer^+(f)\subset \SPer(f)$ 
 of saddle periodic points   of asymptotic density 1 such that 
 if $(p_n)$ is an arbitrary sequence of periodic points with $p_n\in \Per_n(f)$, then
\begin{equation}\label{eq:equidistribution_reinforced}
\lim_{n\to\infty} \unsur{n}\sum_{k=0}^{n-1}  \delta_{f^k(p_n)}  = \mu_f
\quad \text{ and } \quad 
\lim_{n\to\infty}   \chi^u(p_n) = \chi^+(\mu_f).
\end{equation}
\end{thm}
  
As a consequence, if  $\mathrm{P}$ is a set of saddle periodic points with 
 positive upper density, there is a sequence $(p_{n_j})_{j\geq 0}$ with 
$p_{n_{j}}\in \mathrm{P}\cap \Per_{n_j}(f)$ such that the two properties in~\eqref{eq:equidistribution_reinforced}
 hold for $p_{n_j}$ as $j$ goes to $+\infty$. 

\subsection{$\jstar$-stability}\label{subs:stability}
Let $(f_\lambda)_{\lambda\in \Lambda}$ be a holomorphic family of loxodromic automorphisms of $\C^2$ of dynamical degree $d$.
We assume in this section that $\Lambda$ is connected. 
Abusing notation, we often  identify the parameter in $\Lambda$ with the corresponding map $f_\lambda\in \Aut(\C^2)$ and simply write ``$f\in \Lambda$''. 
When a parameter $\lo\in \Lambda$ is used as a base point we often set $f_0 = f_\lo$.   

A notion of stability was developed for such families in~\cite{tangencies}. 
It can be expressed in several ways, which we now describe.
Recall that by the implicit function theorem,  any saddle periodic point  
can be followed holomorphically along $\Lambda$  until one of its multipliers hits the value $1$. 
 Let us consider the following 6 different types of stability for $(f_\lambda)$:
 \begin{enumerate}[(I)]
 \item there is a parameter $\lo$ such that 
 all saddle periodic points of $f_0$  persist as saddles  throughout $\Lambda$;
 \item  there  is $\lo\in \Lambda$ and a set 
 of saddle periodic points of $f_0$ of asymptotic density $1$ that persist 
 as saddles throughout $\Lambda$;
 \item  there  is $\lo\in \Lambda$ and a set 
 of saddle periodic points of  $f_0$ of positive  upper density 
  that persist as saddles throughout $\Lambda$;
 \item there  is $\lo\in \Lambda$ and a set $\mathrm P_0$
 of saddle periodic points of 
  $f_0$   such that $\overline{\mathrm P_0} = \jstar_0$, that  persist as saddles
  throughout $\Lambda$;
\item There is a locally defined branched holomorphic motion  of $\jstar$ over $\Lambda$. (See \cite{tangencies} for the notion of branched holomorphic motion.)
\item For any $\lambda\in \Lambda$, any saddle point for $f_\lambda$ can be followed holomorphically as a periodic point over $\Lambda$ (but not necessarily as a saddle). 
 \end{enumerate}

One comment is in order about the meaning of ``persist'' in these assertions. This persistence refers to 
the analytic continuation of a saddle or periodic point. It means that the given periodic point can be followed
along any path; but along a homotopically non-trivial loop the periodic point can possibly come back to another periodic point. 
The same comment applies to  condition (V): the branched holomorphic motion of $\jstar$ may have non-trivial monodromy over $\Lambda$.

\begin{pro}\label{pro:stability_conditions}
These stability conditions satisfy $\mathrm{(I)}\Rightarrow \mathrm{(II)} 
\Rightarrow \mathrm{(III)}\ \Rightarrow \mathrm{(IV)}\Rightarrow \mathrm{(V)}\Rightarrow \mathrm{(VI)}$.
\end{pro}

Notice  that for    statements  of the form ``any stable algebraic family of type (X) is trivial'', the implications go in the reverse direction.

\begin{rem}
\label{rem:substantiality}
In~\cite{tangencies}, Definition 4.1, condition (V) is called {\emph{weak $\jstar$-stability}} and it is shown that 
if the family is \emph{substantial}, then  $\mathrm{(VI)}\Rightarrow \mathrm{(I)} $ so all these conditions are equivalent (conditions (II) and (III) are actually not considered in that paper). Substantial means that either the family is dissipative, or that 
the following {\emph{non-degeneracy condition}} is satisfied:
for any periodic point, no relation of the form 
   $\lambda_1^a\lambda_2^b = c$ hold persistently in parameter space, where 
  $\lambda_1$, $\lambda_2$ are its eigenvalues, $a$ and $b$ are some fixed real numbers
  \begin{vlongue}
(\footnote{It is inaccurately stated that $a$ and $b$ could be complex in~\cite{tangencies} but the actual issue is for real $a$ and $b$. Indeed, with notation as in~\cite[Lem. 4.11]{tangencies}, the issue is about the $\R$-linear independence of the $\theta_i$. Furthermore, for specific choices of $a$ and $b$, like $a>0$, $b<0$ and $c=1$, the generic linearizability can still be verified, see e.g.~\ref{subsub:alpha} below.})
\end{vlongue}
\begin{vcourte}
(\footnote{It is inaccurately stated that $a$ and $b$ could be complex in~\cite{tangencies} but the actual issue is for real $a$ and $b$. Indeed, with notation as in~\cite[Lem. 4.11]{tangencies}, the issue is about the $\R$-linear independence of the $\theta_i$.})
\end{vcourte}
and  $c$ is some fixed complex number of modulus $1$. 
 If the family is not substantial, one cannot exclude the possibility that, even if a dense set of saddle points moves holomorphically (condition (IV)), some saddle points bifurcate by having both their multipliers crossing the unit circle without breaking the holomorphic motion.  Conjecturally, such a phenomenon never happens and the substantiality assumption is superfluous.   
\end{rem}

\begin{proof}[Proof of Proposition~\ref{pro:stability_conditions}]
The implications $\mathrm{(I)}\Rightarrow \mathrm{(II)} \Rightarrow \mathrm{(III)}$ are obvious. The  equidistribution theorem of~\cite{bls2} shows that 
$\mathrm{(II)}\Rightarrow \mathrm{(IV)}$ and  Theorem~\ref{thm:equidistribution_reinforced}
shows that $\mathrm{(III)}\Rightarrow \mathrm{(IV)}$.

For $\mathrm{(IV)}\Rightarrow \mathrm{(V)}$, we start with the holomorphic motion of $\mathrm P_0$ above $\Lambda$. 
Locally, this gives a normal family of graphs above $\Lambda$ because the periodic points of $f_\lambda$ are contained in $K_{\lambda}$ and, for any compact subset $B$ of  $\Lambda$, there is a compact $K_B\subset \C^2$ containing $K_{\lambda}$ for any $\lambda\in B$. 
Thus, by Lemma~3.4 of~\cite{tangencies}, the motion $(\mathrm P_\lambda)$ of $\mathrm P_0$ extends to a branched holomorphic motion $(j_\lambda)$ of $\jstar_{0}$. 
Since the points of $\mathrm P_\lambda$ are saddles, we get $ j_\lambda(\jstar_{0}) \subset \jstar_\la$. 
To prove the converse inclusion,  we put $S_0=\SPer(f_0)$ and   remark that by~\cite[Lem. 4.10]{tangencies},
$(j_\lambda)$ induces a motion $(S_\lambda)$ of all saddle periodic points,  because $S_0 \subset J^*_0$. 
By holomorphic continuation, periodic points of period $n$ remain periodic of period dividing $n$, hence $S_\lambda\subset \Per(f_\lambda)$ for every $\lambda\in \Lambda$. But along the motion there might be collisions of periodic points, as well as periodic points changing type. 
Nevertheless, Lemma~4.13 from~\cite{tangencies}
asserts that for any fixed $\lambda_1$, 
$\overline {S_{\lambda_1}}$ contains $\jstar_{\lambda_1}$. The reason is that 
by the equidistribution theorem of~\cite{bls2}, the asymptotic 
density of points in $S_0$ which must also be saddles at $\lambda_1$ tends to $1$ as the period tends to infinity. Thus $ j_\lambda(\jstar_{0}) = \jstar_\la$, as claimed.
A similar  argument also shows that  $\mathrm{(V)}\Rightarrow \mathrm{(VI)}$.
\end{proof}

\begin{vlongue}
\begin{rem}\label{rem:partial_substantiality}
A slightly more precise version of 
Lemma~4.11 from~\cite{tangencies} can be stated as follows: {\emph{if $\mathcal P_\lambda$ is a 
set of periodic points that can be followed holomorphically throughout $\Lambda$ and is 
dense in $\jstar$ for some parameter $\lo$, and  if $q(\lo)$ is any saddle point satisfying the non-degeneracy assumption,  then $q$ persists as a saddle throughout $\Lambda$}}. In this version
 the  substantiality assumption is replaced by an assumption on the specific point $q$, but the proof is the same.  In particular, if instead of substantiality we assume that a set 
 of asymptotic density $1$ (resp. of positive upper density)  
 of saddle points satisfies the non-degeneracy assumption, then 
 (VI)$\Rightarrow$(II)  (resp.   (VI)$\Rightarrow$(III)). 
\end{rem}
\end{vlongue}

\begin{rem}\label{rem:local_stability}
For each of the stability conditions there is a corresponding local version:  a family is \emph{locally 
stable} of type (X) over $\Lambda$ if  every $\lambda\in \Lambda$ admits 
 a neighborhood $\mathcal N$ such that the family is stable of type (X) over $\mathcal N$. 
Conditions (I),  (V) and (VI) are local in the sense that locally stable is equivalent to stable,  
however for   conditions  (II),  (III) and (IV) we obtain a priori  two different notions because 
 the set of persistent saddle points may depend on $\mathcal N$. 
\end{rem}

\begin{convention} 
In the situations that we have to consider, we cannot guarantee 
 that the families of interest are substantial (for instance a family of automorphisms of Jacobian 1 is not substantial).
So we do not care about substantiality, and by convention,  
\begin{center}
\emph{unless otherwise specified   stability means condition (III)}. 
\end{center}
The reason is that our approach to Theorem~\ref{mthm:stable} based on   the divergence of $\chi^+(\mu_f)$  is naturally associated to stability of type (III) (see Lemma~\ref{lem:jacobian2}). 
Still the other conditions are also useful: we will use the implication (III)$\Rightarrow$(VI) several times and    in Section~\ref{sec:huguin}  we have to work with type (I) families.  \end{convention}

\begin{vlongue}
\subsection{A worked out example} \label{subs:eg_fa}  
Consider the family $f_a(x,y) = (ay+x^2, x)$. A first observation is that 
$ f_a\inv(x,y)  = \lrpar{y, a\inv x -a\inv y^2 }$
 is conjugate to 
$(x,y)\mapsto (a\inv y - a\inv x^2, x)$ by $\tau\colon (x,y)\mapsto (y,x)$, and then   to 
$f_{1/a}$ by a further diagonal conjugacy, so it is enough to analyze the family  for $\abs{a}\leq 1$. 
We set $J_a=J_{f_a}$, $\mu_a=\mu_{f_a}$, etc.

\subsubsection{Existence of bifurcations} For small (and hence for large) $\abs{a}$, $f_a$ is uniformly hyperbolic and $f_a\rest{J_a}$  is 
conjugate to the natural extension of $x^d$. In particular the connectedness locus is not compact in this family (it contains all large enough $a\in \C$).  By ~\cite{bs3}, the  Lyapunov exponents satisfy 
$$\chi^+(\mu_a) >0>\chi^-(\mu_a), \quad
 \chi^+(\mu_a) + \chi^-(\mu_a)  = \log\abs{a},  \text{ and} \quad \abs{\chi^\pm(\mu_a)}\geq \log d.$$ Furthermore
  by~\cite{bs6}, we have $\chi^+(\mu_a) = \log d$ 
  for small $a$. By symmetry, $\chi^-(\mu_a) = -\log d$ for large $a$. For $\abs{a}\geq 1$, we obtain
  \begin{equation}
  \chi^+(\mu_a) = \log \abs{a}  - \chi^-(\mu_a) \geq \log d + \log \abs{a},
  \end{equation}
  with equality when $\abs{a}$ is large. Thus, 
  the function $a \in \C^*\mapsto   \chi^+(\mu_a)\in \R $ is  equal to $\log(d)$ when $\abs{a}$ is small and to $\log\abs{a}+\log(d)$ when $\abs{a}$ is large, so it is not 
  harmonic. This non-harmonicity can be considered as a kind of bifurcation. Moreover, 
  since the Jacobian is not constant, Lemmas~\ref{lem:jacobian} and~\ref{lem:jacobian2} below
  imply that bifurcations of types (I), (II), and (III) arise.
  
  Still, this does not guarantee that bifurcations of type (V) occur, that is, it could a priori be the case that $\jstar_a$  still moves holomorphically with $a$. 
  What might happen is the following: when $\abs{a}$ is small  all periodic points, except (0,0), are saddles, by hyperbolicity of the dynamics; then, as $\abs{a}$ increases, periodic points could remain saddle points,  with some of them bifurcating exactly when $\abs{a} = 1$, 
where  both eigenvalues would cross the unit circle at the same time
without destroying the holomorphic motion. This is exactly the issue of substantiality.

\subsubsection{The fixed point $\alpha_a$ and bifurcations of type (V)}\label{subsub:alpha}  Since $f_a$ is quadratic, it has  2 fixed points and one   orbit of period 2. The fixed points are $\alpha_a=(0,0)$ and $\beta_a=(1-a, 1-a)$; they are continuations of the corresponding fixed points of $f_0\simeq  x^2$.

The first one $\alpha_a$ has  eigenvalues $\lambda_1, \lambda_2=\pm\sqrt{a}$ so it is attracting for $\abs{a}<1$ and repelling for $\abs{a}>1$. Suppose $\abs{a}=1$, and $a$ is not a root of unity. Then, there is no resonance between $\lambda_1$ and $\lambda_2$ that is,  no relation of the form 
  $\lambda_i = \lambda_1^k \lambda_2^\ell$ with $k, \ell\in \N$ and $k+\ell\geq 2$. Moreover, 
  \[
  \abs{\lambda_1^k \lambda_2^\ell  - \lambda_i} = \abs{\lambda_1^{k+\ell} \pm\lambda_1} =  \abs{(\sqrt{a})^{k+\ell-1} \pm 1},
  \] so for generic $a\in \mathbb{S}^1$ we get a lower bound 
of 
Diophantine
type (see e.g.~\cite{zehnder}), and the fixed point $\alpha_a$ is linearizable, 
so it creates a Siegel ball. 
  By~\cite[Lem. 4.11]{tangencies}, 
  these Siegel balls cause bifurcations of type (V).  Thus we have shown:

\begin{pro}
The circle $\set{\abs{a}=1}$ is contained in the (type (V)) bifurcation locus. 
\end{pro}  

\begin{rem}
Since the eigenvalues of $\alpha_a$ are opposite, they satisfy the persistent relation $\lambda_1^2\lambda_2^{-2}=1$. Thus, the family is not substantial! But it is almost so. Indeed, suppose $\gamma$ is some saddle periodic point  with eigenvalues satisfying a persistent relation  $\lambda_1^s\lambda_2^t =c$ (with $\abs{c}=1$  and $s,t\in \R$, not both $0$). Any such  periodic point can be followed holomorphically for small 
$\abs{a}$, and as $a$ goes to $0$, one of the eigenvalues tends to $0$ while the other tends to a complex number of modulus $>1$, in contradiction  with $\abs{c}=1$.
\end{rem}

\subsubsection{Bifurcations in the unit disk}  A natural question is whether the bifurcation locus is larger  than   $\set{\abs{a}=1}$. 
For this, it is enough to work with $\abs{a}<1$, 
and we  need to find some   saddle periodic point  
 bifurcating  to a sink. 
 
The first candidate is the second fixed point $\beta_a= (1-a,1-a)$. 
 Its eigenvalues are solutions of 
 the equation $\lambda^2 - 2(1-a) \lambda -a=0$. If some eigenvalue crosses the unit circle then $e^{i\theta}$ must be solution of this equation, which easily shows that $\abs{a}=1$ (in fact, $a$ must be on an arc of the unit circle of angle $2\pi/3$). Thus, $\beta_a$ is a saddle for all $a$ in the complement of the unit circle. 
 
 It follows from the calculations of Proposition~\ref{pro:huguin}, 
in particular from Equations~\eqref{eq:huguin_period1} and~\eqref{eq:huguin_period2},  that the period $2$ orbit has exactly the same behavior. Indeed, since all repelling   period-2 points of $p(x) = x^2/(1-a)$  have multipliers $d^2$ (including the repelling fixed one), 
the period-2 points of  $f_a$ have the same multipliers, except for $\alpha_a$  (\footnote{A more delicate result is that both $\beta_a$ and the period 2 orbit are neutral precisely for $a  = e^{i\rho}$, $ \rho\in [-\pi/3,  \pi/3]$, and non-persistently resonant, so they also simultaneously create Siegel balls in that interval. }).    So we must go to  period at least 3, which does not seem 
to be   tractable by hand.

\begin{pro}
The bifurcation locus is not contained in the circle $\set{\abs{a} = 1}$. In particular there exist parameters with   $\abs{a}<1$ with disconnected Julia sets. 
\end{pro}

\begin{proof}
Numerical calculations show that  the parameter $a = -0.669+0.73 i $  admits an attracting cycle of period 3 containing $(x,y)\simeq(0.111236 -0.069787i , -0.017098+ 0.856247i)$: see Figure~\ref{fig:slice}, and note that since $\abs{a}\simeq 0.990$, attraction is very slow. The second assertion follows from the fact that a stably connected dissipative family is stable (see~\cite[Thm. 5.7]{tangencies}). 
\end{proof}

\begin{figure}[h]
\includegraphics[width=\textwidth]{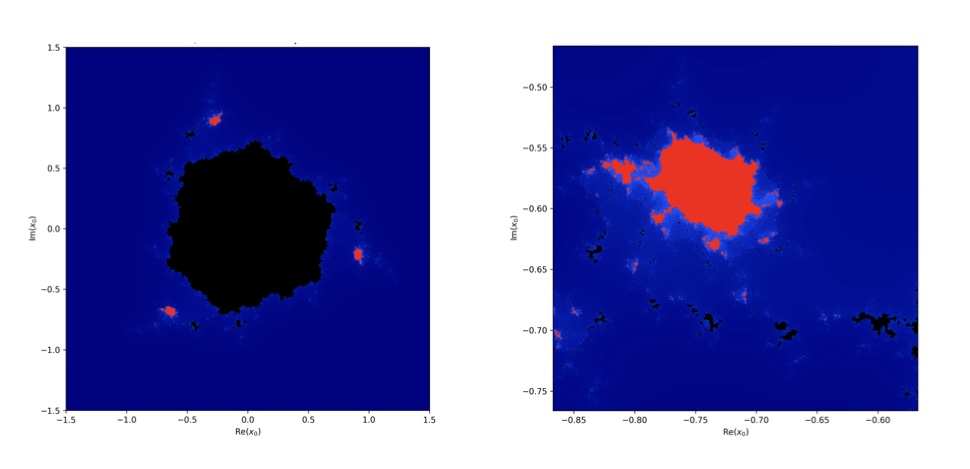}
\caption{{\small Slice of the filled Julia set of $f_{-0.669+0.73 i}$ by the line $\set{y=0}$. The basin of $(0,0)$ is colored black, and the basin of the period-3 cycle is in red. Detail of the picture on the right.}}\label{fig:slice}
\end{figure}

Here are some comments on this numerical calculation. We have explored the parameter space 
by testing 500 different parameters for each value of $\abs{a}$  (incrementing by  steps of 0.01),
looking for  orbits on the slice $\set{y=0}$ 
neither converging to (0,0) nor escaping to infinity (10,000  different initial conditions). 
Note that any attracting basin must intersect this slice. 
The first example arises for $\abs{a}=0.98$ (which is the continuation of the one of Figure~\ref{fig:slice}), and for 
$\abs{a} = 0.99$ we obtain 12 such parameters (in conjugate pairs).  
This suggests that the stable  component   containing the origin    almost fills the whole unit disk.  It  may be 
qualified as the  ``main hyperbolic component'' of the family (it is indeed a hyperbolic  component  by~\cite[Thm C]{berger-dujardin}).

\begin{que}
Does the closure of the main hyperbolic component  touch the unit circle? 
\end{que}

A positive answer to this question would provide the first known example of 
 a conservative automorphism  with connected Julia set.    
\end{vlongue}

\section{Periodic points and multiplier spectra}\label{sec:spectra}

\subsection{Update on periodic points}

Let $f\in \Aut(\C^2)$ be   loxodromic. We conjugate $f$ so 
that $f$ becomes a  composition of monic and centered Hénon maps of multidegree $\mathbf d=(d_k, \ldots , d_1)$. 
We work in the parameterized family $ (f_\lambda)_{\lambda\in \mathcal H^k_{\mathbf d}}$, so that $f = f_{\lambda}$ for some $\lambda$, 
and we identify a map $f$ with the corresponding parameter $\lambda$.   
Define
$$\widehat\Fix_n    = \set{(f, z)\in \mathcal H^k_{\mathbf d}\times \C^2 \; ; \;  f^n(z) = z}, $$ 
which is  an algebraic  subset of $\mathcal H^k_{\mathbf d}\times \C^2$. 
The projection $\widehat\Fix_n \to \mathcal H^k_{\mathbf d}$ is surjective and finite to one.
Defining the corresponding fibered version of $\Per_n(f)$ as an analytic subset of $\mathcal H^k_{\mathbf d}\times \C^2$ 
 is not obvious, because periodic points of period $n$ may degenerate to points of lower period. 
As a preliminary remark, suppose that $f\in \mathcal H^k_{\mathbf d}$ is close to the parameter $((0, \ldots , 0), (x^{d_k}, \ldots , x^{d_1}))$. 
Then, as a consequence of  uniform hyperbolicity, all periodic points of $f$ are simple.
It follows that {\emph{all irreducible components of $\widehat\Fix_n$ have multiplicity $1$}}.

We define the algebraic variety $\widehat \Per_n\subset \mathcal H^k_{\mathbf d}\times \C^2$ by
$$\widehat \Per_n = \overline{ \widehat\Fix_n\setminus \bigcup_{k\vert n, k<n} \widehat\Fix_k},$$  which is a union of irreducible components of $\widehat \Fix_n$ 
(\footnote{Note that in this definition we can take the closure for the euclidean or the (analytic or algebraic) Zariski topology.
}). 
By construction, at the generic point 
$(f,z)$ of $\widehat \Per_n$, $z$ has exact period $n$. 
For every 
$f\in \mathcal H^k_{\mathbf d}$, we set 
$$\Per_n^*(f)=\widehat \Per_n \cap (\set{f}\times \C^2).$$ 
By definition,  
$\Per_n^*(f)$ is the set of points with \emph{formal period}
$n$, a terminology due to Milnor (note that a periodic point can have several formal periods).
By construction $\widehat \Per_n$ is contained in $\widehat \Fix_n$, so  $\Per_n^*(f)$ is made of points of period dividing $n$. 
Each $z\in \Per_n^*(f)$ comes equipped with a multiplicity, 
which is the intersection multiplicity of $\widehat \Per_n$ with $\set{f}\times \C^2$ at $z$. 
The number $p_n$ of points in $\Per_n^*(f)$ counted with multiplicities is the cardinality of $\Per_n$ at the generic parameter. 
This sequence $(p_n)$  is determined inductively by
$d^n  = \sum_{k\vert n} p_k.$
 
\begin{pro}
\label{pro:basics_on_per*_vs_fix}
For every $n\geq 1$ we have the following properties
\begin{enumerate}[\rm (1)]
\item $\Per_n(f)\subset \Per_n^*(f)$, that is exact period $n$ implies formal period $n$; 
\item  $\displaystyle \widehat\Fix_n  = \bigcup_{k\vert n} \widehat \Per_k$;
\item if $z\in \Per_n^*(f) $ is a simple periodic point (as a point of period $n$), 
then $z\in \Per_n(f)$. 
\end{enumerate}
\end{pro}

If $1$ is an eigenvalue of $z\in \Per_n^*(f)$, the period of $z$ can be a proper divisor of~$n$. 
The corresponding situation in dimension $1$ is completely understood, thanks to the 
theory of dynatomic polynomials (see \S 4.1 in~\cite{silverman:book}, but be aware that our notation is not completely consistent with Silverman's choices). For $\Aut(\C^2)$, the specifics still need to be worked out. 

\begin{proof}
By construction, $\widehat \Per_n \subset \widehat \Fix_n$. If $z\in \Per_n(f)$, then 
$z\notin \Fix_k(f)$ for any $k<n$, thus 
$(f,z)\in \widehat\Fix_n\setminus \bigcup_{k\vert n, k<n} \widehat\Fix_k$,   in particular
$z $ belongs to $\Per_n^*(f)$. This proves~(1). 

For  Assertion~(2), note that $\widehat \Per_k \subset \widehat \Fix_k$ and $\widehat \Fix_k\subset \widehat \Fix_n$ if $k\vert n$, which proves the reverse inclusion.  We prove the direct inclusion by induction on $n$. The case $n=1$ holds by definition. Now take  $(f,z) \in \widehat\Fix_n$. If $z \in \Per_n(f)$ then $(f,z) \in \widehat \Per_n$ by~(1). Otherwise $(f,z)$ belongs to $\Fix_k$ for some proper divisor $k$ of $n$, and the induction hypothesis shows that $(f,z)\in \widehat \Per_\ell$ for some $\ell$ dividing $k$, as desired.
 
Finally, let  $z\in  \Per_n^*(f) $ be a simple periodic point. Since $z$ is simple, 
it admits a unique local continuation as a solution of $f^n(z)  =z$, 
so there is a  unique branch of $\widehat \Fix_n$ in a 
euclidean neighborhood of  $(z,f)$. 
If the period of $z$ were a proper divisor $k$ of~$n$, 
this branch would be contained in  $\widehat \Fix_k$, and $z$ would not belong to the closure of 
$\widehat\Fix_n\setminus \bigcup_{k\vert n, k<n} \widehat\Fix_k$. Thus $z\in \Per_n(f)$ and we are done. 
\end{proof}

\subsection{Trace and multiplier spectra} \label{subs:spectra}
Let $q$ be a positive integer and let  $\C^{(q)}$ denote the symmetric product of $q$ copies of $\C$ that is, $\C^q/\Sym_q$ where the symmetric group acts by permutations. Points of $\C^{(q)}$ will be called {\emph{multisets}} (of length $q$). Thus, a multiset is a finite set $\set{x_1, \ldots x_q}$ in which the elements $x_i\in \C$ are repeated according to their multiplicity; equivalently, it is a divisor $\sum_j n_j z_j$ supported by some points $z_j\in \C$ with multiplicities $n_j\geq 1$ such that $\sum_j n_j=q$.
The elementary symmetric functions provide an isomorphism $\iota\colon \C^{(q)}\to \C^q$. 

Let us fix some integer $n\geq 1$. On the variety  $\mathcal H^k_{\mathbf d}\times \C^2$, the function 
$$
\tr_n\colon (f, z)\mapsto \tr(D_zf^n)
$$
is regular. By restriction, it induces a regular function $\tr_n\colon \widehat \Per_n\to \C$, and we can 
then define a regular map $\Trace_n\colon \mathcal H^k_{\mathbf d} \to \C^{(p_n)}$ such that 
$$ 
\Trace_n(f)=\set{\tr(D_zf^n) \; ; \;  z\in \Per_n^*(f)},$$
viewed as a multiset of length $p_n$.  
We define the {\emph{trace spectrum}} and the {\emph{multiplier spectrum}}  of $f$ by
\begin{align}
\Trace(f) &=  \lrpar{\Trace_n(f)}_{n\geq 1} \in \prod_{n\geq 1} \C^{(p_n)} \\
\Mult(f)  &= ( \jac(f),  \Trace(f)) \in \C^*\times  \prod_{n\geq 1} \C^{(p_n)}.
\end{align}
Since $\jac(f)^n=\det(D_zf^n)$ for every fixed point of $f^n$, the multiplier spectrum encodes the conjugacy invariants $(\tr, \det)$ of $D_zf^n$, whose  knowledge is equivalent to that of the eigenvalues $\{\lambda_1, \lambda_2\}\in \C^{(2)}$,  at all points of formal period $n$, for every $n\geq 1$. The use of the formal period instead of the exact period is meant to guarantee good analytic and algebraic behavior for these quantities. 

For saddle points, we directly consider the unstable multiplier and define
 $$\SMult^u_n(f)  = \set{\lambda^u(z) \; ; \; z\in \SPer_n(f)} \in \C^{\#\SPer_n(f))},$$
\begin{equation*}\SMult^u(f) = \lrpar{\SMult^u_n(f)}_{ n\geq 1} \text{ and }\SMult(f)  =  ({ \jac(f)}, {\SMult^u(f)}).\end{equation*}
(Since $\#\SPer_n(f)$ depends sharply on $f$, we do not claim that $\SMult^u_n(f)$ depends in a regular or analytic way on $f$.)

\begin{rem} \label{rem:determined_degree}
Since the number of  periodic (resp.\   saddle periodic)
points  of period $n$ grows like $d^n$, 
 the dynamical degree  of a loxodromic automorphism is tautologically determined by 
its  multiplier (resp. unstable multiplier) spectrum.
\end{rem}

\subsection{From    Theorem~\ref{mthm:stable} to rigidity}

\begin{pro}
\label{pro:multipliers}
Let $(f_\lambda)_{\lambda\in \Lambda}$ be an algebraic family of loxodromic 
polynomial automorphisms of $\C^2$, in which any stable irreducible 
algebraic subfamily  
 is trivial.  For every $n\geq 1$, consider a finite subset $E_n\subset \C$. Then the set
   $$\set{f\in \Lambda \; ; \;  \forall n\geq 1, \Trace_n(f) \subset E_n} $$ is finite. Likewise  the set 
$$\set{f\in \Lambda \; ; \; \forall n\geq 1, \SMult^u_n(f) \subset E_n}$$ is discrete. 
\end{pro}

\begin{proof}[Proof for the trace spectrum]
For $n\geq 1$, set 
\begin{equation}
\Sigma_n =   \set{\lambda \in \Lambda \; ; \;  \Trace_n(f_\lambda) \subset E_n} \, \text{ and } \, 
\Sigma = \bigcap_{n\geq 1} \Sigma_n. 
\end{equation}
These are algebraic subsets of $\Lambda$ because each of the maps $\Trace_n$ is regular.
To conclude that $\Sigma$ is finite, we fix an irreducible component of $\Sigma$, still denoted by $\Sigma$ for simplicity, and we show that $(f_\lambda)_{\lambda\in \Sigma}$  is  stable. 
We follow the same strategy as for~\cite[Lem. 8.5]{rigidity}, 
except that we work with traces instead of multipliers. 

Fix $\lambda_0\in\Sigma$ and set $f_0 = f_{\lambda_0}$. Fix $\alpha>0$   such that $\chi^+(\mu_{f_0})\geq \log^+\abs{\jac(f_0)}+2\alpha$. Let $\mathcal N$ be an open  ball
around $\lambda_0$ in which $ \log \abs{\jac(f_\lambda)}   - \log^+\abs{\jac(f_0)} < \alpha/4$.
Theorem~\ref{thm:equidistribution_reinforced} provides
a  set $\SPer^+(f_0)$ of  saddle points  of asymptotic density 1, such that $\chi^u(z_0)\geq \chi^+(\mu_{f_0}) - \alpha$ for 
every $z_0\in \SPer^+(f_0)$.  
Let us show  that points in 
$\SPer^+(f_0)$ of sufficiently large   period 
persist as saddle points for all parameters in $\mathcal N\cap \Sigma$. 

Fix $z_0\in \SPer^+(f_0)$, of period $n$. We have 
 \begin{equation}
\abs{\tr(D_{z_0}f_0^n)} \geq \abs{\lambda^u(z_0)}- 1 \geq \exp\lrpar{n(\log^+\abs{\jac(f_0)}+\alpha)}-1. \end{equation}
Let  $z_\lambda$ be the analytic continuation of $z_0$ as a saddle point, for $\lambda$ in some connected 
neighborhood $\mathcal N(z_0)$ of $\lambda_0$ in $\Sigma$ (that  depends a priori on $z_0$). 
Along $\Sigma$, 
$\Trace_n(f_\lambda)$ is contained in the finite set $E_n$, so 
 $\lambda \mapsto \tr(D_{z_\lambda}f_\lambda^n)$ is constant on
 $\mathcal{N}(z_0)$, equal to $\tr(D_{z_0}f_0^n)$. Thus for $\lambda\in \mathcal N(z_0)$ we have $\abs{\lambda^u(z_\lambda) - \tr(D_{z_0}f_0^n)}<1$, and then
\begin{equation}\abs{\lambda^u(z_\lambda)}\geq \abs{\tr(D_{z_0}f_0^n)} -1 \geq 
\exp\lrpar{n(\log^+\abs{\jac(f_0)}+\alpha)}-2.\end{equation} 
Fix  $n_0$ so large  that for   $n\geq n_0$,
\begin{equation} \label{eq:apriorichiu}
\unsur{n}\log \lrpar{\exp\lrpar{n(\log^+\abs{\jac(f_0)}+\alpha)}-2}\geq\log^+\abs{\jac(f_0)} + \alpha/2.
\end{equation} 
Then $ \chi^u(z_\lambda)\geq   \alpha/2$ for all $z_0\in \SPer^+(f_0)$ of period $\geq n_0$, and  the definition of $\mathcal N$ gives
\begin{equation}
\chi^s(z_\lambda)  = \log \abs{\jac(f_\lambda)}  - \chi^u(z_\lambda)\leq 
 \log \abs{\jac(f_\lambda)}   - \log^+\abs{\jac(f_0)} - \alpha/2  \leq - \alpha/4. 
\end{equation}
Thus, the analytic continuation $z_\lambda$ of $z_0$ remains a saddle periodic point when 
 $\lambda\in \mathcal N \cap  \Sigma$, with the uniform estimates
$  \chi^s(z_\lambda) \leq  -\alpha/4 \text{ and } \chi^u(z_\lambda)\geq   \alpha/2. $ 
Thus, no bifurcation is possible and we can take  $ \mathcal N(z_0)= \mathcal N\cap \Sigma$, as was to be shown.

This argument shows that $(f_\lambda)_{\la\in \Sigma}$ is  locally stable of type (II),
hence also of type (III). To enhance this to global stability, 
we follow an idea used for \cite[Prop. 8.6]{rigidity}.
We will show that the Jacobian must be constant on $\Sigma$, so that we can choose $\mathcal N  = \Sigma$.  
By Remark~\ref{rem:local_stability},  $(f_\lambda)_{\la\in \Sigma}$ is globally stable of type (VI):
 all periodic points can be followed, but not necessarily as saddles. By way of contradiction, assume that 
 $\lambda \mapsto \jac(f_\lambda)$ is not constant on $\Sigma$ and fix a parameter $\lambda_1\in \Sigma$ and 
 $\alpha>0$ such that $\log \abs{\jac(f_{\lambda_1})} \geq \log \abs{\jac(f_{\lambda_0})} + 2\alpha$, 
 as well as a path $\gamma$  in $\Sigma$ joining $\lambda_0$ to $\lambda_1$. 
 Note that $\chi^+(\mu_{f_0})\leq \log \abs{\jac(f_{\lambda_0})}$, hence 
 $\log \abs{\jac(f_{\lambda_1})} \geq \chi^+(\mu_{f_0})+2\alpha$. Fix 
 a  set $\SPer^+(f_0)$ of  saddle points  of asymptotic density $1$, such that $\chi^u(z_0)\leq \chi^+(\mu_{f_0}) + \alpha$ for 
every $z_0\in \SPer^+(f_0)$. 
For such a $z_0$ we have 
\begin{equation}\label{eq:borne_sup_trace}
\abs{\tr(D_{z_0} f_{\lambda_0}^n)} = \abs{\lambda^u(z_0) +\lambda^s(z_0)}
 \leq  \abs{\lambda^u(z_0) }+ 1 \leq \exp\lrpar{n (\chi^+(\mu_{f_0}) + \alpha)} +1.
\end{equation}
Let us follow $z_0$ along the path $\gamma$ up to some periodic point $z_1$. By assumption the trace stays constant so $\abs{\tr(D_{z_1} f_{\lambda_1}^n)}$ satisfies the inequality~\eqref{eq:borne_sup_trace}. If $z_1$ is a saddle point for $f_{\lambda_1}$
 we infer 
\begin{align}
\abs{\jac(f_{\lambda_1})} ^n\leq  \abs{\lambda^u(z_1) } &\leq \abs{\tr(D_{z_1} f_{\lambda_1}^n)} +1\\
&\notag =\abs{\tr(D_{z_0} f_{\lambda_0}^n)} +1 \leq \exp\lrpar{n (\chi^+(\mu_{f_0}) + \alpha)} +2.
\end{align}
But  $\log \abs{\jac(f_{\lambda_1})} \geq \chi^+(\mu_{f_0})+2\alpha$ so if $n$ is large enough we get a contradiction. Therefore, as we follow any point of $\SPer^+(f_0)$ of sufficiently large period along $\gamma$, we obtain a non-saddle point for $f_1$, which contradicts 
 the fact that $\SPer^+(f_0)$ has asymptotic density $1$. So 
 the Jacobian is constant on $\Sigma$, and we are done. 
 \end{proof}

\begin{proof}[Proof for the unstable  multiplier spectrum] 

This is more delicate because the sets
\begin{equation}\label{eq:Sigmaun}
\Sigma^u_n :=   \set{\lambda \in \Lambda \; ; \;  \SMult^u_n(f_\lambda) \subset E_n} \text{ and } \; \Sigma^u := \bigcap_{n\geq 1} \Sigma_n^u  
\end{equation}
are not  naturally algebraic, for being a saddle point is not an algebraic property. 
Fix $f_0\in \Sigma^u$ 
and assume that $f_0$ is not an isolated point of $\Sigma^u$. 
We will show that $\Sigma^u$ contains a positive dimensional stable algebraic family through $f_0$, which is a contradiction. 

Let $V_n$ be a  neighborhood of $f_0$ in $\Lambda$ 
in which all saddle points up to period $n$ persist. Since 
  $f_0$ is not an isolated point of $\Sigma_n^u$, 
 $\Sigma^u_n\cap V_n$ is an analytic subset of $V_n$ 
of positive dimension,  which is a union 
 of irreducible components (as analytic subsets in $V_n$)
 of the  restriction to $V_n$ of an algebraic subvariety of dimension $\geq 1$ at $f_0$. 
 It may have several local 
 irreducible components at $f_0$, but reducing $V_n$ if necessary, we can assume that $\Sigma^u_n\cap V_n$ admits a unique connected component in $V_n$, 
 which contains   $\Sigma^u\cap V_n$. Let $n_1$ be such that 
 for $n\geq n_1$ the germ of $\Sigma^u_n$ at $f_0$ is constant (i.e.\ does not depend on $n$). 
   
Since $f_0$ belongs to $\Sigma^u$, $\Sigma^u$ is not empty, thus for every  
  $0<\alpha<\chi^+(\mu_{f_0})/10$ 
and   large enough $n$, we can find points $t\in E_n$  such that 
$\unsur{n}\log\abs{t}\geq \chi^+(\mu_{f_0})-\alpha$. 
As in the first part of the proof, fix $\alpha>0$   such that 
$\chi^+(\mu_{f_0})\geq \log^+\abs{\jac(f_0)}+2\alpha$, 
and a neighborhood $\mathcal N  = \mathcal N(\alpha)$ in which 
$\log\abs{\jac(f)}< \log^+\abs{\jac(f_0)} + \alpha/4$. 
For   $f\in \mathcal N$, every periodic point with an eigenvalue $t$ such that 
$\unsur{n}\log\abs{t}\geq \chi^+(\mu_{f_0})-\alpha$ is automatically a saddle. 

Theorem~\ref{thm:equidistribution_reinforced} provides a set $\SPer^+(f_0)$ of saddle points  of asymptotic density $1$  
such that $\chi^u(z_0)\geq \chi^+(\mu_{f_0}) - \alpha$ for    $z_0\in \SPer^+(f_0)$. 
Fix   $z_0\in \SPer^+_n(f_0)$, with $n\geq n_1$. It can be followed holomorphically as $z =z(f)$ along $\Sigma^u_n\cap V_n$, 
and $\lambda^u(z)$ stays constant there. 
Consider the subset of $\mathcal N$,  defined by the condition that  $\Mult_n(f)$  contains $\lambda^u(z_0)$; 
it is the intersection of $\mathcal N$ with the algebraic variety of automorphisms $f$ such that $\lambda^u(z_0)$
is an eigenvalue of $Df^n$ at some point of period $n$;
it contains $\Sigma^u_n$ as a germ at $f_0$;  
we let   $S(z_0)$ be the union of its irreducible
  components (as analytic subsets in $\mathcal N$)
containing $\Sigma^u_n\cap V_n$. 
As in the first part of the proof, $z_0$ can be followed holomorphically along   
$S(z_0)$. Set $S^+  = \bigcap_{z_0\in \SPer^+(f_0)} S(z_0)$. This is an 
analytic subset
 in $\mathcal N$ containing the germ of $\Sigma^u_n$ for every $n\geq n_1$, so 
it has positive dimension (possibly greater than that of  $\Sigma_n^u\cap V_n$). By construction all points in $\SPer^+(f_0)$ can be followed holomorphically as saddles 
along $S^+$, so $(f_\lambda)_{\lambda\in S^+}$ is a stable family of type (II). 

Now we work inside this family $S^+$. By the implication (II)$\Rightarrow$(VI) of 
Proposition~\ref{pro:stability_conditions}, all saddle periodic points of $f_0$ 
can be followed holomorphically as periodic points 
in $S^+\cap \mathcal N$, so we can define  
an analytic subset $S\subset S^+$ in $\mathcal N$
 by the condition that the eigenvalue corresponding to the unstable multiplier at $z_0$ remains constant for \emph{every} $z_0\in \SPer(f_0)$. 
In modulus this eigenvalue stays larger than $1$ along $S$, so the exceptional phenomenon described in  
Remark~\ref{rem:substantiality} cannot occur and we conclude that $S$ is a stable family of type (I), which must coincide with $\Sigma^u\cap V_n$ for any $n\geq n_1$.

At this point  we have shown that $\Sigma^u$ is an analytic subset of $\Lambda$. 
Let $f_0$ be a non-isolated point of $\Sigma^u$: its is contained in a complex analytic  component $W$ of $\Sigma^u$ of dimension $k\geq 1$. 
We constructed a neighborhood $\mathcal N  = \mathcal N(f_0)$ such that 
$\Sigma^u\cap \mathcal N$ coincides with a union of components of $W'\cap \mathcal N$, for some algebraic subvariety $W'$ of $\Lambda$ of dimension $k$. 
By analytic continuation, $W$ is an irreducible component of $W'$ in the analytic sense, hence it coincides with an algebraic component of $W'$, and this component determines
a stable algebraic subfamily of type (I) in $\Lambda$, which  contradicts our assumptions.  
Therefore $\Sigma^u$ is a discrete set, as announced. 
\end{proof}

\begin{rem}\label{rem:type1}~
\begin{enumerate}[\rm (1)]
\item 
 The proof shows that
 the non-existence of  
stable algebraic families of type (II) (resp. type (I))  is enough to get the conclusion about the trace spectrum (resp. unstable multipliers).  
\item On the other hand, using the full strength of the type (III) assumption, we can relax  the condition $\Trace_n(f) \subset E_n$   (resp. $\SMult^u_n(f) \subset E_n$) by requiring that this property holds 
only along  some  subsequence $(n_j)$. 
  Thus, in the spectral rigidity theorems, it is enough to use the information on periodic points along a subsequence $(n_j)$. 
  The verification is left to the reader.
    \end{enumerate}
\end{rem}

 Recall from Remark~\ref{rem:determined_degree} that the degree is automatically determined by the spectrum. Hence 
the following result contains the statements of  Theorems~\ref{mthm:henon_multipliers} and~\ref{mthm:multi-Jacobian} relative to the trace spectrum.  
 
\begin{thm}\label{thm:traces}
If $\mathcal H$ is either $\mathcal H^1_d$ for some $d\geq 2$ or $\mathcal H^k_{\mathbf d,\mathbf a}$ for some $a\in (\C^*)^k$ and $\mathbf d\in \N^k_{\geq 2}$, then 
for any $f_0 \in \mathcal H$, the isospectral subset 
 $\set{f \in \mathcal H\; ; \;   \Trace(f ) = \Trace(f_0)}$
is finite.
\end{thm}

\begin{proof}
It suffices to apply  the first part of 
 Proposition~\ref{pro:multipliers} together with  Theorem~\ref{mthm:stable}.
 \end{proof}
 
Now, let us prove the $C^1$ rigidity assertion in  Theorem~\ref{mthm:C1}. 
 
\begin{thm}
\label{thm:C1}
Let $\mathcal H$ be either  $\mathcal H^1_d$ for some $d\geq 2$ or $\mathcal H^k_{\mathbf d,\mathbf a}$ for some $\mathbf a\in (\C^*)^k$ and $\mathbf d\in \N^k_{\geq 2}$. If $f_0$ is an element of $\mathcal H$, there is a discrete set $F\subset \mathcal H$ with the following property: if $f\in \mathcal H$ is   $C^1$ conjugate to $f_0$ on some neighborhood of $\jstar$, then $f\in F$. \end{thm}

 By ``$C^1$ conjugate in a neighborhood of $\jstar$'' we mean that there exists 
 a $C^1$ conjugacy $\varphi$ 
 between $f_0$ and $f$ whose domain $\mathcal U$
 contains a neighborhood of $\jstar(f_0)$.
  The uniqueness of the measure of maximal entropy, obtained in~\cite{bls}, implies 
  that   $\varphi_*\mu_{f_0} = \mu_f$, so  
 $\varphi(\jstar(f_0)) = \jstar(f)$.  

\begin{proof}
If $f_0$ and $f$ are $C^1$ conjugate in a neighborhood of $\jstar$, then for every saddle point $z_0$
of $f_0$, $\varphi(z)$ is a saddle point of $f$, with $\lambda^u(z)\in \set{\lambda^u(z_0), \overline{\lambda^u(z_0)}}$ and $\lambda^s(z)\in \set{\lambda^s(z_0), \overline{\lambda^s(z_0)}}$ 
(see the discussion preceding~\cite[Thm 7.5]{friedland-milnor}). So if we define
\begin{equation}\label{eq:En_C1}
E_n = \set{\lambda^u(z_0) \; ;\;  z_0\in \SMult^u_n(f_0)} \cup \set{  \overline {\lambda^u(z_0)} \; ; \;   z_0\in \SMult^u_n(f_0)}, \end{equation} we see that 
$\SMult^u_n(f)\subset E_n$ for every $n\geq 1$, and we conclude by Proposition~\ref{pro:multipliers}.
\end{proof}

\begin{vcourte}
\begin{rem} 
A question that arises naturally in many rigidity issues is whether there may exist a map with all unstable multipliers  equal  to $\kappa^n$ for some   $\kappa\in \C\setminus\overline\disk$, where $n$ is the period (see e.g.~\cite{BK1} or~\cite[\S 2]{rigidity}). By adapting the proof of Proposition~\ref{pro:multipliers} it can be shown  that the set of  maps of this type is discrete in $\mathcal H^1_d$ or $\mathcal H^k_{\mathbf d, \mathbf a}$.   
\end{rem}
\end{vcourte}

\subsection{Further results: uniformity and arbitrary fields}

\begin{thm}\label{thm:noether}
Let $\mathcal H$ be either of 
$\mathcal H^1_d$ or ${\mathcal H}_{\mathbf d, \mathbf a}^k$ for some
 $\mathbf d\in \N^k_{\geq 2}$ and $\mathbf{a}\in (\C^*)^k$. There exists $P$ and $N$ depending only on $\mathbf d$ such that for every $f_0\in \mathcal H$, the set 
$$\set{f\in \mathcal H  \; ; \; \  \Trace_n(g)=\Trace_n(f_0) \, \text{ for every }n\leq P}$$ has cardinality at most $N$. 
Moreover, this result holds for polynomial automorphisms of $\mathbb{A}^2(\mathbf K)$, for any algebraically closed field $\mathbf K$ of characteristic zero. 
\end{thm}

\begin{proof}
We first work over $\C$. Assume    $k\geq 2$, and 
denote by $\mathbf a(f)$ the multi-Jacobian of $f\in \mathcal H^k_{\mathbf d}$.  
 Introduce a   sequence of algebraic varieties
$$S_n  = \set{(f,g ) \in \mathcal H^k_{\mathbf d} \times \mathcal H^k_{\mathbf d}\;  ; \; 
 \forall \ell\leq n, \ \Trace_\ell(f)=\Trace_\ell(g) 
 \text{ and } 
\mathbf a(f)= \mathbf a(g) } .$$ Since this sequence is non-increasing ($S_{n+1}\subset S_n$)
there is an index $P\geq 1$ such that $S_n=S_P$ for all $n\geq P$. 
The first projection $\pi\colon S_P\to \mathcal H$ is a surjective morphism and by Theorem~\ref{thm:traces} all its fibers are finite. Thus, there is an integer $N$ such that every fiber of $\pi$ contains at most $N$ elements, as was to be shown. 

For $k=1$ the argument is the same by considering 
$$S_n  = \set{ (f,g) \in \mathcal H_d^1\times \mathcal H_d^1\;  ; \; \forall \ell\leq n, \ 
\Trace_\ell(f)=\Trace_\ell(g)}.$$

Now we work over 
$\mathbf K$.
The Friedland-Milnor classification holds in this case as well, 
and we consider    $\mathcal H (\mathbf K)$ to be either of 
$\mathcal H_d^1(\mathbf K) $ or   ${\mathcal H}_{\mathbf d, \mathbf a}^k(\mathbf K)$.
Let $N$ and $P$ be the integers defined in the previous paragraph (for the field $\C$). 
Pick  distinct elements   $ f_1, \ldots, f_M$ of  $\mathcal H(\mathbf K)$ 
such that $\Trace_n(f_m)=\Trace_n(f_1)$ for all~$m\leq M$ and $n\leq P$. 
Let $\mathbf L_0\subset \mathbf K$ be the 
field generated by the coefficients of the formulas defining the $f_n$
and  let $\mathbf L\subset \mathbf K$ be its algebraic closure. Since $\mathbf L$ is countable and of characteristic $0$, there is an embedding $\iota \colon \mathbf L \to \C$. 
Applying $\iota$ to the coefficients of the $f_m$, we get distinct 
 elements $f_m^\iota \in \mathcal H(\C)$ such that $\Trace_n(f_m^\iota)=\Trace_n(f_1^\iota)$
 for every $n\leq P$. Therefore $M\leq N$, and the proof is complete. 
\end{proof}

\begin{eg}
\label{eg:charac_p}
Let $\bfK$ be an algebraically closed field of characteristic $p > 0$, 
such as  $\overline{{\mathbf{F}}_p}$. 
Consider the additive automorphism of the affine plane $\A^2_\bfK$ defined by 
\begin{equation}
f(x,y)=(a y + sx^p, x)
\end{equation}
where $a\in \bfK^*$ is fixed and $s\in \bfK^*$ is a parameter. 
Then, the differential $Df_{(x,y)}$ does not depend on $(x,y)$, 
its trace is $0$ and its determinant is $-a$. If $\alpha$ is a square root of $a$, then $\tr(Df^n_{(x,y)})=\alpha^n+(-\alpha)^n$, 
so that the trace spectrum of $f$ does not depend on the parameter $s$. 
This shows that Theorems~\ref{mthm:henon_multipliers} and~\ref{mthm:multi-Jacobian}
do  not hold in characteristic~$p$. These examples are somehow similar to Lattès maps: Lattès maps come from endomorphisms of
abelian groups (elliptic curves) while additive automorphisms are automorphisms of the commutative group ${\mathbb{G}}_a \times {\mathbb{G}}_a$.
\end{eg}

\section{Hénon maps and Huguin's theorem}\label{sec:huguin}

This section provides a direct proof, based on~\cite[Thm A]{huguin:polynomials}, of the next  two theorems.

\begin{thm}\label{thm:henon_huguin1}
Let  $(f_\lambda)_{\lambda\in \Lambda}$ be an irreducible  algebraic family 
of complex Hénon maps, 
 such that    $\jac(f_\lambda)\not \equiv -1$. If it is stable of type (I), then it  is trivial.
\end{thm}

This is really a result about type (I) stability since, as we will see, it  deals specifically with  points of period $1$ and $2$.  
So even for Hénon maps with  Jacobian different from $ -1$,  this result is weaker than 
Theorem~\ref{mthm:stable}  (see the comment about the chain of implications after Proposition~\ref{pro:stability_conditions}). 
On the other hand, by Remark~\ref{rem:type1}\,(1), for Hénon maps of Jacobian  different from $-1$,  Theorem~\ref{thm:henon_huguin1} can be used  instead of   Theorem~\ref{mthm:stable}   to get  Theorem~\ref{mthm:C1} and the unstable multiplier version of 
   Theorem~\ref{mthm:henon_multipliers}. An amusing consequence is that  global considerations about periodic points of 
   periods $1$ and $2$ in $\mathcal H^1_d$ lead to local unstable multiplier rigidity (resp.\ $C^1$ rigidity  on $\jstar$), 
   even for maps whose points of periods 1 and 2 are all attracting.
   
\begin{thm}\label{thm:henon_huguin2}
A complex Hénon map $f$ of given degree
 whose Jacobian is different from $-1$ is determined up to finitely many choices by 
the multipliers of its periodic points of period $1$ and $2$ (or equivalently by $\jac(f)$, $\Trace_1(f)$ and $\Trace_2(f)$). That is, for any $f_0\in \mathcal H^1_d$ such that $\jac(f_0)\neq -1$ the set 
$$ \Sigma_2(f_0)= \set{f\in \mathcal H^1_d\; ; \;  \jac(f)=\jac(f_0),  \Trace_1(f)=\Trace_1(f_0), \ 
\Trace_2(f)=\Trace_2(f_0)}$$ is finite.  
Furthermore, for  $f_0$ in a dense Zariski open subset of $\mathcal H^1_d$, $\Sigma_2(f_0)   = \set{f_0}$. 
\end{thm}

These two theorems stem from a computational miracle which does not hold for  
Jacobian $-1$. 
The following example shows that 
Theorem~\ref{thm:henon_huguin2} does not carry over to that case.

\begin{eg}\label{eg:jacobian-1}
The general (reduced) form of a Hénon map of Jacobian $-1$ is 
$f(x,y) = (y+p(x), x)$, with $p$ monic and centered. The equation 
$(y_0+p(x_0), x_0)  = (x_0, y_0)$ for fixed points  is equivalent to   
\begin{equation}
x_0=y_0 \;  {\text{ and }} \; p(x_0) = 0 
\end{equation}
and the trace of the differential at such a fixed point equals $p'(x_0)$. For fixed points of $f^2$ we obtain 
\begin{equation}
\begin{cases} x_0+p(y_0 + p(x_0) ) = x_0 \\
 y_0+p(x_0) = y_0
 \end{cases}  
{\text{ or equivalently }} 
 \begin{cases}  p(x_0) = 0 \\ p(y_0) = 0\end{cases} .
\end{equation}
Among these solutions, points of period 2 are those for which $x_0\neq y_0$. The computation of the differential gives
\begin{equation}
\tr (D_{(x_0, y_0)}f^2 )= p'(x_0)p'(y_0)+2.
\end{equation}
 Let us choose for instance $p_\lambda(x) = (x -\lambda)^2(x+\lambda)^2$. Then all solutions of 
 $p_\lambda (x_0) = 0$ satisfy 
   $p'_\lambda(x_0)=0$, and we conclude that the family 
   $f_\lambda(x,y) = (y+ (x^2-\lambda^2)^2, x)$ 
   violates the conclusions of Theorem~\ref{thm:henon_huguin2}. Note that 
Theorem~7.1 of~\cite{friedland-milnor} asserts that Hénon maps of degree $2$ and $3$ are determined by the multipliers of their fixed points, therefore  $4$ is the smallest possible degree for such an example.   \qed 
\end{eg}

We now proceed to the proof of Theorems~\ref{thm:henon_huguin1} and~\ref{thm:henon_huguin2}. 

\begin{lem}\label{lem:jacobian}
Let $(f_\lambda)_{\lambda\in \Lambda}$ be an algebraic family  of loxodromic automorphisms, with $\Lambda$ irreducible. 
If it is stable of type (I), then 
the stable and unstable multipliers of all saddle periodic points, as well as the jacobian
 $\jac(f_\lambda)$, are constant on $\Lambda$.  
\end{lem}

\begin{proof}
We can assume that $\dim(\Lambda)=1$. 
We replace $(f_\lambda)_{\lambda\in \Lambda}$ by 
$(f_{\pi(\lambda)})_{\lambda\in \hat \Lambda}$, where $\pi:\hat\Lambda\to \Lambda$ is a desingularization, and assume $\Lambda$ smooth: 
$\Lambda$  is a compact  Riemann surface punctured in finitely many points. 
By assumption, all saddle periodic points of $(f_\lambda)$ can be locally
 followed holomorphically along $\Lambda$. Fix $\lo\in \Lambda$, and let 
$n$ be such that $\SPer_n(f_\lo)$ is non-empty. There is a finite cover $\Lambda'$ of $\Lambda$, which 
is also a Riemann surface of finite type,  
such that   all points in $\SPer_n(f_\lo)$ can be  followed   globally along~$\Lambda'$. For 
$z_{\lo}\in \SPer_n(f_\lo)$ with continuation $z_\lambda$, the function 
$\lambda\in \Lambda' \mapsto \lambda^u(z_\lambda)$ is   holomorphic and takes values in 
$\C\setminus \overline\disk$, hence it is constant. Likewise, the stable multipliers are constant. It follows that 
$\jac(f_\lambda)$ is constant on $\Lambda'$, hence on $\Lambda$.  
\end{proof}

For $f\in \Aut(\C^2)$, define $M_k(f)$ by 
\begin{equation}
M_k(f) = \max_{z\in \Per_k(f)} \lrpar{\unsur{k} \log\abs{\tr D_z f^k}}. 
\end{equation}

\begin{pro}\label{pro:huguin}
Let $(f_n)\in (\mathcal H^1_d)^\N$ be a diverging sequence of complex Hénon maps with fixed Jacobian different from $-1$. Then 
$$\lim_{n\to\infty}\max\lrpar{M_1(f_n), M_2(f_n)}  = +\infty.$$
\end{pro}

\begin{proof}
This relies on an explicit calculation. Write 
$f (x,y)= (ay + p(x), x)$, with $\jac(f)  =-a\neq -1$. 
The equation for fixed points is $(ay_0  +p(x_0) , x_0) = (x_0,y_0)$, or equivalently 
\begin{equation}\label{eq:huguin_period1}
   \unsur{1-a}p(x_0) = x_0    \; {\text{ and }} \; y_0 = x_0,
 \end{equation}
while the trace at such a fixed point is 
\begin{equation}
\tr(D_{(x_0, y_0)}f) = p'(x_0) = (1-a) \lrpar{\unsur{1-a}p}'(x_0).
\end{equation}
For points of period $2$ we obtain the following equivalent systems
\begin{equation*}
\begin{cases}
ax_0  +p(ay_0+p(x_0))=x_0 \\
ay_0+p(x_0) = y_0
\end{cases} 
 \Leftrightarrow \begin{cases}   y_0 = \unsur{1-a}p(x_0) \\  x_0 = \unsur{1-a}p(y_0)  \end{cases}
\Leftrightarrow \begin{cases}   
y_0 = \unsur{1-a}p(x_0) \\ x_0= \lrpar{\unsur{1-a}p}^{\circ 2}(x_0)   
\end{cases}
\end{equation*}
and the corresponding trace satisfies
\begin{align}\label{eq:huguin_period2}
\tr \lrpar{D_{(x_0, y_0)}f^2} &= p'(ay_0+p(x_0))p'(x_0)+2a \\
&= p'\lrpar{\unsur{1-a}p(x_0)} p'(x_0)+2a \notag \\ 
&= \notag
(1-a)^2 \lrpar{\lrpar{\unsur{1-a}p}^{\circ 2}}'(x_0) + 2a.
\end{align}

Now suppose that $a\neq 1$ is fixed and the sequence $f_n \colon (x,y)\mapsto (ay+p_n(x), x)$ diverges in $\mathcal H^1_d$ as $n$ goes to $+\infty$. Then $(p_n)$ diverges in $\mathcal P_d$, 
and $\lrpar{\unsur{1-a}p_n}$ diverges as well. By~\cite[Thm A]{huguin:polynomials}  
some multiplier  of period $1$ or $2$ of $\unsur{1-a}p_n$ must tend to infinity. The above calculations then 
imply that the  trace  of the corresponding periodic point  of periods $1$ or $2$ of $f_n$ 
must tend to infinity as well, and we are done. 
\end{proof}

\begin{rem}
The correspondence between periodic points of $f$ and $\unsur{1-a}p$ does not   survive for higher periods, 
so we are fortunate that period $2$ is enough in~\cite{huguin:polynomials}. 
\end{rem}

\begin{proof}[Proof of Theorem~\ref{thm:henon_huguin1}]
If $(f_\lambda)_{\lambda\in \lambda}$ is an irreducible stable (of type (I)) algebraic family of complex Hénon maps, 
then by Lemma~\ref{lem:jacobian}, $\jac(f_\lambda)$ is constant, as well as the unstable 
multipliers of saddle points. On the other hand, if  $\dim(\Lambda)\geq 1$, $\Lambda$ is unbounded in 
$\mathcal H^1_d$, so by Proposition~\ref{pro:huguin},   we can find a periodic point of period $\leq 2$ with 
an arbitrary large trace. Such a point must eventually be a saddle, with unstable multiplier tending to infinity. This contradiction finishes the proof. 
\end{proof}

\begin{proof}[Proof of Theorem~\ref{thm:henon_huguin2}] 
Let us start with a preliminary remark:  two Hénon maps $f_1$ and $f_2$ defined by 
$f_i(x,y)=(ay+p_i(x),x)$ are conjugate in $\Aut(\C^2)$ if and only if the affine conjugacy classes of the $p_i$ coincide: $[p_1] = [p_2]$.
Indeed, $f_1$ and $f_2$ must be  conjugate by a diagonal linear map of the form $(x,y)\mapsto (\beta x, \beta y)$ with $\beta\in \mathcal U_{d-1}$, which corresponds to linear conjugacy of the (monic, centered) polynomials $p_i$.

To prove Theorem~\ref{thm:henon_huguin2}, we write $f(x,y) = (ay+p(x), x)$ and apply 
the proof of  Proposition~\ref{pro:huguin}: it shows that, given  $a\neq 1$,  
$\Trace_1(f)$ and $\Trace_2(f)$ determine the multipliers of periodic points of period at most $2$ of 
$\unsur{1-a}p$. By \cite[Thm C]{huguin:polynomials}, there exists a Zariski open set $\mathcal U\subset \mathcal P_d$ such that a class $[q]\in \mathcal U$  is uniquely 
determined in $\mathcal P_d$ by the multipliers of its periodic points of period $1$ and  $2$. 
Also, if $p_1$ and $p_2$ are monic and centered polynomials   such that 
$\left[\unsur{1-a}p_1\right]= \left[\unsur{1-a}p_2\right]$, then $[p_1] = [p_2]$. 
Thus, if $p$ is monic and centered and 
 $\left[\unsur{1-a}p\right]\in \mathcal U$, then  $[p]$ is determined by $a$, 
$\Trace_1(f)$ and $\Trace_2(f)$.  
Define a Zariski open subset $\mathcal V\subset \mathcal H^1_d$ by  
$\mathcal V= \set{(a,p)\in \C^*\times \mathcal P_d^{\rm cm}\; ; \;  \left[\unsur{1-a}p\right]\in \mathcal U,  \ a\neq 1}$  (where the superscript `cm' stands for monic and centered, see \S~\ref{subs:polynomials}).
The preliminary remark implies that  any $f \in \mathcal V$  is determined up to conjugacy by $a$, 
$\Trace_1(f)$ and $\Trace_2(f)$, and the proof is complete.

Now, being algebraic, if  $\Sigma_2(f_0)$ is infinite  
it must admit  a component  of dimension $\geq 1$, so it is unbounded in $\mathcal H^1_d$, and as before we get a contradiction. 
\end{proof}

\section{Some  preliminary facts for Theorem~\ref{mthm:lyapunov}}\label{sec:more_prel}

\subsection{Crossed and Hénon-like mappings} \label{subs:hl}

\subsubsection{Definitions}
In this paper by definition a bidisk is a domain $\bb\subset \C^2$ that is 
biholomorphic to a standard open bidisk $D_1\times D_2$. It comes equipped with two projections $\pi_1:\bb\to D_1$ and $\pi_2:\bb\to D_2$, respectively called vertical and horizontal. 
 We always assume that the biholomorphism, hence also the projections extend holomorphically to some neighborhood of $\overline\bb$ in $\C^2$.
The vertical and horizontal boundaries are $\fr^v\bb = \pi_1\inv(\fr D_1)$  and $\fr^h\bb = \pi_2\inv(\fr D_2)$, respectively. 
An object (subvariety, subdomain, current) is \emph{vertical} (resp. \emph{horizontal}) if it
is disjoint from a neighborhood of $\fr^v\bb$ (resp. $\fr^h\bb$). 
We use a slightly different convention for lines:  by definition, a \emph{horizontal line} (resp. \emph{vertical line}) is a fiber of $\pi_2$ 
(resp. $\pi_1$). Since the $\pi_i$ are typically not linear, those lines are usually not lines in  the usual sense in  $\C^2$.  

Any horizontal submanifold $V$ in $\bb$ has a \emph{degree} $\deg(V)$, 
defined as the degree of the branched covering $\pi_1:V\to D_1$. 
Likewise, any horizontal closed positive current $T$ has a \emph{slice mass} $\sm(T)$, 
which is the   mass of the positive measure $T\wedge [L^v]$,  where $L^v$ is any vertical line. 
If $T = [V]$ is the integration current on a horizontal manifold, then $\sm(T)  =\deg(V)$. 
Finally, if $S$ (resp. $T$) is a horizontal (resp. vertical) closed positive current, 
the measure $T\wedge S$ is well-defined and its total mass satisfies $\norm{ T \wedge S} = \sm(T)\sm(S)$.  
We refer to~\cite{henonl} for these results. By convention, a horizontal disk $\Delta$ is a horizontal submanifold biholomorphic to a disk; in particular $\fr\Delta\subset \fr^v\bb$. 

If $\bb$ and $\bb'$ are two bidisks, a \emph{crossed mapping} $(f, \bb, \bb')$ is the data of an injective holomorphic map defined in a neighborhood of $\overline \bb$, such that (i) $f(\bb)\cap \bb'\neq\emptyset$, 
(ii) $f(\fr^v\bb)\cap \overline {\bb'} = \emptyset$, and (iii) $f(\overline \bb)\cap \fr\bb'\subset\fr^v\bb'$. 
We abuse notation and  simply write 
 $f\colon\bb\to \bb'$ even if by definition $f(\bb)$ is not contained in $\bb'$. 
 Crossed mappings were introduced in \cite{HOV2}, however this definition is from~\cite{henonl}. The degree of a crossed mapping is the degree of the horizontal curve $f(L^h)\cap \bb'$, where $L^h$ is any horizontal line in $\bb$. 
 If $\bb= \bb'$ we say that $f$ is a \emph{Hénon-like map}.
 
\subsubsection{Unramified maps}
The following considerations  are new (see   Ishii~\cite{ishii:non_planar} for related ideas). 
\begin{defi}
\label{defi:non_ramified}
A crossed mapping  $f\colon\mathbb{B}\to \mathbb{B}'$ is \emph{unramified over $\fr^v\bb'$} if
for every horizontal line $L^h$ in $\bb$,
\begin{enumerate}[(i)]
\item   $f(L^h)$ is transverse to $\fr^v \bb'$, and
\item $f(L^h)\cap \bb'$ is a disjoint
union of holomorphic disks.
\end{enumerate}
\end{defi}
An equivalent formulation of (i) is that 
there is a neighborhood $\mathcal N$ 
of $\fr^v\bb'$ such that for every $L^h$, $\pi_1\rest{f(L^h)\cap \mathcal N}$ 
is an unramified covering onto its image, hence the terminology. 

\begin{rem}\label{rem:non_ramified}
In~\cite{Wermer:1980} Wermer constructed a domain $\Omega\subset \C^2$ such that $\Omega$ is biholomorphic to a bidisk and $\Omega \cap (\C\times \set{0})$ is not polynomially convex. 
This suggest that there might exist crossed maps $(f,\bb, \bb')$
 with a horizontal line that satisfies (i) but fails to satisfy (ii).
An easily checkable sufficient  condition for~(ii) is the following: 
\emph{the natural projection $\pi_2':\bb\to D_2'$ extends holomorphically to $f(\bb)$. }
It holds for the unit bidisk and more generally 
when  $\pi_2'$ extends to a globally defined holomorphic map $\C^2\to \C$. 
 Indeed in this case $\set{z\in L^h, \  {\pi_2'\circ f(z)}\in D'_2}$ is a 
 union of disks by the maximum principle,
 and $L^h\cap f\inv(\bb')$ is the   union of the disks among those whose image intersects $\bb'$. 
 For completeness let us mention a more intrinsic version. 
\end{rem}

\begin{lem} \label{lem:runge}
Let $\Omega\subset \C^2$ be a Runge domain. If 
$D\subset \C^2$ is a  holomorphic disk   transverse to $\fr\Omega$ 
and $D\cap\Omega\Subset D$, then $D\cap \Omega$ is a disjoint union of disks.
Thus, Condition~(i) implies Condition~(ii) in Definition~\ref{defi:non_ramified} when  
  $\bb'$ is a Runge domain.
  \end{lem}

\begin{proof} 
Write $D\cap \fr\Omega =  \bigcup  \gamma_i$, where the $\gamma_i$ are smooth Jordan curves. 
Let  $D_i$ be the Jordan domain in $D$ bounded by $\gamma_i$.
We claim that the $D_i$ are disjoint and $D\cap \Omega = \bigcup D_i$. 
To prove this, observe that by Wermer's theorem~\cite{wermer:1958}, the polynomial hull $\widehat \gamma_i$ is either equal to $\gamma_i$ or a Riemann surface bordered by $\gamma_i$. But  $D_i\subset \widehat \gamma_i$, so $\widehat \gamma_i = \overline {D_i}$, and 
 the Runge property gives $D_i\subset \Omega$.  Therefore  the $D_i$ are disjoint and 
 $\bigcup D_i\subset D\cap \Omega$. 
 Finally, $D\setminus \bigcup \overline {D_i}$ is connected, is disjoint from $\fr \Omega$, and contains points outside $\Omega$ (near the $\gamma_i$, by transversality),  so it is entirely  outside $\Omega$. We conclude that  $D\cap \Omega = \bigcup D_i$, as claimed. 
 \end{proof}

\begin{propdef}\label{propdef:non-ramified}
Let $f:\bb\to \bb'$ be a crossed mapping that is unramified over $\fr^v\bb'$.
Let $\Omega$ be a component of $f(\bb)\cap \bb'$. 
Let $L$ be any horizontal line in $\bb$ and set $V = f(L)\cap \Omega$. 
Then $V$ is non-empty and connected, and its horizontal degree is independent of $L$. 
By definition, $\deg(V)$ is the \emph{degree} of the component $\Omega$.
 
Moreover, if  $W$ is any horizontal disk  contained in $\Omega$, $\deg(W)$ is a multiple of $\deg(V)$.  
\end{propdef}

 If $\deg(V)\geq 2$, the component  $\Omega$ is  said to be \emph{solenoidal}. 
 It will be clear from the proof that for the conclusion to hold 
 it is enough that 
 the non-ramification property holds   in $\Omega$, that is, we only need that 
$\pi_1\rest{f(L^h)\cap\Omega}$ is a  covering  near the boundary. 

\begin{proof}
  The first claim is that for every horizontal line $L$ in $\bb$, 
  $f(L)$ intersects $\Omega$. Indeed, if $L^v_1$ is any vertical line in $\bb'$
   intersecting $\Omega$, by the crossed mapping property 
  $f\inv(L^v_1\cap \Omega)$ is a vertical submanifold in $\bb$, so it must intersect $L$ and the claim follows. 

Now, assume that $L^v\subset \bb'$ is a vertical line so close  to  $\fr^v\bb'$
that for every horizontal line $L\subset \bb$,  $f(L)$ is transverse to $L^v$. 
  Let $V = f(L)\cap \Omega$ and $k =\deg(V)$. Then 
  $V\cap L^v$ is made of $k$ distinct points $a_1(L), \ldots , a_k(L)$, and since there are no tangencies between $V$ and $L^v$, these points can be followed holomorphically with $L$.   
  If $a_i(L_1) = a_j(L_2)$, 
 we must  have $L_1 = L_2$ because $f$ is injective, and then $i = j$ because there are no collisions 
 between the $a_i(L)$. Thus, 
 when $L$ ranges over the set of horizontal lines in $\bb$ (which is a disk $D_2$, each $a_i(L)$ describes a 
 topological disk $\Delta_i$, and these disks are disjoint.
 In addition, if $\Delta$ is a component of $\Omega\cap L^v$, then $f\inv(\Delta)$ is a vertical submanifold in $\bb$, so it intersects $L$. 
 This shows that $\Omega\cap L^v = \bigcup_{i=1}^k \Delta _i$ and   
that $k$ is independent of $L$.

Since $V$ is a union of components of $f(L)\cap \bb'$, by Property~(ii) of Definition~\ref{defi:non_ramified} it is  a union of holomorphic disks. 
Let us show that $V$ is connected.  
Pick a component   $V'$ of $V$. 
The preimage $f\inv(\fr V')$ is a Jordan curve in $L$, and since $f(L)$ is transverse to $\fr^v\bb$ for every $L$, this Jordan curve can be followed holomorphically with $L$, cutting out 
 a connected vertical open subset $U$  that is a component of $f\inv(\bb')\cap \bb$. Therefore 
$f(U)$ is a component of $f(\bb)\cap \bb'$ contained in $\Omega$, hence  $f(U) = \Omega$. 
This discussion shows that $f\inv(\Omega)\cap L = f\inv(V')$, but 
$f\inv(\Omega)\cap L = f\inv(V)$ by definition 
hence $V=V'$ and $V$ is connected. 
 
If we vary $L^v\subset \bb'$ by turning around $\fr^v\bb'$, the points $ {a_1(L), \ldots , a_{k}(L)}$ are permuted by holonomy. 
Since  $V$ is a topological disk, it has a unique boundary component, 
 so the holonomy acts transitively on the $a_i$ (resp. on the disks $\Delta_i$). 
Let now $W$ be an arbitrary horizontal disk contained in $\Omega$ 
(which may have vertical tangencies near the boundary). 
 When varying $L^v$ in a small neighborhood $\mathcal N$ of $\fr^v\bb'$  in $\bb'$ (\footnote{We   say that $\mathcal N$ is a neighborhood of $\fr^v\bb$ in $\bb$ if $\mathcal N =\widetilde {\mathcal N}\cap \bb$ for some neighborhood of $\fr^v\bb$ in $\C^2$. })
 the disks $\Delta_i$ form a topological fibration of $\mathcal N\cap \Omega$ over $ \mathcal N\cap V$. 
 The number of points in $W\cap \Delta_i$, counted with multiplicity, is locally constant as $L^v$ varies, 
 and since the holonomy is transitive,  this number does not depend on $i$. 
Denoting it by $\ell$,  we conclude that $\deg(W) = k\ell$, and the proof is complete. 
\end{proof}

\subsubsection{Iteration vs. cut-off iteration}\label{par:iteration_henon_like}
When considering a Hénon-like map $f\colon \bb\to \bb$, it is understood in~\cite{henonl}  that 
iteration means cut-off iteration, that is, 
 the only orbits that we consider are  those that never leave $\bb$. In other words, we  only iterate $f\rest{f\inv(\bb)}$ (often denoted $f\rest{\bb}$). 
 Now suppose that the Hénon-like map $f$ is the restriction of a global holomorphic diffeomorphism $\tilde f$ of~$\C^2$. Then, some orbits of $\tilde f$ might leave $\bb$ before entering it again. 
 We say that 
 $f$ \emph{is iterated properly} if ${\tilde f}^k({\tilde f}(\bb)\setminus \bb) \cap \bb= \emptyset$  for every $k\geq 0$. This implies the corresponding property for negative iterates. For instance, if $\tilde f$ is a Hénon map and $\bb = \disk(0, R)^2$   is a sufficiently large bidisk,  orbits  leaving $\bb$ must escape to infinity, and this property holds. 
 
Likewise, if $f_1$ and $f_2$ are two Hénon-like maps in $\bb$, which are restrictions of globally defined mappings $\tilde f_i$, their composition in the sense of Hénon-like maps is $f_2\rest{\bb}\circ f_1\rest{\bb}$, which  may differ from $(\tilde f_2\circ \tilde f_1)\rest{\bb}$. 
This happens precisely when $\tilde f_2(\tilde f_1(\bb)\setminus\bb)\cap \bb\neq \emptyset$. 
As above, we say that $f_2$ and $f_1$ \emph{are composed properly} if  
$\tilde f_2(\tilde f_1(\bb)\setminus\bb)\cap \bb= \emptyset$, in which case $f_2\rest{\bb}\circ f_1\rest{\bb}=(\tilde f_2\circ \tilde f_1)\rest{\bb}$. 
This definition can be generalized to compositions of arbitrary length. 

In what follows, we will use the same notation for $f$ and its extension $\tilde f$, when such an extension exists.

\subsection{Estimates for polynomials in one variable}\label{subs:polynomials}
Recall that  $\mathcal P_d = \mathcal P_d(\C)$ is the complex variety of conjugacy classes of polynomials of degree $d$ in one  complex variable under the action of the affine group. 
We   view it as   the quotient of the space 
$\mathcal P_d^{\rm cm}\simeq \C^{d-1}$ of monic and centered polynomials $p(z)=z^d+ \sum_{j=0}^{d-2} a_j z^j$ by the group of $(d-1)^{\rm th}$ roots of unity, acting by 
 $p\mapsto \beta\inv p(\beta z)$. 
For us, most of the time,  it will be enough to work on $\mathcal P_d^{\rm cm}$. 

\subsubsection{The maximal escape rate} Recall from Equation~\eqref{eq:maximal_escape_rate_intro} that  the  {maximal escape rate} is defined by
$
 M(p)= \max\set{G_p(c) \; ;\; p'(c)=0}.
$
We will be interested in the regime where  $M(p)$ is large.
As explained in~\cite{branner-hubbard1},  the asymptotic geometry of $p$ depends on $M(p)$. To see this, 
we define 
\begin{equation}
R_p =  e^{M(p)};
\end{equation}
as we will see, it provides a good measure of the size (or norm) of $p\in \mathcal P_d^{\rm cm}$ and of the characteristic scale of its Julia set.
For instance, the sublevel sets $\set{G_p<t}$ are connected precisely for $t>M(p)$.

\subsubsection{The Böttcher coordinate} 
The Böttcher coordinate $\varphi_p$ is defined on the open set 
\begin{equation}
{\mathcal{U}}_p=\set{z\in \C; \; G_p(z)> M(p)}
\end{equation}
and provides a holomorphic diffeomorphism from ${\mathcal{U}}_p$ to $\C\setminus\overline{\disk_{R_p}}$
such that 
\begin{equation}
G_p(z)=\log\abs{\varphi_p(z)}\quad \forall z\in {\mathcal{U}}_p.
\end{equation}  
The reciprocal diffeomorphism will be denoted $\psi_p\colon \C\setminus\overline{\disk_{R_p}}\to {\mathcal{U}}_p$. 
Both $\varphi_p$ and $\psi_p$ are tangent to the identity at infinity: $\varphi_p(z) = z+ O(1/z)$ 
and $\psi_p(w) = w+O(1/w)$.

Recall the distorsion estimates of Koebe: {\textit{if $g\colon \disk_\rho \to \C$ is a univalent function such that $g(0)=0$ and $g'(0)=1$, then 
\begin{equation}
\frac{\abs{u}}{(1+\abs{u}/\rho)^2} \leq \abs{g(u)} \leq 
\frac{\abs{u}}{(1-\abs{u}/\rho)^2}.
\end{equation}
}}
Applying these estimates to $u\mapsto 1/\psi_p(1/u)$ and making
 the change of variable $\psi_p(w)=z$, we obtain
\begin{equation}
\frac{1}{(1+R_p/\abs{\varphi_p(z)})^2}
\leq
\frac{\abs{\varphi_p(z)}}{\abs{z}}
\leq
\frac{1}{(1-R_p/\abs{\varphi_p(z)})^2}.
\end{equation}
This proves the following lemma. 

\begin{lem}
\label{lem:basic_estimates_bottcher}
The Böttcher coordinate $\varphi_p\colon {\mathcal{U}}_p\to \C\setminus{\overline{\disk_{R_p}}}$ satisfies 
\[
\abs{z}-3R_p \leq \abs{z}-2R_p -\frac{R_p^2}{\abs{\varphi_p(z)}} \leq \abs{\varphi_p(z)} 
\leq  \abs{z} + 2R_p -\frac{R_p^2}{\abs{\varphi_p(z)}} \leq \abs{z} + 2R_p.
\]
In particular, for $r>R_p$,
the curve 
\[
S_r=\set{G_p=\log(r)} = \set{\abs{\varphi}=r}=\set{\psi(\partial\disk_r)}
\]
is contained in $\set{r-2R_p\leq \abs{z}\leq r+3R_p}$, and this set is an annulus when $r>2R_p$.
\end{lem}

Consequently,  $\set{G_p<M(p)}$ is 
contained in $\overline \disk_{4R_p}$; in fact, Proposition~3.5 of~\cite{branner-hubbard1} gives
\begin{equation}\label{eq:bh-inclusion}
\set{G_p<M(p)}\subset \overline \disk_{2R_p}.
\end{equation}

\subsubsection{Estimates on the coefficients} 
If
$
p(z)=z^d+\sum_{j=0}^{d-2} a_jz^j
$ is a monic centered polynomial of degree $d$, 
its derivative $p'$ is also centered and can be written 
$p'(z)=d(z-c_1)\cdots (z-c_{d-1})$,
where the critical points $c_i$ of $p$ satisfy $\sum_i c_i=0$. For $j< d$ we obtain
\begin{align}
a_{d-j} &= (-1)^j\frac{d}{d-j}\sigma_j(c_1, \ldots, c_{d-1})  
\end{align} 
where $\sigma_j(c_1, \ldots, c_{d-1})=\sum_{1\leq i_1 < \ldots < i_j \leq d-1} c_{i_1}\cdots c_{i_j}$ is the $j$-th symmetric function, while $a_0\in \C$ is not determined by the $c_i$. This gives the estimate 
\begin{equation}
\label{eq:a_j_from_c_j}
\abs{a_{d-j}}\leq {d \choose j} \max\set{\abs{c_1}, \ldots, \abs{c_k}}^j
\end{equation}
for $j< d$.
From the inclusion~\eqref{eq:bh-inclusion}, 
we immediately get 
\begin{equation}
\label{eq:cj_vs_Rp}\abs{c_j}\leq 2 R_p, 
\end{equation}
which  gives $\abs{a_{d-j}}\leq {d \choose j} (2R_p)^{j}$ for $j< d$, and 
Proposition~3.6 of~\cite{branner-hubbard1} asserts that 
\begin{equation}
\label{eq:aj_vs_Rp}\abs{a_0}\leq 2 (4R_p)^d.
\end{equation}
 Moreover, on $\mathcal P_d^{\rm cm}$ the functions $M\colon p\mapsto M(p)$ and $R\colon p\mapsto R_p$ are  continuous. Thus, $M$ and $R$ are continuous and proper: a subset $\set{ p_s; \; s\in S }$ of $\mathcal P_d^{\rm cm}$ is bounded if and only if the set $\set{ M(p_s); \; s\in S}$ is bounded. 

Set 
$A=\max\{\abs{c_1}, \ldots, \abs{c_{d-1}}, \abs{a_0}^{1/d}\}$. By~\eqref{eq:a_j_from_c_j}, 
$\abs{p(z)}\leq 2^d\abs{z}^d$ as soon as $\abs{z}\geq A$ and, by the maximum principle, $\abs{p(z)}\leq 2^d A^d$ for $\abs{z}\leq A$. By induction, we obtain
\begin{equation}
\frac{1}{d^n}\log^+\abs{p^n(z)} \leq (1+ \frac{1}{d}+ \cdots +\frac{1}{d^{n-1}})\log(2) + \log^+\abs{z}
\end{equation}
when none of the $\abs{z}$, $\ldots$, $\abs{p^{n-1}(z)}$ is smaller than $ A$, and 
\begin{equation}
\frac{1}{d^n}\log^+\abs{p^n(z)} \leq (1+ \frac{1}{d}+ \cdots +\frac{1}{d^{n-1}})\log(2) + \log^+(A)
\end{equation}
otherwise. Thus, $G_p(z)\leq \frac{d}{d-1}\log(2)+ \log^+\max\{A,\abs{z}\}$. This implies that 
$M(p)\leq \frac{d}{d-1}\log(2)+ \log^+(A)$.
With the inequalities~\eqref{eq:cj_vs_Rp} and~\eqref{eq:aj_vs_Rp}, we obtain
\begin{equation}
\label{eq:comparaison_Mp_cj}
M(p)-2\log(2) \leq  \log^+\max\{\abs{c_1}, \ldots, \abs{c_{d-1}}, \abs{a_0}^{1/d}\} \leq M(p)+ 3\log(2)  .
\end{equation}

\subsubsection{Asymptotic estimates}
\begin{vlongue}
\begin{lem}
If $\abs{z}\geq R_p$ (resp. $\geq 2R_p$), then  
\[
\abs{ \frac{p(z)}{z^d} -1 } \leq 4^{d+1} \frac{R_p^2}{\abs{z}^2} \quad \left(\text{resp. }\;  \leq 2^{d+4}  \frac{R_p^2}{\abs{z}^2} \right).
\]
\end{lem}
\begin{proof}
Write $p(z)=z^d+a_{d-2}z^{d-2} +\cdots + a_1z+a_0$ and use the above estimates to get 
\[
\abs{ \frac{p(z)}{z^d} -1 } \leq {d \choose 2} \left(\frac{2R_p}{\abs{z}}\right)^2 +\cdots +  {d \choose d-1} \left(\frac{2R_p}{\abs{z}}\right)^{d-1} + 2\left(\frac{4R_p}{\abs{z}}\right)^{d}. 
\]
If $\abs{z}\geq R_p$ we obtain 
\[
\abs{ \frac{p(z)}{z^d} -1 } \leq \left(\frac{R_p}{\abs{z}}\right)^2 \left( (1+2)^d + 2\cdot 4^d\right) \leq 4^{d+1} \left(\frac{R_p}{\abs{z}}\right)^2.
\]
The case $\abs{z}\geq 2R_p$ is similar.  
\end{proof}
Here is a sharper upper bound, using another strategy of proof. 
\end{vlongue}

\begin{lem}\label{lem:basic_estimate_on_p/zd}~ 
\begin{enumerate}[\rm (1)]
\item For $\abs{z}\geq 6R_p$, we have the estimate 
$\displaystyle \abs{\frac{\varphi_p(z)}{z}-1}\leq \frac{6R_p^2}{{\abs{z}^2}}$. 
\item  For $\abs{z}\geq 6\sqrt{d-1} R_p$,  we have  
$\displaystyle
\abs{ \frac{p(z)}{z^d} -1 } \leq (12d+2)\frac{R^2}{\abs{z}^2}$.
\end{enumerate}
\end{lem}
\begin{proof} Let us write $\varphi$  for $\varphi_p$, and 
$\psi$ for the inverse of $\varphi$.
From Proposition~4.1 of \cite{branner-hubbard1}, we get  $\abs{\psi(w)/w - 1} \leq 3 R_p^2/\abs{w}^2  $ for $\abs{w}>R_p$; thus, for  $z\in \mathcal{U}_p$ we have
\begin{equation} \label{eq:varphip1}
\abs{z-\varphi(z)} \leq 3\frac{R_p^2}{\abs{\varphi(z)}}.
\end{equation} 
If  $\abs{z}\geq 2R_p$ then $\abs{z}+2R_p\leq 2\abs{z}$ and $\abs{\varphi(z)}\leq 2\abs{z}$ by Lemma~\ref{lem:basic_estimates_bottcher}.
If  $\abs{z}\geq 6R_p$
then $\abs{z}\leq 2(\abs{z}-3R_p)$ hence  $\abs{z}\leq 2\abs{\varphi(z)}$, again by Lemma~\ref{lem:basic_estimates_bottcher}. 
Thus for $\abs{z}\geq 6R_p$ we have $\abs{ \varphi(z)-z}\leq 3R_p^2/\abs{\varphi(z)}\leq 6R_p^2/\abs{z}$, which
proves the first inequality.

Now, suppose that $\abs{z}\geq 6 \sqrt{d-1} R_p$, in particular $\abs{z}\geq 6R_p$. 
Note that $\log(2)>1/6$.
This implies that $\log((1+t)^{d-1})\leq (d-1) t \leq \log(2)$
as soon as $0\leq t \leq 6R_p^2/\abs{z}^2$. Thus, the derivative of $t\mapsto (1+t)^d$ is bounded from above by $2d$ on the interval $[0,6R_p^2/\abs{z}^2]$. This implies that 
\begin{equation}\label{eq:phid_zd}
\abs{\frac{\varphi(z)^d}{z^d}-1} \leq 12d\frac{R_p^2}{\abs{z}^2}
\end{equation}
when  $\abs{z}\geq 6 \sqrt{d-1} R_p$. Now, if we apply Equation~\eqref{eq:varphip1} to $p(z)$ and use $\varphi\circ p=\varphi^d$ we get 
\begin{equation}\label{eq:p_phid}
\abs{\frac{p(z)}{\varphi(z)^d}-1}\leq 3 \frac{R_p^2}{\abs{\varphi(z)}^{2d}}.
\end{equation}
 Put $a= \varphi(z)^d/z^d$ and $b   = p(z)/\varphi(z)^d$. From~\eqref{eq:phid_zd} 
we get $\vert \varphi(z)^d/{z}^d -1\vert\leq 2/3$ hence 
$ \abs{\varphi (z)}^d   \geq \abs{z}^d/3$. Plugging this into~\eqref{eq:p_phid} and using $\abs{z}\geq 6$ gives 
\begin{equation}
\abs{b-1} \leq \frac{3 R_p ^2}{\abs{\varphi(z)}^{2d}} \leq \frac{27 R_p^2}{\abs{z}^{2d}} \leq \frac{27}{\abs{z}^{2d-2}} \cdot \frac{R_p^2}{\abs{z}^2} \leq \frac{3}{4} \cdot \frac{R_p^2}{\abs{z}^2} . 
\end{equation}
Then, from the elementary inequality $\abs{ab-1}\leq \abs{a - 1}+  \abs{b-1}+\abs{a-1}\cdot \abs{b-1}$
we obtain 
\begin{align}
\abs{\frac{p(z)}{z^d}- 1 }&\leq 12d \frac{R_p^2}{\abs{z}^2} + \frac{3}{4} \cdot \frac{R_p^2}{\abs{z}^2} + 12d \frac{R_p^2}{\abs{z}^2} \cdot \frac{3}{4}\cdot \frac{1}{36(d-1)} \\ \notag
& \leq \lrpar{12d + \frac34 + \frac14\cdot \frac{d}{d-1}} \frac{R_p^2}{\abs{z}^2} \leq (12d+2) \frac{R_p^2}{\abs{z}^2},
\end{align}
as was to be shown. 

%
%
\end{proof}

\subsection{Compositions of Hénon maps and the quantity $M(f)$}\label{subs:M_for_loxodromic} 
Let us complete the discussion of \S~\ref{subs:fm}  by describing the lack of uniqueness in the Friedland-Milnor normal form. 
Let $\mathbf d =(d_k, \ldots, d_1)$ be a multidegree and set $d = \prod_{i=1}^k d_i$.
Let $g = h_k\circ \cdots \circ h_1$ be a composition of monic and centered Hénon maps $h_i(x,y) = (a_i y+ p_i(x), x)$ 
of respective degrees  $d_i=\deg(p_i)$.  

First, conjugating $g$ by  $h_j\circ\cdots \circ h_1$ permutes the factors $h_i$, 
which defines an action of $\Z/k\Z$ on normal forms by cyclic permutations. 
Second, the group $\mathbb{U}_{d-1}=\{\alpha\in \C^* \; ; \; \alpha^{d-1}=1\}$ acts on the normal form as follows.
For  $\alpha\in \mathbb{U}_{d-1}$, consider the finite sequence $(\alpha_i)_{1\leq i\leq k}$  defined  by $\alpha_1 = \alpha$ and 
$\alpha_{i+1} = \alpha_i^{d_i}$  (so that $\alpha_k^{d_k} = \alpha_1$),
and set $s_i(x,y) = (\alpha_ix , \alpha_{i-1}y)$ where the  index $i$ is taken modulo $k$. 
Then, define
\begin{align}\label{eq:alphaf}
\alpha\cdot g &:=  
  s_1 \inv\circ g \circ s_1 = \tilde h_k\circ \cdots \circ \tilde h_1, \text{ with }  \tilde h_i = s_{i+1}\inv \circ  h_i \circ s_i.
\end{align}
This is another composition of monic and centered Hénon maps that is conjugate to $g$. 
On the parameters $a_j\in \C^*$ and $p_j=x^{d_j}+\sum b_{j,m}x^m$, this action of $\mathbb{U}_{d-1}$ is given by the diagonal transformations
\begin{equation}
\label{eq:alpha_conjugation}
\alpha\cdot a_j  = (\alpha_{j+1}\inv\alpha_{j-1}) a_j \; \text{ and } \; 
\alpha\cdot p_j =x^{d_j}+\sum_{m=0}^{d_j-2}  (\alpha_{j+1}\inv\alpha_j^m)b_{j,m}x^m.
\end{equation}
Altogether, we obtain an action of $\Z/k\Z\ltimes   \mathbb{U}_{d-1}$ on $\mathcal H^k_{\mathbf d}$. 
Friedland and Milnor proved that an element $g'$ of $\mathcal H^k_{\mathbf d}\simeq (\C^*)^k\times   \C^{d_1-1}\times \cdots \times \C^{d_k-1}$
is conjugate to $g$ if and only if $g'$ is in the orbit of $g$ under this action. 


Thus, the moduli space of polynomial automorphisms of $\C^2$ of multidegree 
$\mathbf d$ is the quotient of $\mathcal H^k_{\mathbf d}$ by $\Z/k\Z\ltimes   \mathbb{U}_{d-1}$. 
In conclusion, two conjugacy invariants are attached to a loxodromic automorphism $f$:
\begin{equation}
[\mathbf{d}(f)] \in \N_{\geq 2}^k/\Z/k\Z \;\, {\text{ and }} \;\, [\mathbf{a}(f)]\in (\C^*)^k / (\Z/k\Z\ltimes   \mathbb{U}_{d-1}).
\end{equation}
So, when talking about the multidegree or multi-Jacobian of an automorphism, there is an abuse of notation since this quantity should be considered in the above quotient space.

As in 
the introduction we define
\begin{equation}
M(g) = \max (M(p_i)), \quad \text{and} \quad R_g = e^{M(g)}
\end{equation}
The radius $R_g$  (sometimes abbreviated into $R$) will be the key geometric scale below. 

\begin{rem} {\emph{The quantity $M(g)$ is not a conjugacy invariant, 
is varies under the $\mathbb{U}_{d-1}$ action.}}
To see this, for  an arbitrary integer $e\geq 2$ and  $c\in \C$, let  $g_c=h'\circ h_c$, 
where $h'(x,y)=(y+x^e, x)$ and $h_c(x,y)=(y+x^2+c,x)$, 
so that  $d=2e$.  For $c\in \C$, set $p_c(z)=z^2+c$ and 
\[
G_{\mathcal M}(c)=\lim_{n\to +\infty} 2^{-n} \log^+\abs{p_c^n(c)} = G_{p_c}(c)=M(p_c),
\]
which is the Green function of the Mandelbrot set $\mathcal M$, so that 
 \begin{equation}
 M(g)=\max\{M(p_0), M(p_c)\}=M(p_c)
 \end{equation} 
 for every $c\in \C$. 
Now, fix $\alpha\in \mathbb{U}_{d-1}$ and 
suppose by contradiction that $M(\alpha\cdot g_c)=M(g_c)$ for     any $c$. 
Formula~\eqref{eq:alpha_conjugation} shows that 
 $\alpha\cdot g_c=h'\circ h_{\alpha^{-2}c}$, so  we obtain 
$G_{\mathcal M}(c)=G_{\mathcal M}(\alpha^{-2}c)$ for some   root of unity, 
hence $G_{\mathcal M}$ is invariant under a non-trivial rotation, and so is $\mathcal M$, which is the desired 
  contradiction. 
  \end{rem}

To get a function that depends only on the conjugacy class of $g$ in $\mathcal H^k_{\mathbf d}$, we set
 \begin{equation}\overline M(g) := \max(M(\alpha\cdot g), \ \alpha\in \mathbb U_{d-1}),\end{equation} where $\alpha\cdot g$ is as in~\eqref{eq:alphaf}. 
 Then, when $f\in \Aut(\C^2)$ is loxodromic   we define 
 \begin{equation}
 \overline M(f) = \overline M([f]) = \overline M(g),
 \end{equation} 
 where $g$ is any Friedland-Milnor normal form of $f$. The following lemma justifies 
  that this definition is well-behaved. 

\begin{lem}\label{lem:M}
Let $(f_n)$ be a sequence of elements of $\mathcal H_{\mathbf d}^k$.  Then
\begin{enumerate}[\rm (1)]
\item $M(f_n)\to \infty$ if and only if $\overline M(f_n)\to \infty$;  for fixed $\mathbf{a}$ this holds if and only if $f_n$ tends to infinity in $\mathcal H_{\mathbf d, \mathbf a}^k$; 

\item   $M(f_n)\sim \overline M(f_n)$ as $n\to \infty$;

\item if furthermore the multi-Jacobian of $f_n$ does not depend on $n$, then 
$\overline M(f_n)\to\infty$ if and only if $\overline M(f_n\inv )\to \infty$. 
\end{enumerate}
\end{lem}

The third assertion does not hold when the multi-Jacobian is not fixed:
indeed if $f(x,y) = (ay + x^2+c, x)$ then $f\inv$ is conjugate to 
$(x,y)\mapsto (a\inv y + x^2+ ca^{-2}, x)$, so if $c\to \infty$ and $\abs{a}^2 \gg \abs{c}$, then 
$M(f)\to\infty$ while 
$M(f\inv) \to 0$.  

\begin{proof} 
Set $\mathbf d = (d_k, \ldots, d_1)$ and pick $f=h_k\circ \ldots \circ h_1$ in $\mathcal H_{\mathbf d}^k$. 
Then, with notation as in Equation~\eqref{eq:alphaf}, we have  $\alpha\cdot f = \tilde h_k\circ \cdots\circ \tilde h_1$
with $\tilde h_i   (x,y) = (\tilde a_i y+\tilde p_i(x), x)$,  
where $\tilde a_i = \alpha\cdot a_i =\alpha_{i+1}\inv\alpha_{i-1} a_i$ and 
$\tilde p_i(x)   = \alpha\cdot p_i=\alpha_{i+1} \inv p_i(\alpha_ix)$ as in Equation~\eqref{eq:alpha_conjugation}. 
Note that $\tilde p_i$ is not conjugate to $p_i$ when 
 $\alpha_i$ is not a $(d_i-1)^{\rm th}$-root of unity.

For a sequence $(f_n)$, $M(f_n)$ is unbounded if and only if $M(p_{n,i})$ is unbounded for at least one of the factors $h_{n,i}$ of the decomposition $f_n=h_{n,k}\circ \cdots \circ h_{n,1}$. 
And $M(p_{n,i})$ is unbounded if and only if 
$M(\beta\inv p_{n,i}(\alpha x))$ is unbounded, for any   pair of complex numbers $(\alpha, \beta)$ of modulus $1$ . Thus, $M(f_n)$ is unbounded if and only if $\overline M(\alpha\cdot f_n)$
is unbounded. This proves the first assertion. 
 
For the second assertion, we can apply directly the Inequality~\eqref{eq:comparaison_Mp_cj}. 
 
 For the third assertion, we will show that the Friedland-Milnor normal form $g_n$ of 
  $f_n\inv$  has a    multi-Jacobian determined by that of $f_n$ (up to some roots of unity) 
  and  
$M(g_n)  \to \infty$. 
For this we write 
 $f\inv= h_1\inv \circ\cdots \circ h_k\inv$ with $h_i\inv(x,y)   =  ({y,  {a_i}\inv(x-p_i(y)} )$, so that
 the involution $\tau: (x,y) \mapsto  (y,x)$ conjugates 
 $f\inv$ to $  f^- =   h_1^-  \circ\cdots \circ   h_k^-$, 
 where $  h_i^- (x,y) =  ({a_i}\inv y - {a_i}\inv p(x), x  )$. The polynomials 
 $-p_i/a_i$ are not monic.   
 To get to a normal form, we follow the method of~\cite{friedland-milnor} and  
 conjugate  $  f^-$ to 
\begin{equation}
\ell_k\inv\circ f\circ \ell_k
= \ell_k\inv \circ \tilde  h_1 \circ  \ell_1 \circ  \cdots \circ h_{k-1} \circ  \ell_{k-1}  \rond \ell_{k-1}\inv  \circ  h_k\rond \ell_k 
\end{equation}
where $\ell_i(x,y)= (\gamma_i x, \delta_i y)$. 
The $\ell_i$ must be chosen so that $\ell_{i}\inv \tilde h_{i+1}\ell_{i+1}$ is monic and centered, which corresponds to the  system of equations: 
\begin{equation}
\begin{cases}
\delta_{k-1}  \inv \gamma_k = 1 &, \quad   -a_k\inv \gamma_k^{d_k}= \gamma_{k-1} \\
 \quad \vdots & \quad \quad \vdots  \\
\delta_k\inv \gamma_{1} = 1 &, \quad    -a_1\inv\gamma_{1}^{d_1} = \gamma_k 
\end{cases}
\end{equation}
So the $\delta_i$ are determined by the $\gamma_i$ and the 
$\gamma_i$ are determined by  $\mathbf a$, up to some choice of $(d-1)^{\rm th }$ root of unity.   
 In this way we obtain a new automorphism
$g = g_1\circ \cdots \circ  g_k$, with $g_i(x,y) = (b_i y+ q_i(x), x)$ with 
\begin{equation}
b_i = \gamma_{i-1}\inv\gamma_{i+1} a_i\inv \quad \text{ and } \quad 
q_i (x)= -\gamma_{i-1}a_i\inv p_i (\gamma_ix).
\end{equation}   
We conclude that $f\inv$ is conjugate to an element 
  $g\in \mathcal H_{(b_1, \ldots b_k), (d_1, \ldots , d_k)}$, 
  with fixed $(b_1, \ldots b_k)$ 
and that $M(f_n)\to \infty$ implies that $\overline M(g_n)\to \infty$
 \end{proof}

\section{Divergence of  Lyapunov exponents}\label{sec:lyapunov}

Our first goal in this section is 
  Proposition~\ref{pro:lyapunov_multi}, which gives 
precise asymptotic bounds for the Lyapunov exponents of compositions of Hénon maps.
We  first deal with  Hénon maps in \S~\ref{subs:lyapunov_henon} 
and use this for the general case in  \S~\ref{subs:lyapunov_multi}. 
 In \S~\ref{subs:comments_corollaries} we derive Theorem~\ref{mthm:lyapunov}, and then 
 Theorem~\ref{mthm:stable}, and complete the proofs of  Theorems~\ref{mthm:henon_multipliers}, \ref{mthm:multi-Jacobian} and~\ref{mthm:C1}. We also 
 discuss the compactness of the connectedness locus.

\subsection{The case of  Hénon maps}\label{subs:lyapunov_henon}

In this section we establish    the following precise version of 
Theorem~\ref{mthm:lyapunov} for complex Hénon maps.   

\begin{pro}\label{pro:lyapunov_henon} 
Let $f\colon (x,y)\mapsto (ay+p(x), x)$ with   $\abs{a}\leq     R_f^{d-1}$. Then 
\begin{equation}\label{eq:lyap_estimate_sup}
\chi^+(\mu_f)\leq  \log d +  d    M(f)  + d \log(15d).
 \end{equation}
If furthermore $\abs{a}\leq     R_f^{d-1}/(400d)$ and $M(f)$ is   sufficiently large then 
\begin{equation}\label{eq:lyap_estimate_inf}
  \chi^+(\mu_f)\geq     \log d +M(f) - \unsur{d} \log\frac43.
 \end{equation}
\end{pro}

 \subsubsection{Discussion}\label{subsub:comments}
The Manning-Przytycki formula, for the polynomial $p$ provides the estimate $\log d+M(p)\leq \chi(\mu_p)\leq \log d + (d-1)M(p)$, 
so it is natural to wonder whether the additional constants in the estimates are necessary 
(no real attempt  is made in the proof  to    optimize these  constants).

It will  be clear from the proof (\footnote{Indeed  in  Lemma~\ref{lem:folding} we can choose the horizontal radius of $\bb_2$ to be $o(R^d)$, which improves the lower bound on $G_f$ in~\eqref{eq:lower_lyap}.}) that if we strengthen the assumption on $a$  to 
 $\abs{a}\leq h(R^{d-1})$, with $h(t)  =o(t)$ for $t\to\infty$, the lower bound becomes  
 $\chi^+(\mu_f)\geq  \log  d+     M + o(1)$. 
 For the upper estimate, the question whether   we can achieve the bound 
$\chi^+(f)\leq  \log  d+     dM + o(1)$  is open, even   for bounded $a$. 
 A crude bound on how large must $M(f)$ be for  the estimate~\eqref{eq:lyap_estimate_inf} to hold is given in Remark~\ref{rem:bound_M}. 
 
 By Equation~\eqref{eq:comparaison_Mp_cj}
 the coefficients of $p$ are of order of magnitude  $O(R^d)$, 
 so the assumption on $\abs{a}$  means that the Jacobian is not too large with respect to   the coefficients of $p$. Actually, allowing large $\abs{a}$ in Proposition~\ref{pro:lyapunov_henon} is not really meaningful: 
we only do it because it will be useful later in Proposition~\ref{pro:lyapunov_multi}. Indeed, 
 a Hénon map with large Jacobian is conjugate to the inverse of a Hénon map with small Jacobian, which together with the formula   $\chi^+(\mu_f) + \chi^-(\mu_f)  = \log \abs{a}$ 
 leads to an   estimate  for $\chi^+(\mu_f)$.

\begin{vlongue}
Let us develop this idea. 
First, the relation $\chi^+(\mu_f) + \chi^-(\mu_f)  = \log \abs{a}$ gives  $\chi^+(\mu_f)\geq \log\abs{a}$. 
If  $\abs{a}\geq c R^\beta$ with $\beta>1$, this gives $\chi^+(\mu_f)\geq \beta\log M + O(1)$ which is better than the estimate~\eqref{eq:lyap_estimate_inf}. 

Next, writing  
$\chi^+(\mu_f)  = \log\abs{a}  + \abs{\chi^-(\mu_f)}   = \log\abs{a}  + \chi^+(\mu_{f\inv})$ 
reduces the problem to finding good estimates for the Lyapunov exponent $\chi^+(\mu_{f\inv})$ of $f^{-1}$ 
with respect to $\mu_{f\inv} = \mu_f$.
But $f\inv(x,y) = (y, \unsur{a}(x-p(y)))$ is conjugate to the normal form 
$(\unsur{a} y  + q(x), x)$,  where $q(x) =  - \unsur{a\alpha}p(\alpha x)$ for $\alpha^{d-1} = - a$. 
This automorphism $(x,y)\mapsto (\unsur{a} y  + q(x), x)$ is a small, invertible perturbation of the one-variable map $x\mapsto q(x)$. 
As such, it is eligible to the techniques of \cite[\S 3]{lyapunov}. 
In this paper $q$ is considered  to be fixed, 
but  the argument can easily be adapted to the case of a variable polynomial $q$
and yields $\chi^+(f) = \log\abs{a} + \chi^+(q) +o(1)$. 
Depending on the order of magnitude of $\abs{a}$, $q$ diverges in $\mathcal P_d$ or not. 
For instance, suppose $\abs{a}\geq c R_p^\beta$ and $\beta$ is large enough  ($\beta >(d-1)^2/2$ suffices).  Then from the bounds~\eqref{eq:a_j_from_c_j} and~\eqref{eq:comparaison_Mp_cj} 
it is easy to see that $q(x)\to x^d$ as $R\to \infty$. Then Theorem 1.3 in \cite{lyapunov} implies that $\chi^+(\mu_{f\inv})\to \log d$, hence 
$\chi^+(\mu_f)  = \log\abs{a}  +\log d+o(1)$. 

With these ingredients in hand, we can hope for a complete description of the behavior of $\chi^\pm(\mu_f)$ at infinity in $\mathcal H^1_d$. 
\end{vlongue}

\subsubsection{Estimates on $G_f$ and $\varphi_f$}\label{subsub:estimates_varphi}
Since  in the following we only work with the forward Green  and Böttcher functions $G^+_f$ 
and $\varphi^+_f$, for notational lightness we remove the superscript $+$.
It is well known that when $f$ is fixed, $\abs{x}$ is large, and $\abs{y}\leq \abs{x}$, then $G_f(x,y) = \log \abs{x} + O(1)$ and 
$\varphi_f(x,y) = x+O(1)$. 
We need to understand how large $x$ must  be  to get such an estimate when $f$ varies in the space of Hénon maps. 
 Under an appropriate  bound on the Jacobian, it 
  turns out that the right scale is determined by $R_f$.

\begin{pro}\label{pro:estimate_varphi}
 Suppose
  $\abs{a}\leq R_f^{d-1 }$. Then $\varphi_f$ is well-defined in the domain defined by  $\abs{y}\leq \abs{x}$ and $\abs{x}\geq   
12dR_f$,  and in this domain we have  
$$\abs{\frac{\varphi_f(x,y)}{x} - 1} <  2 \frac{R_f}{\abs{x}} \quad \text{ and }  \quad \abs{G_f(x,y) - \log \abs{x}}<    2 \frac{R_f}{\abs{x}} .$$
\end{pro}

\begin{proof}
 Denote $(x_n, y_n) = f^n(x,y)$. We come back to the construction of $\varphi_f$:  recall that   
  $\varphi_f(x,y)= \lim_{n\to\infty}  (\pi_1\circ f^n(x,y))^{1/d^n} =\lim_{n\to\infty}   x_n^{1/d^n}$ as $n$ goes to $+\infty$, for an appropriate choice of $(d^n)^{\rm th}$-root. For later use in Proposition~\ref{pro:estimate_varphi_multi} we split the proof in two steps. 
 
\noindent{\bf Step 1.--} Estimation of $x_k$.   

Consider
 \begin{equation}\label{eq:deftheta}
\theta(x,y) = \frac{x_1}{x^d} - 1 = \frac{\pi_1\circ f(x,y)}{x^d} - 1 = \frac{p(x)+ay}{x^d} -1.
\end{equation}
Since $\abs{x}\geq 12dR_f$ and $\abs{a}\leq R_f^{d-1}$, Lemma~\ref{lem:basic_estimate_on_p/zd} implies that 
 \begin{align}\label{eq:theta}
\abs{\theta(x,y)}  &\leq \lrpar{12d+2}\frac{R_f^{2}}{\abs{x}^2} + \frac{R_f^{d-1}}{\abs{x}^{d-1}} \\
&\notag \leq 
\lrpar{(12d+2) \frac{R_f}{\abs{x}}+1}\frac{R_f}{\abs{x}} 
\notag \leq
 \lrpar{  2+\frac{1}{6d}  }\frac{R_f}{\abs{x}}.
\end{align}
In particular, $\abs{\theta(x,y)}\leq 1/11$, so that 
\begin{equation}\label{eq:10d}
\abs{x_1}\geq (10/11) \abs{x}^d\geq 10d \abs{x} \geq 20 \abs{x}.\end{equation} 
This implies that the domain defined by $\abs{y}\leq \abs{x}$ and $\abs{x}\geq 12d R_f$ is stable, 
that $\abs{x_k}\geq (10d)^k   \abs{x}\geq 20^k \abs{x}$,  and that $\abs{\theta(x_k,y_k)}\leq 1/11$ for every $k\geq 0$.

\noindent{\bf Step 2.--} Estimation of $\varphi_f$.  From now on we only use the bound    $\abs{\theta(x,y)}\leq 1/6$. 

Set $\varphi_0(x,y) = x$ and for $n\geq 1$,
 \begin{equation}\label{eq:varphin}
 \varphi_n(x,y) = x \prod_{k=0}^{n-1} \lrpar{1+ \theta( f^k (x,y))}^{1/d^{k+1}} 
 \end{equation}
where $(\cdot)^{1/d^n}$ is the principal branch of the $(d^n)$-th root. 
 This makes sense because 
$\abs{\theta( f^k (x,y))}\leq 1/6$ for all $k\geq 0$.   
Therefore the infinite product converges and $\varphi_f (x,y) = \lim_{n\to \infty} \varphi_n(x,y) $ is well-defined. 
To estimate this product we will use the following elementary inequalities, where $\alpha = 1.1$
\begin{align}
\label{eq:elementary_inequalityI}
\abs{\exp(v) - 1} & \leq \alpha \abs{v} \text{ for } \abs{v}\leq 1/6, \\ 
\label{eq:elementary_inequalityII}\abs{\log(1+u)} & \leq \alpha \abs{u} \text{ for } \abs{u}\leq 1/6. 
\end{align}
(To prove these upper bounds, one simply expands $\exp(v)$ and $\log(1-u)$ in power series). Let us write
\[
\lrpar{1+ \theta( f^k (x,y))}^{1/d^{k+1}} = 1+u_k.
\]
Then, the Inequalities~\eqref{eq:elementary_inequalityI} and~\eqref{eq:elementary_inequalityII} give
\begin{align}
\abs{u_k} = \abs{\exp\lrpar{\unsur{d^{k+1}}\log\lrpar{1+ \theta( f^k (x,y))}} - 1}  \leq 
\frac{\alpha^2}{d^{k+1}}  \abs{\theta\circ f^k(x,y)};
\end{align}
and from Equation~\eqref{eq:theta} and the lower bound $\abs{x_k}\geq 20^k\abs{x}$, we derive
\begin{align}
\abs{u_k}  \leq \frac{\alpha^2}{d} \lrpar{2 +\frac{1}{6d}} \frac{1}{(20d)^k}\frac{R_f}{\abs{x} }.
\end{align} 
Now, the modulus of $\sum_{k=0}^{n-1} \log(1+u_k)$ is bounded from above by $\alpha \sum_{k=0}^{n-1} \abs{u_k}$, and from the last upper bound this is less than $1/6$.
Thus,   inequality~\eqref{eq:elementary_inequalityI}  and the definition of $\varphi_n$ in~\eqref{eq:varphin} give
\begin{align}
\label{eq:varphin2}
\abs{\frac{\varphi_n(x,y)}{ x}-1} &=
\abs{\exp\lrpar{\sum_{k=0}^{n-1} \log(1+u_k)} - 1} \leq \alpha^2 \sum_{k=0}^{n-1} \abs{u_k}\\ 
&\label{eq:varphin3}
\leq \frac{\alpha^4}{d}\lrpar{2 +\frac{1}{6d}} \sum_{k=0}^{n-1} \lrpar{ \frac{1}{(20d)^k}  } \frac{R_f}{\abs{x}} \leq \frac{8}{5} \frac{R_f}{\abs{x}}. \end{align}
The desired bound on $\varphi_f$ follows by letting $n$ go to $\infty$. 
 The corresponding estimate for $G_f$ is immediate by writing 
 \begin{equation}\label{eq:varphin4}
 \abs{G_f(x,y)  - \log\abs{x}} = \abs{\log \abs{\frac{\varphi_f(x,y)}{x}}}\leq  { \max (\abs{\log(1+u)} ,  \abs{\log(1-u)})}  ,
 \end{equation}
 where $u =  \abs{ {\varphi_f(x,y)}/{x} - 1}$   is bounded by $(8/5)(1/12d) \leq 1/15$
 by~\eqref{eq:varphin3}, hence  
 $\max (\abs{\log(1\pm u)}) \leq \alpha {u}$, and the inequality $8\alpha/5 <2$ completes the proof. 
\end{proof}

\begin{pro}\label{pro:slope}
Keep  notation as above, and suppose as before  that $\abs{a}\leq R_f^{d-1 }$. 
Let  
\[
V^+ = \set{(x,y)\; ; \;   \abs{x}> \max(\abs{y}, 12dR_f)+4R_f}.
\]
Then for every 
$(x_0, y_0)\in V^+$, the subvariety $\set{\varphi_f(x,y) = \varphi_f(x_0,y_0)}$ is a vertical graph in $V^+$, whose slope over the disk
$\set{\abs{y}\leq \abs{x_0}/2}$
is bounded by 
$\frac{5R_f}{\abs{x_0}}$.
Furthermore, $\Phi: (x,y)\in V^+\mapsto (\varphi_f(x,y), y)$ is a biholomorphism onto its image, which contains 
$\set{(x,y) \; ; \;   \abs{x}> \max(\abs{y}, 12dR_f)+6R_f}$.
\end{pro}

\begin{proof}
By Proposition~\ref{pro:estimate_varphi}, we have 
$\abs{\varphi_f(x,y)- x}< 2R_f$ when $\abs{x}\geq \max(\abs{y}, 12dR_f)$.  
Let $(x_0, y_0)$ be an element of $V^+$ and set $x^\varstar =  \varphi_f(x_0,y_0)$, 
which satisfies $\abs{x^\varstar} \geq \abs{x_0} - 2R_f > (12d+2)R_f$. 
 Set $R^\varstar=  \abs{x^\varstar}-2R_f$.
The disk $\disk(0,R^\varstar)$
contains $\disk(0, 12dR_f)$.
Fix $y_1$ in $\disk(0, R^\varstar)$  
and consider the function 
$z\mapsto \varphi_f(z, y_1) - z$ for  
$z\in \disk(x^\varstar, 2R_f)$.
For such a $z$, we have $\abs{z}\geq \abs{y_1}$ and $\abs{z}\geq 12dR_f$, so that Proposition~\ref{pro:estimate_varphi} can be applied to $(z,y_1)$. 
 The functions $h(z)=\varphi_f(z, y_1) - x^\varstar$ and $g(z)=z - x^\varstar$ are holomorphic on $\disk(x^\varstar, 2R_f)$ 
 and on $\partial\disk(x^\varstar, 2R_f)$ they satisfy  $\abs{h-g}=\abs{\varphi_f(z, y_1) - z}<2R_f=\abs{g(z)}$.
 By Rouché's theorem, there exists a unique $x_1\in \disk(x^\varstar, 2R_f)$
 such that  $\varphi_f(x_1,y_1) = x^\varstar$. Hence 
 $\set{\varphi_f(x,y) = \varphi_f(x_0,y_0)}$ is a vertical graph $y\mapsto (\gamma(y),y)$ for some holomorphic function $\gamma$
 defined on $\disk(0, R^\varstar)$ such that $\gamma-x^\varstar$ takes its
 values in $\disk(0, 2R_f)$.  The Schwarz-Pick lemma guarantees that for $y\in \disk(0, (2/3)R^\varstar)$,
 \[
\abs{\gamma'(y)} \leq \frac{1}{1-(2/3)^2} \frac{2R_f}{R^\varstar} = \frac{18}{5}  \frac{R_f}{R^\varstar}. 
\] 
From $R^\varstar \geq \abs{x_0}- 4R_f$, we obtain $R^\varstar\geq (6/7)\abs{x_0}$, 
hence also $(2/3)R^\varstar\geq (1/2)\abs{x_0}$. 
Since $(18/5)(7/6)< 5$, we get
$\abs{\gamma'(y)} < 5 \frac{R_f}{\abs{x_0}}$ for $y$ in $\disk(0, \abs{x_0}/2)$.

 The same argument says that $ \varphi_f(\cdot, y_1)$ is injective 
 for $\abs{x}> \max(\abs{y}, 12dR_f) + 4R_f$, which together with the estimate of 
 Proposition~\ref{pro:estimate_varphi}  implies the last assertion.
\end{proof}

\subsubsection{Crossed mappings}\label{subsub:crossed} 
Let $\bb_1 = \disk(0, 10 R_f)^2$.  
It is an easy consequence of Lemma~\ref{lem:basic_estimates_bottcher}
that for $\abs{a}\leq R_f^{d-1}$,  $f$ is a Hénon-like map in $\bb_1$. Here is a  sharper result:

\begin{lem}\label{lem:crossed}
With notation as above, if $\abs{a}\leq R_f^{d-1}$ then $f$ is a crossed map of degree $d$ from 
$\bb_1$ to $\bb_3:= \disk(0, 6^dR_f^d)\times \disk(0, 10R_f)$. 
\end{lem}

\begin{rem}
\label{rem:largerR}
As the proof will show, the conclusion of the lemma hold if $R_f$ is replaced by any larger radius $R$. This will be used in Proposition~\ref{pro:lyapunov_multi}.
\end{rem}

\begin{proof}
We have to show that  
 (i) $f(\fr^v\bb_1)\cap \overline{\bb_3} = \emptyset$, 
(ii) $f(\bb_1) \cap \fr\bb_3\subset \fr^v(\bb_3)$, 
and (iii)  $f(\bb_1)\cap  {\bb_3} \neq\emptyset$. Recall that  $R_f = R_p$.

For the first condition, we observe that $(x,y)\in \fr^v\bb_1$ means 
$\abs{y}\leq 10R_f=\abs{x}$. 
The lower bound of Lemma~\ref{lem:basic_estimates_bottcher} gives  
$\abs{\varphi_p(x)}\geq 7 R_f$, so  $\abs{\varphi_p(p(x))}\geq (7R_f)^d$; then, the upper bound gives
$\abs{p(x)}\geq (7R_f)^d- 2R_f $. Therefore 
\begin{equation}
\abs{p(x) +ay}\geq (7R_f)^d - 2R_f - 10R_f^d \geq (7^d- 12)R_f^d > 6^dR_f^d.
\end{equation} 

To establish property (ii), suppose $(x,y)\in \bb_1$ is such that  $f(x,y) \in \fr\bb_3$. Property (i) implies $\abs{x}< 10R_f$, 
so that  $\abs{ \pi_2\circ f (x,y) }=\abs{x}< 10R_f$, that is, $f(x,y)\in \fr^v\bb_3$. 

For  property (iii), consider the horizontal line $L=\{(x,0)\; ; \;  \abs{x}\leq 10 R_f\}\subset \bb_1$.
Its image by $f$ is parametrized by $x\mapsto (p(x),x)$, and the estimate obtained for $p(x)$ above in proving~(i) shows that 
the first projection $x\mapsto p(x)$  realizes a branched covering of degree $d$ from $\disk(0, 10R_f)$ 
to a domain containing $\disk(0, 6^d R_f^d)$. This proves~(iii) and the fact that the crossed mapping 
 $f$ is of degree $d$. 
\end{proof}

 \begin{proof}[Proof of the upper estimate~\eqref{eq:lyap_estimate_sup}]
 It is shown in~\cite[Thm A.2]{lyapunov} that if $\varphi_f$ is well defined on a neighborhood of $K^-\cap \set{G_f>t}$, then $\chi^+(\mu_f)\leq   \log d +  dt$ (the infimum of such $t$ is by definition the fastest escape rate $G_{\mathrm{max}}(f)$). 
 We will show that for $\abs{a}\leq R_f^{d-1}$, this is true for  $t= M(f)+\log(15d)$. 
 
Set $\bb_0=\disk(0, 12dR_f)^2$ and $V_0^+ = \set{(x,y)\; ; \;  \abs{y}\leq \abs{x}, \ \abs{x}\geq 12dR_f}$. 
Lemma~\ref{lem:crossed} and Remark~\ref{rem:largerR} 
show that $f$ is a Hénon-like map of degree $d$ in $\bb_0$. 
In particular, $K^-\cap \bb_0$ is horizontally contained in this bidisk.  Furthermore, since 
$\pi_2\circ f(x,y) = x$, we see that $f(\bb_0)\subset \bb_0\cup V_0^+$, that is 
every point escaping $\bb_0$ escapes through $V_0^+$. 
So if $(x,y)\notin \bb_0\cup V_0^+$, we cannot have $f^{-n}(x,y)\in \bb_0$, which 
  implies that $K^-\subset  \bb_0\cup V_0^+$. 

Proposition~\ref{pro:estimate_varphi} gives $G_f\leq \log(12dR_f)+1/(6d)$
 along the vertical boundary of $\bb_0$. 
Applying the maximum principle on horizontal lines, we obtain the same estimate in $\bb_0$. 
We deduce that if $(x,y)$ belongs to  
 $K^- \cap \set{G_f> \log(12dR_f)+1/(6d)}$, then 
 it belongs to $V_0^+$, so 
 Proposition~\ref{pro:estimate_varphi} applies and $\varphi_f$ is well defined near $(x,y)$.  
 Thus, the inequality $\log(12)+1/(6d)\leq \log (15)$ concludes the proof. 
\end{proof}

\subsubsection{Folding}
\label{subsub:folding}
We now study how a dominant critical point of $p$  folds $f(\bb_1)$ in $\bb_3$.  
Let $c$ be a critical point of $p$ with $G_p(c) =M(p)$ and let $v= p(c)$ be the corresponding critical value. 
Note that $v$ can be the image of several critical points. 
The next statement uses the vocabulary introduced in \S~\ref{subs:hl}. 

\begin{lem}
\label{lem:folding}
Assume that $\abs{a}\leq     R_f^{d-1}/(400d)$. 
Then, if $M(f)$ is large enough with respect to $d$, there exists $s\in [1, 2]$ such 
that  the set 
\[
\bb_2:= \varphi_f\inv (\disk(v, sR_f^d/8)) \cap \pi_2\inv (\disk(0, 10 R_f))
\] 
is a vertical sub-bidisk
in $\bb_3$  in 
the $(\varphi_f, y)$ coordinates, 
and $f$ realizes a crossed mapping of degree $d$ from $\bb_1$ to $\bb_2$ that is unramified over $\fr^v\bb_2$, and admits a solenoidal component 
 (relative to the projection $\varphi_f$).   
\end{lem}

\begin{proof} For notational lightness we write $R=R_f = R_p$. We  assume in a first stage that $R$ satisfies the requirements of Proposition~\ref{pro:estimate_varphi} and increase it further if necessary.

\noindent\textbf {Step 1.-- } Construction  of a bidisk $\bb_2'$ in the $(x,y)$ coordinates. 

Counted with multiplicities,  $p$ has $(d-1)$ critical values, one of them being $v$. 
When $y$ ranges in $\disk(0, 10R)$, the critical values of $x\mapsto p(x)+ay$ describe  disks $\Delta_j$, each of which of radius $10 R \abs{a}\leq  R^d/(40d)$.  
Subdivide the interval $\left[ {R^d}/{8},  {R^d}/{4}\right]$ into $2d$ sub-intervals of the form
$\left[ s_{k-1}{R^d}/{8}, s_{k}  {R^d}/{8}\right]$, with $s_k = 1+k/(2d)$, $1\leq k\leq 2d$, and 
 consider the corresponding   family of $2d$ concentric annuli 
$A_k:= \disk(v, s_{k} R^d/8 )\setminus \disk(v, s_{k-1} R^d/8)$.  
Since a disk $\Delta_j$ can intersect at most two consecutive annuli and one of them is centered at $v$, there is $k^\varstar$ such that  $A_{k^\varstar}$   is disjoint from all the $\Delta_j$.

We define $\bb_2'=\disk(v,  s^\varstar R^d/8)\times \disk(0, 10 R)$, where $s^\varstar=  \unsur{2}(s_{k^\varstar-1}+s_{k^\varstar})$.
We want to prove that the crossed map $f:\bb_1\to \bb_2'$ is unramified over $\fr^v\bb_2'$. For this,  
we introduce  the annulus $A'\subset A_{k^\varstar}$ defined by
\[
A' =  \disk(v,  s^\varstar R^d/8)\setminus \disk(v,  (s^\varstar - 1/10d) R^d/8),
\] 
so that $A'\times \disk(0, 10 R)$ is a neighborhood 
of $\fr^v\bb_2'$ in $\bb_2'$.

For $\abs{y}\leq 10R$ let us  estimate the slope 
of $f(\disk(0, 10R)\times \set{y})$   
over $A'$. This disk is parametrized by $t\mapsto (ay+p(t), t)$, with $\abs{t}\leq 10R$.
So, 
we fix $t$ such that $ay+p(t)\in A'$, 
and our goal is to bound the derivative $\abs{p'(t)}$ from below. 

Let $\delta = \dist(t, \Crit(p))$. 
Writing $p'(t)   = d \prod_{j=1}^{d-1} (t- c_j)$ we see that $\abs{p'(t)}\geq d \delta^{d-1}$. 
On the other hand, $\dist(p(t), p(\Crit(p)))\leq \delta \norm{p'}_{\disk(0, 10R)}$, and since by~\eqref{eq:bh-inclusion}
all critical points of $p$ are contained in $\disk(0, 2R)$ we see that $\norm{p'}_{\disk(0, 10R)} \leq d(12R)^{d-1}$. Since 
$ay+p(t) \in A'$ and $A_{k^\varstar}$ is disjoint from the disks $\Delta_j$, 
the distance between $p(t)$ and $p(\Crit(p))$  is $\geq (1/10d)R^d/8$. So we infer that 
\begin{equation}
\frac{R^d}{80d}  \leq \dist(p(t), p(\Crit(p)))\leq \delta d(12R)^{d-1},   
\end{equation}
hence  $\delta \geq c_dR$ with  
\begin{equation}
  c_d= \unsur{80 d^2 (12)^{d-1}}, 
\end{equation}
and we conclude that 
\begin{equation}\label{eq:p't}
\abs{p'(t)}\geq d c_d^{d-1} R^{d-1}.
\end{equation}
Therefore, the complex tangent vector $(p'(t), 1)$ to $f(\disk(0, 10R)\times \set{y})$ at $f(t,y)$  
is far from
the vertical direction, and condition~(i) of Definition~\ref{defi:non_ramified} is satisfied. By Remark~\ref{rem:non_ramified}, since $\bb_2'$ is a genuine  bidisk in $\C^2$, condition~(ii) is automatically satisfied
 and if follows that 
  $f:\bb_1\to \bb_2'$ is unramified over $\fr\bb_2'$. 

\begin{figure}
\includegraphics[width=\textwidth]{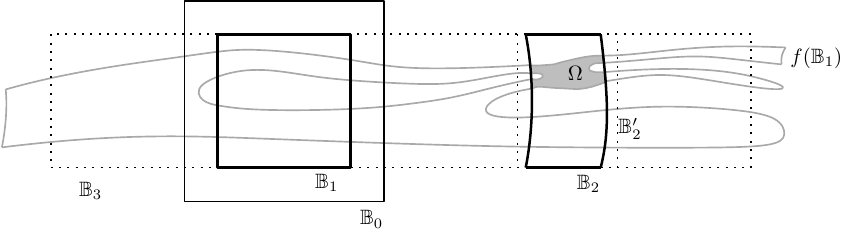}
\caption{\small Proof of Lemma~\ref{lem:folding}. The boundaries of the bidisks $\bb'_2$ and $\bb_3$ are dotted and  the solenoidal component of $f(\bb_1)\cap \bb_2$ is shaded.}\label{fig:bidisks}
\end{figure}

\noindent \textbf{Step 2.--} Perturbation to the $(\varphi_f, y)$ coordinates, and construction of $\bb_2$.

We set
$s = s^\varstar - 1/(100d)$ and $\bb_2 = \Phi\inv(\disk(v,  s R^d/8) \times \disk(0, 10 R))$, where $\Phi(x,y) = (\varphi_f(x,y),y)$. 
Here, we want $\disk(v,  s R^d/8) \times \disk(0, 10 R)$ to be contained in the image of $\Phi$, i.e.\ in the open set
$\set{(x,y) \; ; \;   \abs{x}> \max(\abs{y}, 12dR_f)+6R_f}$ from Proposition~\ref{pro:slope}. 
To get this inclusion, we rely on the estimates of Lemma~\ref{lem:basic_estimates_bottcher}: it suffices that $R^d\geq 40dR$.
Then, by construction $\bb_2$ is a bidisk in the $(\varphi_f, y)$ coordinates. Moreover, 
 it is a vertical sub-bidisk of  $\bb_3$, in the sense that the identity map defines a crossed mapping of degree $1$
  from $\bb_3$ to $\bb_2$ (see Figure~\ref{fig:bidisks}). 
   Then, by composition $(f, \bb_1, \bb_2)$  defines a crossed map of degree $d$.
  
From  our assumption on $R$ the conclusions of Propositions~\ref{pro:estimate_varphi} and~\ref{pro:slope} hold in $\bb_2$. 
So, if furthermore  $4R<\abs{s - s^\varstar} R^d/8$, that is $R^{d-1}>3200d$, we get 
\begin{equation} 
\disk(v,  (s^\varstar - 2/(100d)) R^d/8)\times \disk(0, 10 R)\subset \bb_2\subset \bb'_2.
 \end{equation}
  Likewise if we define
 \begin{equation}
 A =  \disk(v,  s  R^d/8)\setminus \disk(v,  (s - 98/(100d)) R^d/8),
 \end{equation} 
 then  $\Phi\inv (A\times \disk(0, 10 R))$ is a neighborhood of $\fr^v\bb_2$ that is contained  in $A'\times \disk(0, 10 R)$. 

We can now prove that the crossed mapping $f\colon \bb_1 \to \bb_2$ is unramified
over $\fr^v\bb_2$.
The $x$-coordinate in $A'$ is bounded below by $R^d/2$, hence 
 by Proposition~\ref{pro:slope},
  the slope of the $\varphi_f$-fibers over the second coordinate in  $\Phi\inv (A\times \disk(0, 10 R))$
  is bounded by $5R/(R^d/2) = 10 /R^{d-1}$. 
  It follows from  the lower bound~\eqref{eq:p't}   that 
  if $10/R^{d-1}< d c_d^{d-1} R^{d-1}$  (so $R> 30 d (12)^{(d-1)/2}$ is
  enough), 
the image of any horizontal line $L$ in $\bb_1$ is unramified over $\fr^v\bb_2$, 
  that is, condition~(i) of Definition~\ref{defi:non_ramified} holds. 
   Let us check condition~(ii). Since the non-ramification property extends to $\bb'_2$, 
   any component $\Delta'$ of $f(L)\cap \bb'_2$ is of the form $\Delta\cap \bb'_2$, 
   where $\Delta$ is a component of $f(L)\cap \bb_2$. As observed in Step 1, $\Delta$ is a holomorphic disk  
 because $\bb'_2$ is defined by global projections. 
 In addition, $\varphi_f$ is defined in a neighborhood of $\bb_2$, so  $\Delta'= \Delta\cap \varphi_f\inv(\disk(v, sR^d/8))$, which is a disk by the maximum principle. From this, we conclude that $(f,\bb_1, \bb_2)$ is unramified over $\fr^v\bb_2$.

Finally,  let us  show that  $f(\bb_1)\cap \bb_2$ admits a solenoidal 
component $\Omega$. Indeed,   look at the 
component $U\subset \C$ of 
$p\inv(\disk(v,s^\varstar R^d/8))$ containing $c$ 
(it may contain other critical points (\footnote{It is a feature of our proof that for $f$ we cannot distinguish $v$ from any other critical value located at a distance $o(R^d)$ from it.})).
 By the non-ramification property, if $\abs{y}\leq 10R$ then
 $x\mapsto \varphi_f\circ f (x, y)$ has no critical point in a neighborhood of $\fr U$. 
Let $\Omega$ be the component of $f(\bb_1)\cap \bb_2$ containing 
$f(U\times \set{0})\cap \bb_2$. 
Since  $p: U\to \disk(v,s^\varstar R^d/8)$ is a branched covering of degree $q\geq 2$, 
the above analysis shows that 
  $\varphi_f\circ f (\cdot, 0)$  is also
 of degree $q$ in  $U\cap (\varphi_f\circ f)\inv(\disk(v,s R^d/8)$.  
This proves that 
$\Omega$ has degree $q$, so it is solenoidal. 
 \end{proof}

\begin{rem}\label{rem:bound_M}
The proof requires   $R_f\geq \max ((3200d)^{1/(d-1)},  30 d (12)^{(d-1)/2})$, 
which also guarantees the assumption   
$\abs{a}\leq R_f^{d-1}/(400d)$  for any given $\abs{a}\leq 1$, say. 
For $d\leq 4$, a sufficient condition  is  $M(f)\geq 9$.
\end{rem}

\begin{lem}\label{lem:disconnected}
Under the assumptions of Lemma~\ref{lem:folding}, $f(\bb_1)\cap \bb_1$ is disconnected, 
$f$ has a disconnected Julia set, and $f$ is unstably disconnected. 
\end{lem}

\begin{proof}
Equivalently, let us show  that $f\inv(\bb_1)\cap \bb_1$ is disconnected. 
For this, we claim that for every component $U$ of $\set{G_p<M(f)}$, 
$f\inv(\bb_1)\cap \bb_1$ admits a 
vertical component contained in  $U\times \disk(0, 10R_f)$. 
Indeed, $p$ sends $\fr U$ onto the Jordan curve $\set{G_p=dM(f)}$. 
Thus, applying exactly the same estimates as for Lemma~\ref{lem:crossed}, we see that 
$f$ defines a crossed mapping from $U\times \disk(0, 10R_f)$ to $\bb_1$ (and even $\bb_3$)  and 
the claim follows. 

To prove that  $J_f$ is disconnected,  it is enough to show that  
 every component of $f(\bb_1)\cap \bb_1$ intersects $J_f$ (more precisely $J_f^\varstar$). 
Indeed,  a component $\Omega$ of $f(\bb_1)\cap \bb_1$ is a horizontal open subset of $\bb_1$. If
 $L^v$ is  any vertical line intersecting $\Omega$, 
  $f\inv(\Omega\cap L^v)$ is a vertical submanifold in $\bb_1$, so it intersects every horizontal  positive closed current, in particular $T^-$. 
  Pushing forward by $f$, and using the 
 invariance relation $f_*T^- = dT^-$, we deduce that    $T^-$ 
 intersects $\Omega\cap L^v$. Therefore $T^-$
 admits a   component of positive mass in $\Omega$, 
 which must intersect $T^+$, and  finally $T^+\wedge  T^-$ gives positive mass to 
 $\Omega$, as desired. 

Finally, if $\abs{a}\leq 1$, Theorem 0.2 in \cite{bs6} implies that $f$ is unstably disconnected, while 
if   $\abs{a}> 1$ this is automatic (see~\cite[Cor. 7.4]{bs6}).
\end{proof}

\subsubsection{Conclusion of the proof of Proposition~\ref{pro:lyapunov_henon}}\label{subsub:conclusion_proof_henon}

The argument  is based on the formula for $\chi^+$ given in~\cite{bs5}, 
and uses ideas from~\cite{lyapunov}.  
The critical measure $\mu_c^-$ from~\cite{bs5} is obtained by putting a point mass at every tangency point between the ``unstable lamination''  
and the foliation induced by the fibers of $\varphi_f$ in $\C^2\setminus K^+$, and integrating with respect to the transverse measure (see \cite{bs5} or \cite[\S 2.4]{lyapunov} for a formal presentation).
The positive measure $G_f\mu_c^-$ is invariant, and
 the formula of~\cite{bs5} reads
\begin{equation}\label{eq:bs5}
\chi^+(\mu_f)= \log d + \int_{\set{A\leq G_f< dA}} G_f\mu_c^-,
\end{equation}
 where $A>0$ is arbitrary.

By Lemma~\ref{lem:disconnected},  $f$ is unstably disconnected and we will take advantage of this property to control the critical measure. 
  By \cite[Thm 2.4]{connex}, in every bidisk $\bb$ in which $T^-$ is horizontal, 
  $T^-\rest{\bb}$ admits a decomposition 
  \begin{equation} T^-\rest{\bb}  = \sum_{k=1}^\infty T^-_k,  \label{eq:decomposition}
  \end{equation} where 
$T^-_k$ is an integral of horizontal disks   of degree $k$. If in addition $\bb$ is a bidisk in the coordinates $(\varphi_f, y)$, the critical points are the points of tangency 
between the vertical direction (i.e. the fibers of $\varphi_f$)  and  the leaves of $T^-_k$.  
By the Riemann-Hurwitz formula, any horizontal disk of degree $k$ admits $k-1$ vertical tangencies.
Furthermore, the total transverse measure of $T^-_k$ is $1/k$ times the mass
of the vertical slices of $T^-_k$: indeed, intuitively,
the contribution of   a horizontal disk of degree $k$  to the slice mass is $k$ times the transverse measure of one local plaque.
 We    have thus established the following result (see \cite[Prop. 2.15]{lyapunov} for details)

\begin{lem}\label{lem:critical_mass}
Assume that $f$ is unstably disconnected, and let $\bb$ be a bidisk in the  $(\varphi_f, y)$ coordinates in  which $T^-\rest{\bb}$ is horizontal. Write the decomposition 
 $T^-\rest{\bb}  = \sum_{k=1}^\infty T^-_k$  as above. 
Then $$\mu_c^-(\bb) = \sum_{k=1}^\infty \frac{k-1}{k} \sm(T^-_k),$$ where $\sm(\cdot)$ is the slice mass. 
\end{lem}

We are now ready to prove    the lower estimate~\eqref{eq:lyap_estimate_inf}. 
 We resume the notation from \S~\ref{subsub:crossed}-\ref{subsub:folding} in particular  the bidisks $\bb_1$, $\bb_2$. 
Since by Lemma~\ref{lem:disconnected}
$f$ is unstably disconnected,  
  $T^-\rest{\bb_2}$ 
admits a decomposition of the form~\eqref{eq:decomposition}. 
Let $\Omega$ be the solenoidal component of $f(\bb_1)\cap \bb_2$ constructed in Lemma~\ref{lem:folding}, and $q\geq 2$ be its degree. 
By Proposition~\ref{propdef:non-ramified} the degree of every horizontal disk 
in the decomposition of   $T^-\rest{\bb_2}$ is a multiple of $q$, so  we may write 
   $T^-\rest{\Omega} = \sum_{\ell=1}^\infty T^-_{q\ell}$. 
We claim that the slice mass of the horizontal current 
$T^-\rest{\Omega}$ equals $q/d$. Indeed, if $L^v$ is any vertical line in $\bb_2$ 
(that is, a fiber of $\varphi_f$), 
\begin{equation}\label{eq:Lomega}
\norm{(T^-\rest {\Omega})\wedge [L^v]}  = \norm{(T^-\wedge[L^v])\rest{\Omega}}= 
\norm{T^-\wedge [L^v\cap\Omega]}
 \end{equation}
is independent of $L^v$ (where $\norm{\cdot}$ denotes the mass of a positive measure)
and is precisely the slice mass of $T^-\rest{\Omega}$.
By definition of the degree of a solenoidal component, 
$f\inv(L^v\cap \Omega)$ is a vertical submanifold of degree $q$: indeed 
 for any horizontal line $L^h$ in $\bb_1$, 
 \[
 \# (L^h\cap f\inv(L^v\cap \Omega)) = \# (f(L^h)\rest{ \Omega}\cap L^v) = q.
 \] Hence 
by    invariance we get  
\begin{equation}\label{eq:muc}
\norm{T^-\wedge [L^v\cap\Omega]} = \norm{f^*(T^-\wedge [L^v\cap \Omega])}
= \norm{\lrpar{\unsur{d}T^- \wedge [f\inv(L^v\cap \Omega)]}} = \frac{q}{d}.
\end{equation}
Applying Lemma~\ref{lem:critical_mass} to $T^-\rest{\Omega}$, 
we estimate the critical mass of $\Omega$:
\begin{align}
 \mu^-_c(\Omega) & = 
  \sum_{\ell=1}^\infty \frac{q\ell-1}{q\ell} \sm(T^-_{q\ell})  
  \geq \frac{q-1}{q}  \sum_{\ell=1}^\infty \sm(T^-_{q\ell}) 
 \\
  &\notag =  \frac{q-1}{q} \sm (T^-\rest{\Omega})  = \frac{q-1}{d}\geq \frac1d.  
\end{align}
 
We claim that if $R_f$ is chosen so large that    Lemma~\ref{lem:folding} applies, then
  \begin{equation}\label{eq:3/4-5/4} \frac34 R_f^d \leq \abs{\varphi_f}\leq \frac54R_f^d \end{equation} 
  on $\bb_2$. Indeed in this region we have $\abs{\varphi_f(x,y) - x}\leq 2R_f$ and 
for $(x,y)\in \bb_2$ we have $\varphi_f(x,y)\in \disk(v, sR^d/8)$, where $s\leq 2 - 1/(100d)$. We leave the reader check that the condition $R^{d-1}\geq 3200d$ suffices to ensure the desired estimate.

  Set $A = dM(f)+ \log\frac34$. 
  Then   the fundamental domain $\set{A\leq G_f< dA}$ contains $\bb_2$ as soon as  $(d^2-d)M(f)> d\log(4/3)+\log (5/4)$.
  Thus, by the Bedford-Smillie 
 formula~\eqref{eq:bs5} and the bound~\eqref{eq:muc}
we get 
\begin{equation} \label{eq:lower_lyap}
\chi^+(\mu_f)\geq   \log d + \int_{\Omega} G_f\mu_c^- \geq \log d+ \frac{A}{d} 
 = \log d +M(f) + \unsur{d} \log\frac34,
\end{equation}
thereby completing the proof.  
\qed
 

\subsection{Compositions of Hénon maps}\label{subs:lyapunov_multi}
Recall that for a composition $f = h_k\circ \cdots \circ h_1$, we defined 
$M(f) = \max (M(h_i)) = \max (M(p_i))$. 

\begin{pro}\label{pro:lyapunov_multi} 
Let $f = h_k\circ \cdots \circ h_1$, with $h_i(x,y) = (a_i y+ p_i(x), x)$ be a composition of 
monic and centered Hénon maps,  with  $\abs{a_i}\leq  R_f^{d_i-1}$ for every $i$. Then  
\begin{equation}
\chi^+(\mu_f) \leq \log d+   dM(f) + d \log (15d).
\end{equation}
If furthermore
$\abs{a_i}\leq  R_f^{d_i-1}/(400 d_i)$ for every $i$  and 
 $M(f)$ is sufficiently large, we have the estimate
\begin{equation}\label{eq:lyapunov_multi}
   \chi^+(\mu_f) \geq \log d +   M(f) -  \unsur{\min(d_i)}\log \frac43 .
\end{equation}
\end{pro}

\begin{rem} 
In the case of Hénon maps we can use $f\inv$ to obtain 
an estimate without any bound on $\abs{a}$ 
(see \S~\ref{subsub:comments}), but for compositions of Hénon maps 
this trick cannot be used,  so it 
makes sense to consider the possibility for the  $a_i$ to be large. Furthermore,  
our method breaks down if the assumption on the   $a_i$ is dropped: see 
Example~\ref{eg:fafinva} below for the discussion of a specific example.  
To compare with the discussion of \S~\ref{subsub:comments}, we thus have found 
 a large region in the space of compositions of Hénon maps where the Lyapunov exponents of $\mu_f$ can be understood, but for $k\geq 2$, 
a part of the parameter space remains in the shadow. 
\end{rem}

The main part of the proof is already contained in that of 
Proposition~\ref{pro:lyapunov_henon}, that  we will apply to a factor $h_i$ from the composition that maximizes $M(p_i)$. 
We first need to get good  estimates for $G_f$ and $\varphi_f$. 
(Again no effort is made   to optimize the constants.)

\begin{pro}\label{pro:estimate_varphi_multi}
Let $f = h_k\circ \cdots \circ h_1$ be as above, with $\abs{a_i}\leq R_f^{d_i-1}$ for every $i$. Then $\varphi_f$ is well-defined in the domain defined by  $\abs{y}\leq \abs{x}$ and $\abs{x}\geq  12d R_f$,  and in this domain we have
  $$\abs{\frac{\varphi_f(x,y)}{x} - 1} \leq   2 \frac{R_f}{  \abs{x}} \text{ and } \abs{G_f(x,y) - \log \abs{x}}\leq     2  \frac{R_f}{\abs{x}} .$$
\end{pro}

 
\begin{proof}
As before we denote by $V_0^+$ the domain 
 $\set{\abs{x}>12d R_f, \ \abs{y}<\abs{x}}$, and we  
define $ (x_1, y_1) = f(x,y)  $ and  $  \theta(x,y) =  {x_1}/{x^d} - 1$. 
Thanks to  Proposition~\ref{pro:estimate_varphi} we may 
 assume $k\geq 2$, hence  $d\geq 4$.

Suppose that we can prove 
\begin{equation}
\label{eq:upper_bound_on_theta_composition}
\abs{  \theta(x,y)}\leq 4\frac{R_f}{\abs{x}} \quad \forall (x,y)\in V_0^+,
\end{equation}
so in particular, $\abs{\theta}\leq 1/6$  in $V_0^+$. Then, we can  follow the proof of Proposition~\ref{pro:estimate_varphi}.
First, an immediate adaptation of the bound~\eqref{eq:10d} (using $d\geq 4$)
shows that  $V_0^+$  is stable under $f$. 
Then, the second step of the proof shows that in $V_0^+$  the two estimates
$\abs{\varphi_f(x,y) - x}\leq 2   {R_f}$ and $\abs{G_f(x,y)   - \log\abs{x}}\leq 2 \frac{R_f}{\abs{x}}$
hold.

To obtain the bound~\eqref{eq:upper_bound_on_theta_composition}, we set $(x'_0, y'_0)  = (x,y)$ and 
$(x'_i, y'_i) = h_i\circ\cdots \circ h_1 (x,y)$ for $1\leq i\leq k$, so that 
$(x'_k, y'_k) = (x_1, y_1)$. Since  we assume that 
$\abs{x}\geq 12dR_f$, we have $\abs{x}\geq 12d_iR_{h_i}$ for every $i$. 
Define 
\begin{equation}
\theta'_{i+1}   = \frac{x'_{i+1}}{(x'_i)^{d_{i+1}}} - 1 = \frac{\pi_1\circ h_{i+1} (x'_i, y'_i)}{(x'_i)^{d_{i+1}}} - 1
\end{equation}
and write 
\begin{align}\label{eq:theta1}
\notag x_1 = x'_k &= \frac{x'_k}{(x'_{k-1})^{d_k}}\lrpar{\frac{x'_{k-1}}{(x'_{k-2})^{d_{k-1}}}}^{d_k}\cdots 
\lrpar{\frac{x'_1}{(x'_0)^{d_1}}}^{d_k\cdots d_2} x^d\\ 
  &=  x^d (1+\theta'_k) \prod_{i=1}^{k-1} \lrpar{1+ \theta'_{i}}^{d_k\cdots d_{i+1}}. 
\end{align}
Since $\abs{x}\geq 12dR_f\geq 12 d_1R_{h_1}$, 
reasoning as in~\eqref{eq:theta} we get $\abs{\theta'_1}\leq 1/11$ and then we obtain $\abs{x'_1}\geq   10d \abs{x}$. 
This gives also $\abs{y'_1}  =     \abs{x}\leq \abs{x'_1}$. 
Thus we can iterate this reasoning and  we infer that 
$\abs{x'_{i+1}}
\geq 10d \abs{x'_i}$ and  $\abs{y'_{i+1}} < \abs{x'_{i+1}}$, hence 
$\abs{x'_i} \geq (10d)^i \abs{x}$ for $1\leq i\leq k$. Again by~\eqref{eq:theta} and the bound 
$R_f\geq R_{h_i}$, for $i\geq 1$ we get
\begin{equation}
\label{eq:theta2}
\abs{\theta'_{i+1}}
\leq \left( 2+\frac{1}{6d} \right)  \frac{R_f}{\abs{x'_i}} 
\leq \frac{25}{12(10d)^i } \frac{R_f}{\abs{x} } 
\leq\frac{2}{(10d)^{i+1}}, 
\end{equation} 
because $\abs{x}\geq 12dR_f$.
Combining~\eqref{eq:theta1} and~\eqref{eq:theta2} we obtain  
\begin{equation*}
\abs{\theta (x,y)} =\abs{ \frac{x_1}{x^d} - 1} \leq   
\abs{\theta_k'}  \prod_{i=1}^{k-1} \lrpar{1+\abs{\theta'_i}}^d \leq  
\frac{25}{12 (10d)^{k-1}}  \frac{R_f}{\abs{x}}  \prod_{i=1}^{k-1} \lrpar{1+\frac{2}{(10d)^i}}^d  .
\end{equation*}
Using the bound~\eqref{eq:elementary_inequalityII} and $d\geq 4$, one checks that the last product on the right hand side is bounded by $2$. 
Since $k\geq 2$ we reach the bound $\abs{\theta (x,y)}\leq  \frac{R_f}{\abs{x}}$. 
This proves~\eqref{eq:upper_bound_on_theta_composition}, hence also the proposition.
\end{proof}

\begin{proof}[Proof of Proposition~\ref{pro:lyapunov_multi}] 
Since $R_{h_i}\leq R_f$ for every $i$, by Lemma~\ref{lem:crossed} and Remark~\ref{rem:largerR}, 
$h_i$ is a Hénon-like map in 
$\bb_1 = \disk(0, 10R_f)^2$.
Furthermore, the estimates used to obtain  Lemma~\ref{lem:crossed}
and   Proposition~\ref{pro:estimate_varphi_multi}
 show that a point leaving $\bb_1$ under $h_i$ will never re-enter under a composition of the $h_j$, so any composition $h_{i_n}\circ\cdots \circ h_{i_1}$ is proper in $\bb_1$ in the sense of \S~\ref{par:iteration_henon_like}. 

Proposition~\ref{pro:slope} does not depend on the fact that $f$ is a Hénon map (rather than a composition of Hénon maps), 
so it holds verbatim in our setting, with  the same numerical constants. 
Then  the proof of  the upper bound for $\chi^+(\mu_f)$ is identical to the corresponding one in Proposition~\ref{pro:lyapunov_henon}.

After  a cyclic permutation of the $h_i$ (which does not affect the Lyapunov exponents)
we may assume that $R_f = R_{h_k}$. Let $c$ be a critical point  of $p_k$ such that 
$G_{p_k}(c) =  M(f)$ and $v=p_k(c)$. 

\begin{lem}
\label{lem:folding2}
Assume   that $\abs{a_k}\leq     R_f^{d_k-1}/(400d_k)$. 
Then if $M(f)$ is  large enough  there exists $s\in [1, 2]$ such     
that  $\bb_2:= \varphi_f\inv (\disk(v, sR_f^{d_k}/8)) \cap \pi_2\inv \disk(0, 10 R_f)$ 
is a vertical sub-bidisk in $\bb_3:=\disk(0, 6^{d_k}R_f^{d_k})\times \disk(0, 10R_f)$ 
 in  the $(\varphi_f, y)$ coordinates, 
and $h_k$ realizes a crossed mapping of degree $d$ from $\bb_1$ to $\bb_2$ that is unramified over $\fr^v\bb_2$, and admits a solenoidal component 
 (relative to the projection $\varphi_f$).   
\end{lem}

\begin{proof}
This is Lemma~\ref{lem:folding} applied to $h_k$, 
with the only difference that for the vertical projection in $\bb_2$ we use $\varphi_f$ and not $\varphi_{h_k}$.
So for the proof we can apply Step 1 to $h_k$ without modification, 
and in Step 2 we just incorporate the fact that  since  $\abs{x}\geq R_f^{d_k}/2$ in a neighborhood of $\bb_2$,   
the   slope  of the vertical fibers $\set{\varphi_f = \cst}$ is  bounded by $10 R_f^{1-d_k}$   and we proceed exactly in the same way. 
\end{proof}

\begin{lem} Under the assumptions of Lemma~\ref{lem:folding2}, 
$f(\bb_1)\cap \bb_1$ is disconnected, 
the Julia set $J_f$ is disconnected, 
and $f$ is unstably disconnected. 
\end{lem}

\begin{proof}
Indeed, we observe that Lemma~\ref{lem:disconnected} can be  applied to  $h_k$, so 
$h_k(\bb_1)\cap \bb_1$ is disconnected. Let 
$\el_i = d_i\inv 1_{\bb_1}(h_i)_*$ be the graph transform operator for currents 
associated to the Hénon-like map $h_i$ in $\bb_1$. 
The fact that  the $h_i$ are composed properly guarantees that 
$\el_k\cdots \el_1  = \el =  d\inv 1_{\bb_1} f_*$, in particular 
$\el_k(\el_{k-1}\cdots \el_1 (T^-\rest{\bb_1})) = T^-\rest{\bb_1}$ (where $T^- = T^-_f$), and also 
 \begin{equation}\label{eq:hkstar}
 d_k\inv (h_k)_* \lrpar{\el_{k-1}\cdots \el_1 (T^-\rest{\bb_1})}\rest{\bb_3}  = T^-\rest{\bb_3}.
 \end{equation}

Arguing as in Lemma~\ref{lem:disconnected}, we infer that 
$T^- = \el_k(\el_{k-1}\cdots \el_1 (T^-\rest{\bb_1}))$ admits a component of positive mass in every component of $h_k(\bb_1)\cap \bb_1$, which must thus intersect $T^+$, thus  $J_f^\varstar$ intersects every component of $h_k(\bb_1)\cap \bb_1$ and we are done.
\end{proof}
 
 Since $f$ is unstably disconnected, $T^-\rest{\bb_3}$ admits  a decomposition into 
an integral of horizontal disks of finite degree in   $\bb_3$, so by restriction 
this holds in $\bb_2$ as well. 
By Lemma~\ref{lem:folding2}, 
 $h_k(\bb_1)\cap \bb_2$ admits a solenoidal component $\Omega$ of some degree $q\geq 2$, 
 so the degree of any horizontal submanifold in $\Omega$ is a multiple of $q$. 
 We claim that the slice mass of $T^-\rest{\Omega}$ equals $q/d_k$. 
 For this, we take a vertical fiber $L$ of $\varphi_f$ in $\bb_2$, and  exactly as in the argument 
 around Equation~\eqref{eq:Lomega}, we see that the mass of $T^-\wedge [L\cap \Omega]$ is independent of $L$, and that $h_k\inv(L\cap \Omega)$ is a vertical manifold in $\bb_1$ of degree $q$. Using the identity~\eqref{eq:hkstar},  
  we get
\begin{align}
\norm{T^-\wedge [L\cap\Omega]} &= \norm{h_k^*(T^-\wedge [L\cap \Omega])}\\
\notag &= \norm{\lrpar{\unsur{d_k}\el_{k-1}\cdots \el_1 (T^-\rest{\bb_1}) \wedge [h_k\inv(L\cap \Omega)]}} = \frac{q}{d_k},
\end{align}
and our claim is proved.  
Then, as in~\eqref{eq:muc}, the critical mass of $\Omega$ is $\geq 1/d_k$. 

For the final estimate we argue as in Proposition~\ref{pro:lyapunov_henon}: since for large $M(f)$ 
\begin{equation}
\Omega\subset \set {\frac34 R_f^{d_k}\leq \abs{\varphi_f}\leq \frac54R_f^{d_k} } \subset
\set{A\leq  G_f<  dA} 
\end{equation}
with 
$ A = d_k M(f)+  \log\frac34$,    by using 
 $\mu_c(\Omega)\geq 1/d_k$    we get 
 \begin{align} \label{eq:lower_lyap_multi}
\chi^+(\mu_f) \geq   
\log d + \int_{\Omega} G_f\mu_c^- 
\geq \log d+ \frac{A}{d_k}  
\geq \log d +M(f) + \unsur{d_k} \log\frac34,
\end{align}
thereby finishing the proof.  
 \end{proof}

\begin{eg}\label{eg:fafinva} 
A typical situation where the assumptions  of Proposition~\ref{pro:lyapunov_multi}  
are not satisfied   is the family 
$f_a(x,y) = (y/a+p(x),x )\circ (ay+p(x), x)$, when $p$ is fixed and $a\in \C^*$ goes to infinity. 
Here, the Jacobian is constant equal to $1$ but the multi-Jacobian diverges.
We take $p(x) = x^2$  and study 
 the regime where $\abs{a}$ is large. We claim that if   $R  \geq 1$ is such that 
 $f_a$ is a Hénon-like map of degree $d^2$ in the bidisk $\bb= \disk(0, R)^2$, then 
 $f(\bb)\cap \bb$ is connected. Indeed, look at the image of the horizontal line $L=\set{y=0}$.
  If we parameterize $L$ by $ t\mapsto (t, 0)$, $t\in \disk(0, R)$, then, $f_a(L)$  
 is parameterized by $t\mapsto f_a(t, 0) = (t/a+t^4, t^2)$. The critical points of $\pi_1\circ f_a(t, 0)$ are the cube roots $(4a)^{-1/3}$ so the critical values are of the form $ca^{-4/3}$ and they are all close to 0, from which we conclude that $f_a(L)\cap \bb$ is connected. 
 Arguing as in Lemma~\ref{lem:disconnected}, we see that $f_a(L)\cap \bb$ intersects 
 any component of $f_a(\bb)\cap \bb$, so $f_a(\bb)\cap \bb$ is connected. 
In particular, for large $a$,  we do not know how to estimate the Lyapunov exponents of $\mu_{f_a}$, 
nor if its Julia set can be connected. 
\end{eg}

\subsection{Consequences} \label{subs:comments_corollaries}

\begin{proof}[Proof of Theorem~\ref{mthm:lyapunov}]  
Let $B = B(\mathbf d, \mathbf a)$ be such that the estimate of Proposition~\ref{pro:lyapunov_multi} holds for 
$M(f)\geq B$.  
Lemma~\ref{lem:M} guarantees that  $\set{f\in \mathcal H^k_{\mathbf d, \mathbf a}\; ; \;  1\leq M(f)\leq B}$  
is compact, so  $\chi^+(\mu_f)$ is   bounded from above  on that subset, and 
$\chi^+(\mu_f)\geq \log d$ for every $f$ by~\cite{bs3}. 
The result follows. 
\end{proof}

Let us now  explain how Theorem~\ref{mthm:lyapunov} implies 
   Theorem~\ref{mthm:stable}. 
Recall that `stable' in Theorem~\ref{mthm:stable} means stable of type (III). 
The following lemma should be compared to Lemma~\ref{lem:jacobian}. 

\begin{lem}\label{lem:jacobian2}
Let $(f_\lambda)_{\lambda\in \Lambda}$ be an  algebraic family  of loxodromic automorphisms, with $\Lambda$ irreducible. If $(f_\lambda)_{\lambda\in \Lambda}$ is stable of type (III),
 then  $\jac(f_\lambda)$ and the Lyapunov exponents of $\mu_{f_\lambda}$ are   constant on $\Lambda$.
\end{lem}

\begin{proof}
By assumption, for some $f_0\in \Lambda$, there exists a set $\mathrm P$ of saddle points of $f_0$
of positive upper density 
which can be continued holomorphically over $\Lambda$ as a set of saddle points $P_\lambda$. As in Lemma~\ref{lem:jacobian}, for each of these saddle points, the stable and unstable multipliers are constant,
 so the Jacobian is constant. For any $\lambda\in \Lambda$, $\mathrm P_\lambda$ 
intersects the set $\SPer^+(f_\lambda)$ of Theorem~\ref{thm:equidistribution_reinforced}, thus  $\chi^+(\mu_{f_\lambda}) = \chi^+(\mu_{f_0})$, as announced. Since the Jacobian is constant, we conclude that  $\chi^-(\mu_{f_\lambda})$ is  constant as well.
\end{proof}

\begin{proof}[Proof of Theorem~\ref{mthm:stable}]
Let $\mathcal H$ be either of $\mathcal H^1_d$ or $\mathcal H^k_{\mathbf d,\mathbf a}$ for $k\geq 2$ and let $\Lambda\subset \mathcal H$ be a stable algebraic family.
By Lemma~\ref{lem:jacobian2}, $\lambda \mapsto \chi^+(\mu_{f_\lambda}) $ is constant on every irreducible component of $\Lambda$, 
and when
$k = 1$, $\lambda\mapsto \jac(f_\lambda)  = -a_\lambda$ is   constant as well. 
Thus in both cases, Theorem~\ref{mthm:lyapunov} and Lemma~\ref{lem:M} imply that 
$\Lambda$ is bounded, so it is finite.
\end{proof}

\begin{rem}
Thanks to Proposition~\ref{pro:lyapunov_multi} we could replace 
$\mathcal H^k_{\mathbf{d}, \mathbf{a}}$ in  Theorem~\ref{mthm:stable} by any algebraic family $f_\lambda$ of compositions of Hénon mappings such that the multi-Jacobian satisfies $\abs{a_{i,\lambda}}   = o(R_{f_\lambda}^{d_i-1})$ for every $i$.  
\end{rem}

The following result is the counterpart of Theorem~\ref{thm:traces} for unstable multipliers, and 
implies the   statements of  Theorems~\ref{mthm:henon_multipliers} and~\ref{mthm:multi-Jacobian} relative to the unstable multiplier  spectrum. 
 
\begin{thm}
\label{thm:u_multipliers}
Let $\mathcal H$ be $\mathcal H^1_d$ for some $d\geq 2$ or 
$\mathcal H^k_{\mathbf d,\mathbf a}$ for some $\mathbf a\in (\C^*)^k$ and some $\mathbf d\in \N^k_{\geq 2}$. 
For every $f_0 \in \mathcal H$, the isospectral subset 
 $\Sigma^u  = \set{f \in \mathcal H\; ; \; \SMult^u(f) = \SMult^u(f_0)}$ is  finite.  
\end{thm}

\begin{proof}
Let $E_n =  \SMult^u_n(f_0)$. 
Proposition~\ref{pro:multipliers} and 
Theorem~\ref{mthm:stable} imply that $\Sigma^u$ is a discrete subset of $\mathcal H$. To prove that 
$\Sigma^u$ is finite it is enough  to show that it is bounded. For this, we observe that if 
$\abs{\jac(f_0)}\geq1$, then for any saddle point $p$ of exact period $n$ we have 
$\abs{\lambda^u(p)}\geq  \abs{\jac(f_0)}^n / \abs{\lambda^s(p)} > \abs{\jac(f_0)}^n$. 
We also have $\abs{\lambda^u(p)}\leq (\norm{Df_0}_{\jstar_0}^n$, where $\jstar_0:=\jstar(f_0)$.
 Thus $E_n$ is contained in the annulus $\{\min(1,\abs{\jac(f_0)})^n \leq \abs{z} \leq (\norm{Df_0}_{\jstar_0})^n\}$. 
 For any $f\in \Sigma^u$, the same reasoning applies, therefore 
any $z\in E_n$ satisfies $\abs{z}\geq \min(1, \abs{\jac(f)})^n$, from which we deduce that 
$\abs{\jac(f)}\leq \norm{Df_0}_{\jstar_0}$. (Note that this is automatic in $\mathcal H^k_{\mathbf d,\mathbf a}$ since there, the Jacobian is fixed.) 
In addition  for 
$f\in \Sigma^u$ the Lyapunov exponents of saddle points are bounded by
 $\log \norm{Df_0}_{\jstar_0}$, so $\chi^u(f)\leq \log \norm{Df_0}_{\jstar_0}$.
 By  Proposition~\ref{pro:lyapunov_henon} (for $\mathcal H^1_d$) and 
 Proposition~\ref{pro:lyapunov_multi} or 
 Theorem~\ref{mthm:lyapunov}   (for $\mathcal H^k_{\mathbf d,\mathbf a}$) 
we get a uniform bound on $M(f)$.
Together with the bound on the Jacobian (resp. multi-Jacobian), this entails that 
$\Sigma^u$ is bounded, as was to be shown. 
\end{proof}

We can also complete the proof of Theorem~\ref{mthm:C1}. 

\begin{pro}\label{pro:C1_finite}
The set $F$ defined in Theorem~\ref{thm:C1} is finite. 
\end{pro}

\begin{proof}
Here we define $E_n$ as in Equation~\eqref{eq:En_C1}. Then as in the proof of Theorem~\ref{thm:u_multipliers}, 
if $f$ is $C^1$ conjugate to $f_0$ on a neighborhood of $\jstar$, 
we obtain uniform upper bounds on $\abs{\jac(f)}$ and  $M(f)$, 
so $F$ is a bounded subset of  $\mathcal H$ and we are done.  
\end{proof}

Define the connectedness locus ${\mathcal C}_d \subset \mathcal H_{d}$ (resp.\  ${\mathcal C}^k_{\mathbf d, \mathbf a} \subset \mathcal H^k_{\mathbf d, \mathbf a}$)  to be the set of automorphisms with connected Julia set. By \cite[Cor. 2.2]{connex}, ${\mathcal C}_d $ and ${\mathcal C}^k_{\mathbf d, \mathbf a}$ are closed. 

\begin{cor}\label{cor:connectedness}
The connectedness locus ${\mathcal C}^k_{\mathbf d, \mathbf a} \subset \mathcal H^k_{\mathbf d, \mathbf a}$ is compact
for any $k\geq 1$, $\mathbf d\in (\N_{\geq 2})^k$ and $\mathbf{a}\in (\C^*)^k$. 
\end{cor}

An example like the family $f_a(x,y) = (ay+x^2, x)$
shows that   this result  does not hold if the multi-Jacobian when not fixed: indeed 
since $f_{a\inv}$ is conjugate to $f_a\inv$, $J(f_a)$ is connected for large $\abs{a}$ 
\begin{vlongue}
(see \S~\ref{subs:eg_fa})
\end{vlongue}
.

\begin{proof}
The connectedness locus ${\mathcal C}^k_{\mathbf d, \mathbf a}$  is closed,
so we only need to show that it is bounded. If  
$\abs{ {\jac(f)}} = \prod_{i=1}^k\abs{a_i}\leq 1$, then by Theorems 0.2 and 7.3 in~\cite{bs6}, 
$J_f$ is connected if and only if $\chi^+(\mu_f)   = \log d$, so the boundedness  follows immediately from Theorem~\ref{mthm:lyapunov}. 
If $\prod_{i=1}^k\abs{a_i}> 1$,  we  reduce to the dissipative case by considering  $f\inv$.  
If $f\in {\mathcal C}^k_{\mathbf d, \mathbf a}$, the proof of Lemma~\ref{lem:M} implies 
that $f\inv$ is conjugate to $g\in  {\mathcal C}^k_{\mathbf {d}^\vee, \mathbf b}$, where 
$\mathbf {d}^\vee = (d_1, \cdots , d_k)$ and $\mathbf b$ is fixed.  
The argument for  the dissipative case implies that $M(g)$ is bounded, hence so is $M(f)$ (again by Lemma~\ref{lem:M}), and we are done. 
\end{proof}

\begin{vlongue}
\appendix 

\section{Maps with few multipliers} 

\begin{defi} A loxodromic automorphism $f$  is \emph{exceptional}  if there exists 
a complex number $\kappa$ such that  $\lambda^u(p) = \kappa^{n}$
 for every  $p\in \SPer_n(f)$ and every $n\geq 1$. 
\end{defi}

This concept arises naturally in many rigidity issues 
(see e.g.~\cite{BK1} or~\cite[\S 2]{rigidity} for our context). 
We saw in Example~\ref{eg:charac_p} that analogues of 
exceptional automorphisms exist in positive characteristic, but 
we are not aware of any example in characteristic $0$. Partial non-existence results for exceptional 
automorphisms have been  obtained in \cite[Prop. 6.1]{BK1} and~\cite[Prop. 8.9]{rigidity}.


\begin{thm}
If $(f_\lambda)$ is an algebraic family of loxodromic polynomial automorphisms of $\C^2$ in which any stable irreducible algebraic family is trivial, then the set of exceptional maps is discrete in $\Lambda$. 
\end{thm}

\begin{proof}
Let $\mathcal E\subset \Lambda$ be the set of exceptional maps. Put 
$$\widetilde {\mathcal E}_n = \set{\lambda\in \Lambda\, ; \,  \forall m\leq n, \, \forall p\in  \SPer_m, \, 
\forall q\in  \SPer_n, \, (\lambda^u(p))^n = (\lambda^u(q))^m} $$
and $\widetilde {\mathcal E}= \bigcap_{n\geq 1}\widetilde {\mathcal E}_n$,  
so that $\widetilde {\mathcal E}$ contains $\mathcal E$. Assume by way of contradiction that 
$\widetilde {\mathcal E}$ is not discrete and pick an accumulation point 
$f_0$. For  $n\geq 1$, let $V_n$ be a neighborhood of $f_0$ in which all saddle points of period at most $n$ persist.  
Then $\widetilde {\mathcal E}_n\cap V_n$ is an analytic subset of positive dimension
 in $V_n$; more precisely, it is a 
  union of components of the intersection of $V_n$ with some algebraic subvariety of $\Lambda$.  Reducing $V_n$ if necessary we may assume that (i) $\widetilde {\mathcal E}_n\cap V_n$ admits a unique connected component containing $f_0$ and (ii)  $V_{n+1}\subset V_n$. Let $n_1$ be such that the germ of 
   $\widetilde {\mathcal E}_n$ at $f_0$ is constant for $n\geq n_1$. 
   
We claim that  $\widetilde {\mathcal E}_{n_1}\cap V_{n_1}$ is a stable family of type (I).  Indeed,  if 
$n\geq n_1$ and $f\in \widetilde {\mathcal E}_n\cap V_n$, then 
for every $p = p(f)\in \SPer_n(f)$ and every $q= q(f)\in \SPer_{n_1}(f)$, we have a 
relation of the form $\lambda^u(p)^{n_1} = \lambda^u(q)^n$. Since $f$ has constant Jacobian, we have $\lambda^s(p)^{n_1} = \lambda^s(q)^n$ as well. 
Now, $q$ can be 
continued  holomorphically as a saddle along  any path in
$\widetilde {\mathcal E}_{n_1}\cap V_{n_1}$, and  
the persistence of the relations $\lambda^u(p)^{n_1} = \lambda^u(q)^n$ and 
$\lambda^s(p)^{n_1} = \lambda^s(q)^n$ show that $p$ can be followed 
as a saddle point along that path. Therefore $\widetilde {\mathcal E}_{n_1}\cap V_{n_1}$ is 
a stable family, as claimed. 

By definition $\widetilde {\mathcal E} \subset \widetilde {\mathcal E}_{n_1}$, 
and we have shown that all relations defining $\widetilde {\mathcal E}$ persist along $\widetilde {\mathcal E}_{n_1}\cap V_{n_1}$, 
thus $\widetilde {\mathcal E} \cap V_{n_1} =\widetilde {\mathcal E}_{n_1}\cap V_{n_1}$.
This reasoning shows that $\widetilde {\mathcal E}$ is an analytic subset in $\Lambda$, which 
is locally a union of local components of an algebraic variety. Thus, by analytic continuation, we conclude that any global analytic component of $\widetilde {\mathcal E}$ is a 
stable algebraic family, and the existence of a positive dimensional component contradicts our assumption. Therefore $\widetilde {\mathcal E}$ is a discrete set.
\end{proof}

\begin{rem}~
\begin{enumerate}[\rm (1)]
\item The proof shows that the definition of exceptional maps can be extended by asking that 
$\lambda^u(p) \in   \kappa^{n}F_n$, where $F_n$ is a  finite set of roots of unity. 
 (In~\cite[\S 2]{rigidity} we are interested in the situation where $\lambda^u(p) \in  \pm d^n$.)
\item It is enough to assume that the exceptional identity $\lambda^u(p) = \kappa^{n}$ holds for all saddle points 
along some subsequence $(n_j)$. The proof is the same except that we get a stable family of type (III). 
\item Contrary to the case of given unstable multipliers 
(Theorem~\ref{mthm:henon_multipliers} and~\ref{mthm:multi-Jacobian}) we cannot run the argument 
of Theorem~\ref{thm:u_multipliers} to conclude that the exceptional discrete set is finite. 
\end{enumerate}
\end{rem}
\end{vlongue}

 \bibliographystyle{plain}
\bibliography{bib-multipliers}

@article{favre-gong,
	author = {Favre, Charles and Gong, Chen},
	date = {2025/06/02},
	doi = {10.1007/s42543-025-00100-7},
	isbn = {2524-7182},
	journal = {Peking Mathematical Journal},
	title = {Non-Archimedean Techniques and Dynamical Degenerations},
	url = {https://doi.org/10.1007/s42543-025-00100-7},
	year = {2025}}

@unpublished{bera-verma,
	author = {Bera, Sayani and Verma, Kaushal},
	howpublished = {Preprint, {arXiv}:2601.07681 [math.{CV}] (2026)},
	note = {arXiv:2601.07681},
	title = {Rigidity of the escaping set of polynomial automorphisms of $\mathbb {C}^2$},
	url = {https://arxiv.org/abs/2601.07681},
	year = {2026}}

@incollection{zehnder,
	author = {Zehnder, Eduard},
	booktitle = {Geom. {Topol}., {III}. {Lat}. {Am}. {Sch}. {Math}., {Proc}., {Rio} de {Janeiro} 1976},
	doi = {10.1007/bfb0085385},
	number = {597},
	pages = {855-866},
	series = {Lecture Notes in Math.},
	title = {A simple proof of a generalization of a theorem by {C}. {L}. {Siegel}},
	year = {1977}}

@article{Wermer:1958,
	author = {Wermer, John},
	doi = {10.2307/1970155},
	fjournal = {Annals of Mathematics. Second Series},
	issn = {0003-486X},
	journal = {Ann. Math. (2)},
	language = {English},
	pages = {550--561},
	title = {The hull of a curve in $\mathbb {C}^n$},
	volume = {68},
	year = {1958},
	zbl = {0084.33402},
	zbmath = {3138457}}

@article{Wermer:1980,
	author = {Wermer, John},
	journal = {Mathematische Annalen},
	pages = {193-194},
	title = {On a Domain Equivalent to the Bidisk.},
	url = {http://eudml.org/doc/163377},
	volume = {248},
	year = {1980}}

@article{ishii:non_planar,
	author = {Ishii, Yutaka},
	doi = {10.1016/j.aim.2007.11.025},
	fjournal = {Advances in Mathematics},
	issn = {0001-8708},
	journal = {Adv. Math.},
	language = {English},
	number = {2},
	pages = {417--464},
	title = {Hyperbolic polynomial diffeomorphisms of $\mathbb {C}^2$. {I}: {A} non-planar map},
	volume = {218},
	year = {2008},
	zbl = {1136.37028},
	zbmath = {5268152}}

@article{milnor:quadratic,
	author = {Milnor, John},
	fjournal = {Experimental Mathematics},
	issn = {1058-6458,1944-950X},
	journal = {Experiment. Math.},
	mrclass = {58F20 (30D05 32S70 58F23)},
	mrnumber = {1246482},
	mrreviewer = {J.\ S.\ Joel},
	note = {With an appendix by the author and Lei Tan},
	number = {1},
	pages = {37--83},
	title = {Geometry and dynamics of quadratic rational maps},
	url = {http://projecteuclid.org/euclid.em/1062620736},
	volume = {2},
	year = {1993}}

@article{milnor:two_critical,
	author = {Milnor, John},
	fjournal = {Experimental Mathematics},
	issn = {1058-6458,1944-950X},
	journal = {Experiment. Math.},
	mrclass = {37F50 (37F20)},
	mrnumber = {1806289},
	mrreviewer = {Peter\ Ha\"issinsky},
	number = {4},
	pages = {481--522},
	title = {On rational maps with two critical points},
	url = {http://projecteuclid.org/euclid.em/1045759519},
	volume = {9},
	year = {2000}}

@book{wilkinson:book,
	author = {Wilkinson, Amie},
	publisher = {available at \url{https://www.math.uchicago.edu/%7Ewilkinso/papers/papers.html}},
	title = {Geometry, Dynamics, and Rigidity},
	year = {2025}}

@article{henonl,
	author = {Dujardin, Romain},
	fjournal = {American Journal of Mathematics},
	issn = {0002-9327,1080-6377},
	journal = {Amer. J. Math.},
	mrclass = {32H50 (37A35 37F10)},
	mrnumber = {2045508},
	mrreviewer = {Charles\ Favre},
	number = {2},
	pages = {439--472},
	title = {H\'enon-like mappings in {$\Bbb C^2$}},
	url = {http://muse.jhu.edu/journals/american_journal_of_mathematics/v126/126.2dujardin.pdf},
	volume = {126},
	year = {2004}}

@incollection{HOV2,
	author = {Hubbard, John H. and Oberste-Vorth, Ralph W.},
	booktitle = {Real and complex dynamical systems ({H}iller\o d, 1993)},
	isbn = {0-7923-3521-X},
	mrclass = {58F23 (30D05 32H50)},
	mrnumber = {1351520},
	mrreviewer = {Eric\ Bedford},
	pages = {89--132},
	publisher = {Kluwer Acad. Publ., Dordrecht},
	series = {NATO Adv. Sci. Inst. Ser. C: Math. Phys. Sci.},
	title = {H\'{e}non mappings in the complex domain. {II}. {P}rojective and inductive limits of polynomials},
	volume = {464},
	year = {1995}}

@article{morgan-shalen,
	author = {Morgan, John W. and Shalen, Peter B.},
	doi = {10.2307/1971082},
	fjournal = {Annals of Mathematics. Second Series},
	issn = {0003-486X,1939-8980},
	journal = {Ann. of Math. (2)},
	mrclass = {57N10 (14M99 32G15)},
	mrnumber = {769158},
	mrreviewer = {Christopher\ W.\ Stark},
	number = {3},
	pages = {401--476},
	title = {Valuations, trees, and degenerations of hyperbolic structures. {I}},
	url = {https://doi.org/10.2307/1971082},
	volume = {120},
	year = {1984}}

@article{degenerations_SL2,
	author = {Dujardin, Romain and Favre, Charles},
	doi = {10.5802/ahl.24},
	fjournal = {Annales Henri Lebesgue},
	issn = {2644-9463},
	journal = {Ann. H. Lebesgue},
	mrclass = {37H15 (37F30 37P50 60B15 60B20)},
	mrnumber = {4015916},
	mrreviewer = {Tomoki\ Mihara},
	pages = {515--565},
	title = {Degenerations of {${\rm SL}(2,\Bbb C)$} representations and {L}yapunov exponents},
	url = {https://doi.org/10.5802/ahl.24},
	volume = {2},
	year = {2019}}

@article{luo:2022,
	author = {Luo, Yusheng},
	doi = {10.1215/00127094-2022-0056},
	fjournal = {Duke Mathematical Journal},
	issn = {0012-7094,1547-7398},
	journal = {Duke Math. J.},
	mrclass = {37F20 (14G22 37F32 37P50)},
	mrnumber = {4491710},
	number = {14},
	pages = {2943--3001},
	title = {Trees, length spectra for rational maps via barycentric extensions, and {B}erkovich spaces},
	url = {https://doi.org/10.1215/00127094-2022-0056},
	volume = {171},
	year = {2022}}

@unpublished{favre-rivera:2025,
	author = {Charles Favre and Juan Rivera-Letelier},
	note = {arXiv:2504.20280},
	title = {Rigidit\'e, expansion et entropie en dynamique non-archim\'edienne},
	year = {2025}}

@article{favre:2020,
	author = {Favre, Charles},
	doi = {10.1017/S147474801800035X},
	journal = {Journal of the Institute of Mathematics of Jussieu},
	number = {4},
	pages = {1141--1183},
	title = {DEGENERATION OF ENDOMORPHISMS OF THE COMPLEX PROJECTIVE SPACE IN THE HYBRID SPACE},
	volume = {19},
	year = {2020}}

@unpublished{favre:2025,
	author = {Favre, Charles},
	note = {arxiv:2504.20284},
	title = {Blow-up of multipliers in meromorphic families of rational maps},
	year = {2025}}

@article{przytycki,
	author = {Przytycki, Feliks},
	doi = {10.1007/BF01388554},
	fjournal = {Inventiones Mathematicae},
	issn = {0020-9910,1432-1297},
	journal = {Invent. Math.},
	mrclass = {30D05 (58F11 58F12)},
	mrnumber = {784535},
	mrreviewer = {Robert\ L.\ Devaney},
	number = {1},
	pages = {161--179},
	title = {Hausdorff dimension of harmonic measure on the boundary of an attractive basin for a holomorphic map},
	url = {https://doi.org/10.1007/BF01388554},
	volume = {80},
	year = {1985}}

@article{manning,
	author = {Manning, Anthony},
	doi = {10.2307/2007044},
	fjournal = {Annals of Mathematics. Second Series},
	issn = {0003-486X,1939-8980},
	journal = {Ann. of Math. (2)},
	mrclass = {58F11 (28A75 30D05 54H20 58F13)},
	mrnumber = {740898},
	mrreviewer = {R.\ L.\ Adler},
	number = {2},
	pages = {425--430},
	title = {The dimension of the maximal measure for a polynomial map},
	url = {https://doi.org/10.2307/2007044},
	volume = {119},
	year = {1984}}

@article{bs3,
	author = {Bedford, Eric and Smillie, John},
	doi = {10.1007/BF01934331},
	fjournal = {Mathematische Annalen},
	issn = {0025-5831,1432-1807},
	journal = {Math. Ann.},
	mrclass = {32H50 (32C30 58F11 58F23)},
	mrnumber = {1188127},
	mrreviewer = {Yasuichiro\ Nishimura},
	number = {3},
	pages = {395--420},
	title = {Polynomial diffeomorphisms of {$\bold C^2$}. {III}. {E}rgodicity, exponents and entropy of the equilibrium measure},
	url = {https://doi.org/10.1007/BF01934331},
	volume = {294},
	year = {1992}}

@article{lyapunov,
	author = {Dujardin, Romain},
	doi = {10.1017/S0143385706001052},
	fjournal = {Ergodic Theory and Dynamical Systems},
	issn = {0143-3857,1469-4417},
	journal = {Ergodic Theory Dynam. Systems},
	mrclass = {37F10 (37A35 37D99)},
	mrnumber = {2342968},
	mrreviewer = {Laura\ G.\ DeMarco},
	number = {4},
	pages = {1111--1133},
	title = {Continuity of {L}yapunov exponents for polynomial automorphisms of {$\Bbb C^2$}},
	url = {https://doi.org/10.1017/S0143385706001052},
	volume = {27},
	year = {2007}}

@article{branner-hubbard1,
	author = {Branner, Bodil and Hubbard, John H.},
	doi = {10.1007/BF02392275},
	fjournal = {Acta Mathematica},
	issn = {0001-5962,1871-2509},
	journal = {Acta Math.},
	mrclass = {30D05 (58F08)},
	mrnumber = {945011},
	mrreviewer = {I.\ N.\ Baker},
	number = {3-4},
	pages = {143--206},
	title = {The iteration of cubic polynomials. {I}. {T}he global topology of parameter space},
	url = {https://doi.org/10.1007/BF02392275},
	volume = {160},
	year = {1988}}

@article{demarco-mcmullen,
	author = {DeMarco, Laura G. and McMullen, Curtis T.},
	doi = {10.24033/asens.2070},
	fjournal = {Annales Scientifiques de l'\'Ecole Normale Sup\'erieure. Quatri\`eme S\'erie},
	issn = {0012-9593,1873-2151},
	journal = {Ann. Sci. \'Ec. Norm. Sup\'er. (4)},
	mrclass = {37F10 (37F20 37F50)},
	mrnumber = {2482442},
	mrreviewer = {Charles\ Favre},
	number = {3},
	pages = {337--382},
	title = {Trees and the dynamics of polynomials},
	url = {https://doi.org/10.24033/asens.2070},
	volume = {41},
	year = {2008}}

@unpublished{rigidity,
	author = {Cantat, Serge and Dujardin, Romain},
	note = {arXiv:2411.10339},
	title = {Some rigidity results for polynomial automorphisms of $\mathbb{C}^2$},
	year = {2024}}

@book{silverman:book,
	author = {Silverman, Joseph H.},
	doi = {10.1007/978-0-387-69904-2},
	isbn = {978-0-387-69903-5},
	mrclass = {11-02 (11-01 11G05 11G07 11G50 37-02 37F10)},
	mrnumber = {2316407},
	mrreviewer = {Thomas\ Ward},
	pages = {x+511},
	publisher = {Springer, New York},
	series = {Graduate Texts in Mathematics},
	title = {The arithmetic of dynamical systems},
	url = {https://doi.org/10.1007/978-0-387-69904-2},
	volume = {241},
	year = {2007}}

@unpublished{huguin:polynomials,
	author = {Huguin, Valentin},
	note = {arXiv:2412.19335},
	title = {Moduli spaces of polynomial maps and multipliers at small cycles},
	year = {2024}}

@article{mcm:algorithms,
	author = {McMullen, Curt},
	doi = {10.2307/1971408},
	fjournal = {Annals of Mathematics. Second Series},
	issn = {0003-486X,1939-8980},
	journal = {Ann. of Math. (2)},
	mrclass = {58F08 (58C15 65H05)},
	mrnumber = {890160},
	mrreviewer = {Michael\ Hurley},
	number = {3},
	pages = {467--493},
	title = {Families of rational maps and iterative root-finding algorithms},
	url = {https://doi.org/10.2307/1971408},
	volume = {125},
	year = {1987}}

@article{ji-xie:injective,
	author = {Ji, Zhuchao and Xie, Junyi},
	journal = {J. Eur. Math. Soc. (JEMS)},
	note = {10.4171/JEMS/1750},
	title = {The multiplier spectrum morphism is generically injective},
	year = {2025}}

@article{conjugate,
	author = {Cantat, Serge and Dujardin, Romain},
	journal = {Bull. Lond. Math. Soc.},
	number = {12},
	pages = {3745-3751},
	title = {Holomorphically conjugate polynomial automorphisms of $\mathbb{C}^2$ are polynomially conjugate},
	volume = {56},
	year = {2024}}

@article{friedland-milnor,
	author = {Friedland, Shmuel and Milnor, John},
	doi = {10.1017/S014338570000482X},
	fjournal = {Ergodic Theory and Dynamical Systems},
	issn = {0143-3857},
	journal = {Ergodic Theory Dynam. Systems},
	mrclass = {58F99},
	mrnumber = {991490},
	mrreviewer = {M. Lyubich},
	number = {1},
	pages = {67--99},
	title = {Dynamical properties of plane polynomial automorphisms},
	url = {https://doi.org/10.1017/S014338570000482X},
	volume = {9},
	year = {1989}}

@article{bls2,
	author = {Bedford, Eric and Lyubich, Mikhail and Smillie, John},
	doi = {10.1007/BF01232671},
	fjournal = {Inventiones Mathematicae},
	issn = {0020-9910,1432-1297},
	journal = {Invent. Math.},
	mrclass = {32H50 (58F11 58F23)},
	mrnumber = {1240639},
	mrreviewer = {Gregery\ T.\ Buzzard},
	number = {2},
	pages = {277--288},
	title = {Distribution of periodic points of polynomial diffeomorphisms of {$\bold C^2$}},
	url = {https://doi.org/10.1007/BF01232671},
	volume = {114},
	year = {1993}}

@incollection{milnor:lattes,
	author = {Milnor, John},
	booktitle = {Dynamics on the {R}iemann sphere},
	doi = {10.4171/011-1/1},
	isbn = {3-03719-011-6},
	mrclass = {37F10 (30D05 37-02)},
	mrnumber = {2348953},
	mrreviewer = {Feliks\ Przytycki},
	pages = {9--43},
	publisher = {Eur. Math. Soc., Z\"{u}rich},
	title = {On {L}att\`es maps},
	url = {https://doi.org/10.4171/011-1/1},
	year = {2006}}

@article{ji-xie:multipliers,
	author = {Ji, Zhuchao and Xie, Junyi},
	doi = {10.1017/fmp.2023.12},
	fjournal = {Forum of Mathematics. Pi},
	issn = {2050-5086},
	journal = {Forum Math. Pi},
	mrclass = {37C29},
	mrnumber = {4585467},
	pages = {Paper No. e11, 37},
	title = {Homoclinic orbits, multiplier spectrum and rigidity theorems in complex dynamics},
	url = {https://doi.org/10.1017/fmp.2023.12},
	volume = {11},
	year = {2023}}

@article{berger-dujardin,
	author = {Berger, Pierre and Dujardin, Romain},
	doi = {10.24033/asens.2324},
	fjournal = {Annales Scientifiques de l'\'{E}cole Normale Sup\'{e}rieure. Quatri\`eme S\'{e}rie},
	issn = {0012-9593},
	journal = {Ann. Sci. \'{E}c. Norm. Sup\'{e}r. (4)},
	mrclass = {32H50 (37F44 37F80)},
	mrnumber = {3993324},
	mrreviewer = {Eric Bedford},
	number = {2},
	pages = {449--477},
	title = {On stability and hyperbolicity for polynomial automorphisms of {$\Bbb C^2$}},
	url = {https://doi.org/10.24033/asens.2324},
	volume = {50},
	year = {2017}}

@article{bs5,
	author = {Bedford, Eric and Smillie, John},
	doi = {10.1007/BF02921791},
	fjournal = {The Journal of Geometric Analysis},
	issn = {1050-6926},
	journal = {J. Geom. Anal.},
	mrclass = {32H50 (32C30 37A99 37F10)},
	mrnumber = {1707733},
	mrreviewer = {Dan Coman},
	number = {3},
	pages = {349--383},
	title = {Polynomial diffeomorphisms of {${\bf C}^2$}. {V}. {C}ritical points and {L}yapunov exponents},
	url = {https://doi.org/10.1007/BF02921791},
	volume = {8},
	year = {1998}}

@article{BK1,
	author = {Bedford, Eric and Kim, Kyounghee},
	doi = {10.1007/s12220-017-9796-1},
	fjournal = {Journal of Geometric Analysis},
	issn = {1050-6926,1559-002X},
	journal = {J. Geom. Anal.},
	mrclass = {37F10 (32H50)},
	mrnumber = {3708006},
	mrreviewer = {Romain\ Dujardin},
	number = {4},
	pages = {3085--3098},
	title = {No smooth {J}ulia sets for polynomial diffeomorphisms of {$\Bbb C^2$} with positive entropy},
	url = {https://doi.org/10.1007/s12220-017-9796-1},
	volume = {27},
	year = {2017}}

@incollection{connex,
	author = {Dujardin, Romain},
	booktitle = {Complex dynamics},
	doi = {10.1090/conm/396/07394},
	mrclass = {37F50 (37F10 37F15)},
	mrnumber = {2209087},
	mrreviewer = {Peter Ha\"{\i}ssinsky},
	pages = {63--84},
	publisher = {Amer. Math. Soc., Providence, RI},
	series = {Contemp. Math.},
	title = {Some remarks on the connectivity of {J}ulia sets for 2-dimensional diffeomorphisms},
	url = {https://doi.org/10.1090/conm/396/07394},
	volume = {396},
	year = {2006}}

@article{tangencies,
	author = {Dujardin, Romain and Lyubich, Mikhail},
	doi = {10.1007/s00222-014-0535-y},
	fjournal = {Inventiones Mathematicae},
	issn = {0020-9910},
	journal = {Invent. Math.},
	mrclass = {37F45},
	mrnumber = {3338008},
	mrreviewer = {Matteo Ruggiero},
	number = {2},
	pages = {439--511},
	title = {Stability and bifurcations for dissipative polynomial automorphisms of {$\Bbb{C}^2$}},
	url = {https://doi.org/10.1007/s00222-014-0535-y},
	volume = {200},
	year = {2015}}

@article{bs6,
	author = {Bedford, Eric and Smillie, John},
	doi = {10.2307/121006},
	fjournal = {Annals of Mathematics. Second Series},
	issn = {0003-486X},
	journal = {Ann. of Math. (2)},
	mrclass = {32H50 (37F10 37F50)},
	mrnumber = {1668567},
	mrreviewer = {Marco Abate},
	number = {2},
	pages = {695--735},
	title = {Polynomial diffeomorphisms of {${\bf C}^2$}. {VI}. {C}onnectivity of {$J$}},
	url = {https://doi.org/10.2307/121006},
	volume = {148},
	year = {1998}}

@article{bls,
	author = {Bedford, Eric and Lyubich, Mikhail and Smillie, John},
	doi = {10.1007/BF01232426},
	fjournal = {Inventiones Mathematicae},
	issn = {0020-9910},
	journal = {Invent. Math.},
	mrclass = {32H50 (32C30 58F11 58F23)},
	mrnumber = {1207478},
	mrreviewer = {Yasuichiro Nishimura},
	number = {1},
	pages = {77--125},
	title = {Polynomial diffeomorphisms of {${\bf C}^2$}. {IV}. {T}he measure of maximal entropy and laminar currents},
	url = {https://doi.org/10.1007/BF01232426},
	volume = {112},
	year = {1993}}
\end{document}